\documentclass{article}
\usepackage{PRIMEarxiv}
\usepackage[utf8]{inputenc}
\usepackage{mathtools, algorithm, algpseudocode}
\usepackage{amsthm}
\usepackage{enumitem}
\usepackage{bbm}
\usepackage{amsmath, amssymb}
\usepackage{amsfonts}
\usepackage{mathrsfs}
\usepackage{chngcntr}
\usepackage{bigints}
\usepackage{apptools}
\usepackage{subcaption}
\usepackage{caption}
\AtAppendix{\counterwithin{theorem}{section}}
\usepackage[english]{babel}
\usepackage[autostyle, english = american]{csquotes}
\MakeOuterQuote{"}
\usepackage[backend=biber, citestyle=numeric-comp, sorting=nyt]{biblatex}
\theoremstyle{definition}
\newtheorem{definition}{Definition}[section]
\newtheorem{theorem}{Theorem}
\newtheorem{assumption}{Assumption}

\newtheorem{corollary}{Corollary}[theorem]
\newtheorem{remark}{Remark}
\newtheorem{lemma}[theorem]{Lemma}
\usepackage{graphicx}
\usepackage{enumitem}
\usepackage{comment}
\usepackage{color}
\usepackage{float}
\usepackage{graphicx}
\usepackage{appendix}
\usepackage[utf8]{inputenc} 
\usepackage[T1]{fontenc}    
\usepackage{hyperref}       
\usepackage{url}            
\usepackage{booktabs}       
\usepackage{amsfonts}       
\usepackage{nicefrac}       
\usepackage{microtype}      
\usepackage{lipsum}
\usepackage{fancyhdr}       
\usepackage{graphicx}       
\graphicspath{{media/}}     

\DeclareUnicodeCharacter{00A0}{ }
\newenvironment{assumption*}
 {\expandafter\def\expandafter\theassumption\expandafter{\theassumption*}\assumption}
 {\endassumption}
 
\newcommand{\RNum}[1]{\uppercase\expandafter{\romannumeral #1\relax}}

\newtheorem*{theorem*}{Theorem}
\newtheorem*{corollary*}{Corollary}
\title{Optimal Learning Rates for Kernel Ridge Regression with a Fourier Capacity Condition
}

\author{
Prem Talwai \thanks{corresponding author}\\
Operations Research Center \\
MIT \\
\texttt{email}:talwai@mit.edu
\AND
David Simchi-Levi\\
Institute for Data, Systems, and Society \\
MIT  \\
\texttt{email}:dslevi@mit.edu
}

\begin{document}

\maketitle

\begin{abstract}
We derive minimax adaptive rates for a new, broad class of Tikhonov-regularized learning problems in Hilbert scales under general source conditions. Our analysis does not require the regression function to be contained in the hypothesis class, and most notably does not employ the conventional \textit{a priori} assumptions on kernel eigendecay. Using the theory of interpolation, we demonstrate that the spectrum of the Mercer operator can be inferred in the presence of ``tight'' $L^{\infty}(\mathcal{X})$ embeddings of suitable Hilbert scales. Our analysis utilizes a new Fourier isocapacitary condition, which captures the interplay of the kernel Dirichlet capacities and small ball probabilities via the optimal Hilbert scale function. 
\end{abstract}


\section{Introduction}
Consider the classical learning problem:
\begin{equation}
\label{Learning Problem}
    Y = f^{*}(X) + \epsilon
\end{equation}
where we wish to estimate the \textit{regression function} $f^{*}$ in the presence of additive noise $\epsilon$ using a dataset $\mathcal{D} = \{(X_1, Y_1), (X_2, Y_2), \ldots (X_n, Y_n)\}$ sampled i.i.d from a distribution $P$ over $\mathcal{X} \times \mathcal{Y}$. Here, we focus on kernel-based regularized least-squares (ridge) regression, i.e. we wish to solve the optimization problem

\begin{equation}
\label{eq: Ridge Regression}
    f_{D, \lambda} = \text{arg}\min_{f \in \mathcal{H}_K} \frac{1}{n}\sum_{i = 1}^n (Y_i - f(X_i))^2 + \lambda||f||^2_{\mathcal{K}}
\end{equation}

where $\mathcal{H}_K \subset L^2(P_{\mathcal{X}})$ is a reproducing kernel Hilbert space (RKHS) with kernel $K(\cdot, \cdot)$. 

\par The performance of a regularization scheme is typically quantified using the expected risk: 
\begin{equation}
\label{Risk Definition}
    R(f_{D, \lambda}) = ||f_{D, \lambda} - f^{*}||^2_{L^2(P)}
\end{equation}

\par This problem has been studied extensively in the statistics and machine learning literature \cite{fischer2020sobolev, caponnetto2007optimal, blanchard2018optimal, rastogi2020tikhonov}. These works primarily focus on establishing optimal convergence rates for $R(f_{D, \lambda})$ under light-tailed noise (typically subgaussian or subexponential) and various \textit{source conditions}, which characterize the regularity/smoothness of the learning target $f^{*}$. The two main approaches for this task have included the \textit{integral operator technique} (e.g. \cite{fischer2020sobolev, blanchard2018optimal, caponnetto2007optimal}) and \textit{empirical process technique} \cite{steinwart2008support, mendelson2010regularization}. While the latter approach enjoys easy adaptability to general convex losses, the former technique more directly exploits the Hilbertian structure of the hypothesis space, and aligns seamlessly with the method of real interpolation between $\mathcal{H}_K$ and $L^2(P_{\mathcal{X}})$ (see section \ref{Hilbert Prelims}). 

\par We consider the hard-learning scenario, where $f^{*} \not \in \mathcal{H}$. This setting has been treated in previous works \cite{fischer2020sobolev, lin2020optimal, blanchard2018optimal, rastogi2017optimal, rastogi2020tikhonov} both for the \textit{H{\"o}lder source condition} \cite{fischer2020sobolev, blanchard2018optimal} where $f^{*}$ lies in some interpolation space $[L^2(P_{\mathcal{X}}), \mathcal{H}_K]_{\theta, 2}$ (for $\theta \in (0, 1)$) and the \textit{general source condition} \cite{rastogi2017optimal} where $f^{*}$ lies in a suitable Hilbert scale (defined in section \ref{Hilbert Prelims}), which correspond to more general interpolation spaces between $L^2(P_{\mathcal{X}})$ and $\mathcal{H}_K$  obtained using a function parameter. While these works treat the problem of misspecification, their analysis hinges on a kernel eigendecay assumption, which characterizes the capacity of the hypothesis space $\mathcal{H}_K$ by the decay rate of the eigenvalues of the kernel Mercer operator. While this spectral assumption has been ubiquitous in the statistical analysis of kernel ridge regression \cite{lu2020balancing, blanchard2018optimal, caponnetto2007optimal}, it depends crucially on the data-generating measure $P$, which is typically unknown in practice. 

In this paper, we replace the classical kernel eigendecay condition with a new \textit{Fourier capacity condition} which characterizes the decay of the Fourier power spectrum of the kernel ($P$-independent) as opposed to its Mercer spectrum ($P$-dependent). Indeed, our capacity assumption relates to $P$ only through the eigenbasis of the Mercer operator, which may be shared by several different kernels and whose asymptotic behavior often depends only on the infimal volume decay of $P$ for small balls (see Lemmas \ref{Isoperimetric Equivalance}, \ref{Alternative Spectral Estimates}, and Corollary \ref{Alternative Spectral Estimates with Escape Rate} in \ref{Ridge Assumptions} and discussion therein). Fourier analytic approaches have been applied previously to study the capacity of the hypothesis class \cite{zhou2003capacity, ying2007learnability, smale2004shannon}. However, these works typically either assume a uniform (Lebesgue-equivalent) data-generating measure or particular RKHS structure (i.e a Sobolev space) in order to apply Fourier transform techniques. In this paper, we demonstrate that Fourier techniques remain powerful over general probability spaces, in the presence of the embedding of a suitable Hilbert scale in $L^{\infty}(\mathcal{X})$. Indeed, we demonstrate that if the embedding of an intermediate Hilbert scale is sufficiently ``sharp'' (characterized by the suitability of its index function as an isocapacitary profile; see Lemma \ref{Isoperimetric Equivalance}), then a Gagliardo-Nirenberg type inequality between the scales enables us to estimate kernel eigendecay via the Gelfand widths of the RKHS in $L^{\infty}(\mathcal{X})$. To the best of our knowledge, this is the first work to establish minimax learning rates under general source conditions without appealing to any direct assumptions on the eigendecay or statistical dimension of the Mercer operator. We demonstrate that the optimal rates in the mean-squared error can be completely characterized in terms of three geometric quantities: the ``spectral distances'' of the hypothesis class $\mathcal{H}_K$ to $L^{\infty}(\mathcal{X})$ and the learning target $f^{*}$, and the ``effective dimension'' \cite{villani2009optimal} of the metric measure space $(\mathcal{X}, P|_{\mathcal{X}}, |||\cdot|||_{d})$, where the latter coincides with measure dimension considered by \cite{lott2009ricci, villani2009optimal, ambrosio2018calculus} in $\text{CD}(0, N)$ spaces.   \\

\section{Preliminaries} \label{Regression Example}

We recall our learning problem:

\begin{equation*}
    f_{\lambda, D} = \text{arg}\min_{f \in \mathcal{H}_K} \frac{1}{n}\sum_{i = 1}^n (Y_i - f(X_i))^2 + \lambda||f||^2_{\mathcal{K}}
\end{equation*}

where $\mathcal{H}_K$ be a separable reproducing kernel Hilbert space (RKHS) on $\mathcal{X}$ (see e.g. \cite{steinwart2008support} for a definition). A solution to this problem can be computed in closed form:
\begin{equation*}
    f_{\lambda, D} = (C_D + \lambda)^{-1}g_D
\end{equation*}
where $C_D = \mathbb{E}_{\mathcal{D}}[k(X, \cdot) \otimes k(X, \cdot)]$ is the empirical covariance, and $g_D = \sum_{i = 1}^n y_i k(x_i, \cdot)$. Before discussing our model framework and results, we begin with an overview of some mathematical preliminaries required for this section. 

\subsection{Hilbert Scales} \label{Hilbert Prelims}
Suppose the imbedding $I_{\nu}: \mathcal{H}_K \to L^2(\nu)$ of $\mathcal{H}_{K}$ into $\mathcal{L}^2(\nu)$ is injective (here $\nu = P_{X}$ is the marginal on $\mathcal{X}$). Let $S_{\nu} = I^{*}_{\nu}$ be its adjoint. Then, it can be shown that $S_{\nu}$ is an integral operator given by:
\begin{equation}
\label{eq: Adjoint Formula}
    S_{\nu}f(x) = \int_{\mathcal{X}} K(x, \cdot)f(y)d\nu(y)
\end{equation}
Using $S_{\nu}$ and $I_{\nu}$, we construct the following positive self-adjoint operators on $\mathcal{H}_K$ and $\mathcal{L}^2(\nu)$, respectively:
\begin{align*}
    C_{\nu} & = S_{\nu}I_{\nu} =  I^{*}_{\nu}I_{\nu} \\
    T_{\nu} & = I_{\nu}S_{\nu} =  I_{\nu} I^{*}_{\nu}
\end{align*}
We observe that $C_{\nu}$ and $T_{\nu}$ are nuclear (see Lemma 2.2/2.3 in \cite{steinwart2012mercer}) 
Since, $T_{\nu}$ is nuclear and self-adjoint, it admits a spectral representation:
\begin{equation*}
    T_{\nu} = \sum_{j = 1}^{\infty} \mu_j e_j \langle e_j, \cdot \rangle_{L^2(\nu)}
\end{equation*}
where $\{\mu_j\}_{j = 1}^{\infty} \in (0, \infty)$ are nonzero eigenvalues of $T_{\nu}$ (ordered nonincreasingly) and $\{e_j\}_{j = 1}^{\infty} \subset L^2(\nu)$ form an orthonormal system of corresponding eigenfunctions. Note that formally, the elements $e_j$ of $L^2(\nu)$ are equivalence classes $[e_j]_{\nu}$ whose members only differ on a set of $\nu$-measure zero--- notationally, we consider this formalism to be understood here and simply write $e_j$ to refer to elements in both $\mathcal{H}_K$,  $L^2(\nu)$, and their interpolation spaces (with the residence of $e_j$ understood from context). Given a nonincreasing $\textit{index}$ function $\phi: (0, ||T_{\nu}||) \to \mathbb{R}_{+}$, we define the \textit{Hilbert} scales $\mathcal{H}^{\phi}_K$ as \cite{mathe2003geometry}:

\begin{definition}
Let $\phi: (0, ||T_{\nu}||) \to \mathbb{R}_{+}$ be nondecreasing, continuous, with $\phi(0) = 0$. Then, the Hilbert scale $\mathcal{H}^{\phi}_{K}$ is the completion of the space:
\begin{equation*}
    \mathcal{H}^{\phi}_{K} = \Big\{f \in L^2(\nu): \sum_{i = 1}^{\infty} \frac{\langle f, e_i \rangle_{L^2(\nu)}^2}{\phi(\mu_i)} < \infty \Big\}
\end{equation*}
with respect to the inner product $\langle f, g \rangle_{\phi} = \sum_{i} \frac{\langle f, e_i \rangle \langle g, e_i \rangle}{\phi(\mu_i)}$. 
\end{definition}
It is easy to see that $\mathcal{H}^{\phi}_K \cong \text{ran}(\phi^{\frac{1}{2}}(T_{\nu}))$. The subscript $K$ in $\mathcal{H}^{\phi}_K$ reflects that $\phi$ is acting on the spectrum of $K$, when the kernel is fixed and understood from context, we will omit this subscript and simply denote the Hilbert scale $\mathcal{H}^{\phi}$. 

We define:
\begin{equation*}
    \|k^{\phi}\|^2_{\infty} = \sup_{x \in \mathcal{X}} \sum_{i=1}^{\infty} \phi(\mu_i)e_i^2(x)
\end{equation*}
where we allow $\|k^{\phi}\|_{\infty} = \infty$. If $\|k^{\phi}\|_{\infty} < \infty$, then it is easy to show that $\mathcal{H}^{\phi}_K$ is continuously embedded in $L^{\infty}(\mathcal{X})$ with norm $\|k^{\phi}\|_{\infty}$ (see e.g. Theorem 9 in \cite{fischer2020sobolev} for the case of Holder source conditions). \\ 

The Hilbert scale $\mathcal{H}^{\phi}_{K}$ generalizes the notion of RKHS interpolation (power) spaces discussed in \cite{steinwart2012mercer}, which result from the specific choice of $\phi(t) = t^{\alpha}$ for some $\alpha \in (0, 1)$. Intuitively, the Hilbert scales can be viewed as a nonlinear transformation of the infinite-dimensional RKHS ellipsoid $B(\mathcal{H}_K)$, realized by transforming the axes lengths with the index function $\phi$. As we will see later, the growth properties of $\phi$ play a fundamental role in specifying the smoothness of the elements of $\mathcal{H}^{\phi}_K$. \\ 

We now define \textit{weak embeddings} of Hilbert scales, a notion which will feature prominently in the remainder of the paper:

\begin{definition}
\label{Weak Embedding Def}
For a given Banach space $\mathcal{B}$, We say that $\mathcal{H}^{\phi}_{K}$ is weakly embedded in $\mathcal{B}$ (denoted as $\mathcal{H}^{\phi}_{K} \stackrel{w}{\hookrightarrow} \mathcal{B}$) if $\mathcal{H}^{\phi^{1 + \epsilon}}_{K} \hookrightarrow \mathcal{B}$ compactly for all $\epsilon >  0$.
\end{definition}

Note that in the above definition it is entirely possible that $\mathcal{H}^{\phi}_{K} \stackrel{w}{\hookrightarrow} \mathcal{B}$ and $\mathcal{H}^{\phi}_K \not \hookrightarrow \mathcal{B}$. In fact, in most scenarios of interest, this is indeed the case, where $\mathcal{H}^{\phi}_{K}$ is merely the \textit{limit} of a sequence of Hilbert scales $\mathcal{H}^{\phi^{1 + \epsilon}}_{K}$ contained compactly in $\mathcal{B}$. A classical (and guiding) example of the latter phenomenon involves the canonical Sobolev spaces over a sufficiently regular domain $\mathcal{X} \subset \mathbb{R}^d$, where, $s > \frac{d}{2}$, $\mathcal{H}_K = H^s(\mathcal{X})$, $\psi(t) = t^{\frac{d}{2s}}$, and $\mathcal{B} = L^{\infty}(\mathcal{X})$ in Definition \ref{Weak Embedding Def}. Then $H^{\frac{d}{2}}(\mathcal{X}) \stackrel{w}{\hookrightarrow} L^{\infty}(\mathcal{X})$ by the Sobolev embedding theorem (see e.g. Corollary 5.6 in \cite{steinwart2019convergence} or Theorem 4.12 in \cite{adams2003sobolev} for a more detailed discussion).

\subsection{Approximation Widths}
We review some standard widths in approximation theory that describe the compactness of spaces in various norms. The $n^{\text{th}}$ entropy number $\epsilon_n(\mathcal{X})$ of a compact subset $\mathcal{X}$ of a metric space, is informally the smallest radius $\epsilon$ such that there exists an $\epsilon$-covering of $\mathcal{X}$ of at most $n$ metric balls. Precisely, we have:
\begin{equation}
\label{eq: Entropy Def}
\epsilon_n(\mathcal{X}) = \inf\{\epsilon > 0: \exists \hspace{0.5mm} \{x_i\}_{i = 1}^n \subset \mathcal{X}; \text{s.t.} \hspace{2mm} \mathcal{X} \subset \cup_{i = 1}^n B(x_i, \epsilon)\}
\end{equation}
Let $X, Y$ be two Banach spaces with $X \subset Y$. We further define the $n$-Gelfand width $X$ embedded in $Y$ as:
\begin{equation}
\label{eq: Gelfand Def}
c_n(X, Y) = \inf_{Z \subset X, \text{codim}(Z) \leq n} \hspace{2mm} \sup_{z \in Z} \frac{||z||_{Y}}{||z||_{X}}
\end{equation}
where the infimum in \eqref{eq: Gelfand Def} is taken over all subspaces $Z \subset X$ of codimension at most $n$. The Gelfand widths play a pivotal role in the study of compressed sensing, where they have been demonstrated to characterize the worst-case errors of optimal measurement/reconstruction schemes \cite{foucart2013invitation}. We further define the $n$-Kolmogorov widths:
\begin{equation}
\label{eq: Kolmogorov Def}
d_n(X, Y) = \inf_{A: X \to Y; \text{rank}(A) \leq n} \hspace{2mm} \sup_{z \in \mathcal{X}} \frac{||z - Az||_{Y}}{||z||_{X}}
\end{equation}
where the infimum is taken over all (possibly nonlinear) operators $A$ with rank at most $n$. Intuitively, the Kolmogorov widths describe the distance of $X$ to the closest $n$-dimensional subspace in $Y$. The Kolmogorov widths are dominated by the Gelfand widths, and hence, despite their more natural definition, are often suboptimal for our purposes. Finally, we define the $n$-Bernstein widths. 
\begin{equation}
\label{eq: Bernstein Def}
b_n(X, Y) = \sup_{Z \subset X, \text{dim}(Z) \geq n + 1} \hspace{2mm} \inf_{z \in Z} \frac{||z||_{Y}}{||z||_{X}}
\end{equation}
Intuitively, the Bernstein width gives the radius of the largest $||\cdot||_{Y}$-ball that may be inscribed in a unit $||\cdot||_{X}$-ball in subspaces of dimension greater than $n$. These are the smallest of the three widths, and are typically employed to obtain lower estimates for the singular numbers. When $X$ and $Y$ are both Hilbert spaces, then $b_n$ coincides with the $(n+1)^{\text{st}}$ singular number of the embedding $X \hookrightarrow Y$. Moreover, if $X \hookrightarrow Y$ is compact, then $d^n(X, Y), d_n(X, Y), b_n(X, Y) \to 0$. Intuitively, the rate of the latter convergence quantifies the  ``degree'' of compactness of $X \hookrightarrow Y$ (a significant effort in the paper is devoted to obtaining tight estimates for this rate when $X$ is our hypothesis class). For a detailed discussion of the relationships between various approximation widths and function space embeddings, we direct the reader to \cite{carl_stephani_1990} and \cite{pinkus2012n}. \\ 

\subsection{Potential Theory}
\label{Potential Preliminaries}

We will employ some basic constructions of potential theory, primarily to provide a geometric perspective into the analytic conditions outlined in section \ref{Ridge Assumptions}. In the interest of space, we mostly focus here on offering intuition behind these constructs, and forego a fully self-contained development. An excellent, comprehensive introduction to potential theory and its relationship with Markov processes can be found in \cite{fukushima2010dirichlet}. Throughout this section, we make the standard assumption that when considering a generic measure space $(\mathcal{Z}, \mu)$, $\mathcal{Z}$ is a locally compact separable metric space and $\mu$ has full support $\bar{\mathcal{Z}}$. Note, that we will sometimes refer to the couple $(\mathcal{Z}, \mu)$ as a metric measure space, where the metric is understood to be Euclidean $|||\cdot|||_{d}$ in the ambient topological dimension $d \geq 1$.




\begin{definition}
\label{Dirichlet Form definition}
Let $\mathcal{H} \subset L^2(\mathcal{Z}, \mu)$ be a Hilbert space. A closed symmetric form $\mathcal{E}$ with domain $\mathcal{H}$ is called a \textit{Dirichlet form} on $L^2(\mathcal{Z}, \mu)$ if there exists a nonnegative-definite, self-adjoint operator $A$ on $L^2(\mathcal{Z}, \mu)$ such that:
\begin{equation*}
    \mathcal{E}(u, v)  = \int_{\mathcal{Z}} (Au)vd\mu
\end{equation*}
and the associated semigroup $T_t = e^{-At}$ is Markovian. Moreover, $\mathcal{H}$ is complete with respect to the inner product:
\begin{equation*}
    \mathcal{E}_1(u, v) = \mathcal{E}(u, v) + (u, v)_{L^2(\mu)}
\end{equation*}
\end{definition}

We note that the identification of the symmetric form $\mathcal{E}$ with a  nonnegative-definite, self-adjoint operator is automatic from the $\mathcal{E}_1$-closedness of $\mathcal{H}$ in $L^2(\mathcal{Z}, \mu)$ (Theorem 1.3.1 in \cite{fukushima2010dirichlet}); hence the only restriction in Definition \ref{Dirichlet Form definition} is the Markovian nature of the associated semigroup. In this paper, typically $\mathcal{H} = \mathcal{H}_K(\mathcal{X})$,  $\mathcal{Z} = \mathcal{X}$ and $\mu = dx$ is the Lebesgue measure (see Remark \ref{Trace Clarification}). We identify $\mathcal{E}_1$ with $\langle \cdot, \cdot \rangle_{K}$ up to norm-equivalence. Observe that if instead, $\mathcal{H} = \mathcal{H}_K(\mathcal{X})$, $\mathcal{Z} = \mathcal{X}$ and $\mu = \nu$ then $A$ is simply the inverse of the Mercer operator $T_{\nu}$. The Markovian property of a Dirichlet form could equivalently be defined via closure of $\mathcal{H}$ under normal contractions (Theorem 1.4.1 in \cite{fukushima2010dirichlet}), however we eschew this formulation in order to highlight the association of the inner product with the Hunt process $\{X_t\}_{t \geq 0}$ generated by $T_t$. When $K(x, y) = \kappa(x - y)$ is radial, $T_t$ is a convolution semigroup. In section \ref{Non-Radial}, we examine how the probabilistic properties of $X_t$ can characterize the geometry of $\mathcal{H}_K$. 

\begin{remark}
\label{Trace Clarification}
We note that when $\mathcal{X} \subset \mathbb{R}^d$ is a proper compact subdomain, $\mathcal{H}_K(\mathcal{X})$ is slightly different from our RKHS $\mathcal{H}_K$. Recall that the domains of elements of $\mathcal{H}_K$ are not restricted to $\mathcal{X}$, however by the assumed injectivity of $I_{\nu}: \mathcal{H}_K \to L^2(\nu)$, we have that for each $f \in I_{\nu}(\mathcal{H}_K)$, there exists a unique $\tilde{f} \in \mathcal{H}_K$ such that $f = \tilde{f}|_{\mathcal{X}}$ almost surely on $\mathcal{X}$ (with respect to either $dx$ or $\nu$, due to the absolute continuity of the latter). On the other hand, the domains of elements of $\mathcal{H}_K(\mathcal{X})$ \textit{are restricted} to $\mathcal{X}$ by definition; hence $\mathcal{H}_K(\mathcal{X}) = \{f  = \tilde{f}|_{\mathcal{X}}; \tilde{f} \in \mathcal{H}_K\}$ and $\langle f, g \rangle_{\mathcal{H}_K(\mathcal{X})} = \langle \tilde{f}, \tilde{g} \rangle_{K}$ (where $\tilde{f}, \tilde{g} \in \mathcal{H}_K$ and $f = \tilde{f}|_{\mathcal{X}}$ and $g = \tilde{g}|_{\mathcal{X}}$). Moreover, $\mathcal{H}_K$ is itself typically obtained from $\mathcal{H}_K(\mathbb{R}^d)$ via projection. Namely, $\mathcal{H}_K$ is the orthogonal complement in $\mathcal{H}_K(\mathbb{R}^d)$ of the subspace of $\mathcal{H}_K(\mathbb{R}^d)$ containing functions supported on $\mathbb{R}^d \setminus \bar{\mathcal{X}}$ (recall we always assume $\bar{\mathcal{X}}$ is the support of $\nu$). Here, the Hunt process associated with $\mathcal{H}_{K}(\mathcal{X})$ is simply that associated with $\mathcal{H}_{K}(\mathbb{R}^d)$ with killing at the boundary $\partial \mathcal{X}$. Hence, in this case we abuse notation and simply express the Dirichlet form as $(\langle \cdot, \cdot \rangle_{K}, \mathcal{H}_K(\mathcal{X}))$ instead of the more cumbersome $(\langle \cdot, \cdot \rangle_{\mathcal{H}_K(\mathcal{X})}, \mathcal{H}_K(\mathcal{X}))$ with the aforementioned convention understood. See Appendix \ref{Fourier Capacity and Range Space} and section 6.2 of \cite{fukushima2010dirichlet} for more details on the trace of Dirichlet forms. 
\end{remark}

\begin{definition}
\label{Regular Dirichlet Form}
A Dirichlet form $(\langle \cdot, \cdot \rangle_{\mathcal{H}}, \mathcal{H})$ on $L^2(\mathcal{Z}, \mu)$ is called \textit{regular} if $C_0(\mathcal{Z}) \cap \mathcal{H}$ is dense in both $C_0(\mathcal{Z})$ and $\mathcal{H}$ in the respective norms. Further, $\langle \cdot, \cdot \rangle_{\mathcal{H}}$ is \textit{local} if $\langle u, v \rangle_{\mathcal{H}} = 0$ for $u, v \in \mathcal{H}$ with disjoint compact supports. 
\end{definition}

We note that when $\mathcal{Y}$ is itself compact, the regularity condition above is not particularly restrictive, as $C_{0}(\mathcal{Y})$ simply coincides with $C(\mathcal{Y})$, the space of continuous functions on $\mathcal{Y}$. Hence, when $\mathcal{H}_K$ contains continuous functions, regularity of the kernel inner product $\langle \cdot, \cdot \rangle_{K}$ is equivalent to the universality of the kernel $K(\cdot, \cdot)$, a classical assumption in the application of kernel methods \cite{micchelli2006universal}. See Chapter 1 of \cite{fukushima2010dirichlet} for further details on Dirichlet forms, and \cite{micchelli2006universal, steinwart2008support}for more on universal kernels. 


\begin{definition}
\label{Capacity Definition}
Suppose $(\langle \cdot, \cdot \rangle_{\mathcal{H}}, \mathcal{H})$ is a local, regular Dirichlet form on $L^2(\mathcal{X})$. Let $E \subset F \subset \mathcal{X}$ be two precompact open sets. The relative $\mathcal{H}$-capacity of $E$ in $F$ is:
\begin{equation*}
    \text{cap}_{\mathcal{X}}(E, F; \mathcal{H}) = \inf\{||f||^2_{\mathcal{H}}: f \in \mathcal{H}; f \geq 1 \hspace{1mm} \text{on} \hspace{1mm} E; f \equiv 0 \hspace{1mm} \text{on} \hspace{1mm} \mathcal{X} \setminus F\}
\end{equation*}
with $\text{cap}_{\mathcal{X}}(E, F; \mathcal{H}) = \infty$ if no such admissible $f \in \mathcal{H}$ exists. 
\end{definition}
It is easy to verify that for a fixed open subset $F \subset \mathcal{X}$, $\text{cap}(E, F; \mathcal{H})$ is a \textit{Choquet} capacity in $E$; namely it is an increasing set function that is continuous with respect to countable unions and closures. Moreover, if $E$ is compact, $F$ is a relatively compact open set, and $\langle \cdot, \cdot \rangle$ is regular, we can ensure the existence of an admissible $f$ and hence the finiteness of the capacity. Note that here, following convention, we simply require $f \geq 1$ on $E$ in Definition \ref{Capacity Definition}, however, for practical purposes, we can restrict our consideration to $f \equiv 1$ on $E$ and $0 \leq f \leq 1$ on $\mathcal{X}$ by the assumed Markovian property of the kernel inner product. Moreover, note that when $\mathcal{H} = \mathcal{H}_K(\mathcal{X})$, the norm $||\cdot||_{\mathcal{H}_K(\mathcal{X})}$ in Definition \ref{Capacity Definition} can be replaced by $||\cdot||_{K}$ since the admissible cutoff functions $\phi$ have compact support in $\mathcal{X}$, vanish at the boundary, and hence can be identified with their zero-extension $\tilde{\phi}$ to $\mathbb{R}^d$ (note that if $\phi \in C_{c}(\text{int}(\mathcal{X}))$, then $\langle \tilde{\phi}, f \rangle_{K} = 0$ for $f \in \mathcal{H}_K(\mathbb{R}^d)$ with $f = 0$ a.s. on $\mathcal{X}$ by locality, hence $\phi \in \mathcal{H}_K$ by Remark \ref{Trace Clarification}). 

Geometrically, $\text{cap}(E, F; \mathcal{H})$ characterizes a generalized perimeter of $E$ relative to $F$. A possibly more intuitive interpretation of $\text{cap}(E, F; \mathcal{H})$ stems from physics, where the relative capacity can be viewed as the generalized conductance across the field $F \setminus E$ given electrodes placed at $\partial E$ and $\partial F$ \cite{polya1951isoperimetric}. Analogously, we consider the reciprocal $\frac{1}{\text{cap}(E, F; \mathcal{H})}$ as the $\textit{resistance}$ across the field $F \setminus E$. In the subsequent section, we will be particularly concerned with estimating $\text{cap}(B(x, Kr), B(x, r); \mathcal{H}_K(\mathcal{X}))$ for $x \in \mathcal{X}$ and $K \in (0,1)$ as $r \to 0$, i.e the conductance across a contracting annulus. In the classical case (where $\mathcal{H}_K(\mathcal{X})$ is replaced by $W^{1, p}(\Omega)$ for $p \geq 1$ and $\Omega \subset \mathbb{R}^d$ with smooth boundary), the asymptotics of $\text{cap}(B(x, Kr), B(x, r); W^{1, p}(\Omega))$ play a crucial role in the behavior of generalized Green's functions near their singularity \cite{holopainen2002singular}. It is important to note that the $\mathcal{H}_K(\mathcal{X})$-capacity is completely \textbf{independent of the measure} $\nu$, and purely a characteristic of the RKHS $\mathcal{H}_K$ and the domain $\mathcal{X}$. In Lemma \ref{Isoperimetric Equivalance}, we demonstrate that it is precisely the \textit{interplay} of the $\mathcal{H}_K(\mathcal{X})$-resistance and $\nu$-volume growth of balls that determines optimal Sobolev-type embeddings and minimax learning rates. 

\begin{remark}[Notation]
\label{Notation Remark}
For any two Banach spaces $\mathcal{A}$ and $\mathcal{B}$, we write $\mathcal{A} \hookrightarrow \mathcal{B}$, if $\mathcal{A}$ is continuously embedded in $\mathcal{B}$, and $\mathcal{A} \stackrel{c}\hookrightarrow \mathcal{B}$ if this embedding is further compact. We denote by $B(\mathcal{A})$ the unit ball in $\mathcal{A}$ and $\mathbb{D}$ denotes the complex open unit disk. For $x \in \mathcal{X}$, $B(x, r)$ denotes the open unit ball of radius $r > 0$ in the the metric space $(\mathcal{X}, d)$. $\mathbb{E}_{\mathcal{D}}[\cdot]$ denotes the sample expectation, and $C_{D} = \mathbb{E}_{\mathcal{D}}[k(X, \cdot) \otimes k(X, \cdot)]$ denotes the sample covariance. We will denote by $f_{D, \lambda}$ the solution to \eqref{eq: Ridge Regression} and $f_{\lambda} = (C_{\nu} + \lambda)^{-1}S_{\nu}f^{*}$, the regularized population solution. $\mathcal{F}$ denotes the Fourier transform on $\mathbb{R}^d$, while $\mathcal{F}_d$ denotes the $d$-dimensional Fourier transform of a univariate positive definite function (see e.g. Theorem 5.26 in \cite{wendland2004scattered}). $L^p(\mathcal{Z}, \mu)$ denotes the $L^p$ space on the measure space $(\mathcal{Z}, \mu)$ --- the domain $\mathcal{Z}$ will be omitted in the notation if it is understood from context (i.e. when $\mathcal{Z} = \mathcal{X}$); likewise the measure $\mu$ will be omitted if it is equivalent to Lebesgue measure $dx$. $||\cdot||_{2}$ or $||\cdot||_{L^2(\nu)}$ are used for norms in $L^2(\nu)$ (the latter if we want to make the space explicit); $||\cdot||_{\phi}$ for the norm in Hilbert scale $\mathcal{H}^\phi$, and $||\cdot||_{K}$ for the RKHS $\mathcal{H}_K$. $|||\cdot|||_{d}$ will be used to denote the Euclidean norm in $\mathbb{R}^d$. Typically, we will take $\nu = P_X = P|_{\mathcal{X}}$ where $P$ is the data-generating measure, and $|_{\mathcal{Z}}$ denotes the ``marginal on'' the space $\mathcal{Z}$. When $\mathcal{Z} \subset \mathbb{R}^d$, unless stated otherwise, \textit{we will always assume $\nu$ is Lebesgue absolutely continuous} and Radon. Finally, we will write $a(x) \preceq b(x)$ ($a(x) \succeq b(x)$) if $a(x) \leq Kb(x)$ (resp. $b(x) \leq Ka(x)$) asymptotically (typically as $x \to 0$ or $x \to \infty$), for some $K > 0$ that depends only on problem parameters (and independent of $x$ or the confidence level in the statement). If $a \preceq b$ and $b \preceq a$, then we write $a \asymp b$. Finally for a montonic, positive, univariate function $f(t)$, $f^{-1}(t)$ denotes its (generalized) inverse, while $f(t)^{-1} = \frac{1}{f(t)}$ denotes its reciprocal. 
\end{remark}

We say that a function $\xi$ satisfies $\Delta_2$ condition if there exists constants $D_1, D_2 > 0$, such that:
\begin{equation}
\label{eq: Delta_2}
D_1 \xi(\lambda) \leq \xi(2\lambda) \leq D_2\xi(\lambda) \hspace{2mm} \forall \lambda > 0
\end{equation}

Further, for any positive monotonic function $\phi:(0, \infty) \to (0, \infty)$, we define the dilation function $d_{\phi}: (0, \infty) \to (0, \infty)$ as:
\begin{equation}
    d_{\phi}(t) = \sup_{s \in (1, \infty)} \frac{\phi(st)}{\phi(s)}
\end{equation}
We define the extension indices:
\begin{align}
    \alpha_{\phi} & \equiv \lim_{t \to 0} \frac{\log d_{\phi}(t)}{\log t}  \label{eq: Lower Extension Index}\\
    \beta_{\phi} & \equiv \lim_{t \to \infty} \frac{\log d_{\phi}(t)}{\log t} \label{eq: Upper Extension Index}
\end{align}

We say a measure $\mu$ on $\mathcal{X}$ is doubling if for each $x \in \mathcal{X}$, there exists a $r_x > 0$, such that  $\mu(B(x, \cdot))$ is $\Delta_2$ on $(0, r_x)$ with constants $D_{1, x}$ and $D_{2, x}$. 
\subsection{Assumptions} \label{Ridge Assumptions}
Let $\phi$ and $\psi$ be two admissible index functions (i.e. $\phi$ and $\psi$ are nondecreasing, nonnegative, and continuous, with $\phi(0) = \psi(0) = 0$)
\begin{assumption}
\label{Embedding Condition}
$H_{K} \stackrel{c}\hookrightarrow H^{\phi} \stackrel{c}\hookrightarrow H^{\psi} \stackrel{w}{\hookrightarrow} L^{\infty}(\mathcal{X})$
\end{assumption}
\begin{assumption}
\label{Growth Conditions}
\textbf{(a)} $\frac{t}{\phi(t)}$ and $\frac{\phi(t)}{\psi(t)}$ are nondecreasing, and $\frac{t}{\psi(t)}$ is concave; \textbf{(b)} $\phi, \psi$ satisfy the $\Delta_2$ condition \eqref{eq: Delta_2}  for some $D^{\phi}_1, D^{\phi}_2, D^{\psi}_1, D^{\psi}_2 > 1$
\end{assumption}
\begin{assumption}
\label{Source Condition}
$f^{*} \in \mathcal{H}^{\phi}$
\end{assumption}

\begin{assumption}
\label{Eigenfunction Growth}
$\sup\Big\{\Big\|\sum_{i = 1}^{n} c_i e_i(\cdot)\Big\|_{\infty}: \sum_{i = 1}^n c^2_i = 1 \Big\} \asymp \sqrt{s(n)}$ for some increasing functions $s$ that satisfies \eqref{eq: Delta_2} for some $D^{s}_1, D^{s}_2 > 1$.

\end{assumption}

\begin{assumption}
\label{Subexponential Noise}
There are constants $\sigma, L > 0$ such that:
\begin{equation*}
    \int_{\mathbb{R}} |y - f^{*}(x)|^m dP(y|x) \leq \frac{m!\sigma^2 L^{m-2}}{2}
\end{equation*}
for $\nu$-almost all $x \in \mathcal{X}$
\end{assumption}

\begin{assumption*}
\label{Domain Condition}
$\mathcal{X} \subset \mathbb{R}^d$ is connected, relatively compact where $\epsilon_{n}(\mathcal{\bar{X}}) \asymp n^{-\frac{1}{d}}$ as $n \to \infty$. 
\end{assumption*}

\begin{assumption*}
\label{Radial Kernel}
$K$ is radial, i.e. $K(x, y) = \kappa(||x - y||)$ for some strictly positive definite function $\kappa \in L^1(\mathbb{R})$, such that $\mathcal{F}_d \kappa(t)$ is nonincreasing, continuous and $\alpha_{\mathcal{F}_d\kappa}, \beta_{\mathcal{F}_d\kappa} \in (-\infty, \infty)$. 
\end{assumption*}

\begin{assumption*}
\label{Opt Smoothness}
$\psi\Big(\frac{t\mathcal{F}_d \kappa(t^{\frac{1}{d}})}{s(t)}\Big) \asymp \frac{1}{s(t)}$ as $t \to \infty$
\end{assumption*}

Assumptions \ref{Embedding Condition} and \ref{Growth Conditions} dictates the relative growth rates of the index functions $\phi$ and $\psi$. Indeed, the compactness of the embedding $H^{\phi} \stackrel{c}\hookrightarrow H^{\psi}$ ensures that $\frac{\phi(t)}{\psi(t)} \to 0$ as $t \to 0$. Moreover, the first embedding $H_{K} \stackrel{c}\hookrightarrow H^{\phi}$ ensures that both $\frac{t}{\phi(t)} \to 0$ and $\frac{t}{\psi(t)} \to 0$ as $t \to 0$. Intuitively, this condition implies that the norms in $\mathcal{H}_{\psi}$ and $\mathcal{H}_{\phi}$ are weaker than those in our hypothesis class $\mathcal{H}_K$, allowing us to study convergence in these larger spaces of sample estimators constructed in $\mathcal{H}_K$. The third embedding, $H^{\psi} \stackrel{w}{\hookrightarrow} L^{\infty}(\mathcal{X})$ occurs only in the weak sense (recall Definition  \ref{Weak Embedding Def} and examples therein), i.e. $H^{\psi}$ is the limit of Hilbert scales embedded compactly in $L^{\infty}(\mathcal{X})$. The second part of the assumption requiring $\frac{t}{\phi(t)}$ and $\frac{\phi(t)}{\psi(t)}$ to be nondecreasing extends this behavior as $t \to \infty$; collectively the two growth conditions characterize $\phi^{-1}$ and $\psi^{-1}$ as resembling Young's functions; a requirement commonly imposed in the study of Orlicz spaces \cite{arriagada2021matuszewska}. The requirement that $\frac{t}{\phi(t)}$ is additionally concave enables the application of a Gagliardo-Nirenberg type interpolation inequality \cite{mathe2008direct} that relates $L^{\infty}(\mathcal{X})$ norms to those in $\mathcal{H}_K$ and $L^2(\nu)$, which will be crucial in our analysis of uniform error rates. Finally, the condition that index functions are $\Delta_2$ is found quite commonly in the literature on statistical inverse problems \cite{mathe2003geometry, schuster2012regularization}, where it has been used to demonstrate adaptivity of balancing strategies to unknown source conditions.  It should be noted that all the growth conditions in Assumption \ref{Growth Conditions} are satisfied when $\phi(t) = t^{\beta}$ and $\psi(t) = t^{\alpha}$ are power functions with $0 < \alpha < \beta < 1$ (the so-called H{\"o}lder source conditions). \\

Assumption \ref{Source Condition} characterizes the smoothness of the learning target $f^{*}$. Note a distinctive feature of our analysis is that we allow $f^{*}$ to lie outside the hypothesis class $\mathcal{H}_K$ (in light of Assumption \ref{Embedding Condition}), while still in $L^{\infty}(\mathcal{X})$ (see Assumption \ref{Embedding Condition}). Assumption \ref{Subexponential Noise} is standard subexponential noise condition in kernel regression \cite{fischer2020sobolev, talwai2022sobolev}.  \\\


Assumption \ref{Eigenfunction Growth} characterizes the embedding of $L^{\infty}(\mathcal{X})$ in $L^2(\nu)$ via the eigenspaces of $T_{\nu}$. The ``Lebesgue'' function $M_{n}(x) = \sum_{i = 1}^{n} e_i(x)$ was first considered by Littlewood \cite{hardy1948new} for the Fourier basis; later its $L^{p}$-norms were studied in several works \cite{konyagin1981problem, bochkarev2006multiplicative, borwein2001expected}, primarily in the context of finding roots for various generalized polynomials. \cite{temlyakov2011greedy} and \cite{kashin2005orthogonal} later explored the intimate connection between such Lebesgue-type inequalities and the greedy approximation properties of the basis $\{e_i\}_{i = 1}^{\infty}$. Indeed, it can be seen directly from the definition \eqref{eq: Bernstein Def} that assumption \ref{Eigenfunction Growth} implies a lower bound of $s(n)^{-1}$ on the Bernstein widths $b^{2}_{n-1}(L^{\infty}(\mathcal{X}), L^2(\nu))$ of $L^{\infty}(\mathcal{X})$ in $L^2(\nu)$. Note, that although the orthonormality of $\{e_i\}_{i = 1}^{\infty}$ in $L^2(\nu)$ is $\nu$-dependent, this system is \textit{always} an orthonormal basis in $\mathcal{H}_K$ by separability and the injectivity of $I_{\nu}$ (see Theorem 3.3 in \cite{steinwart2012mercer}). Moreover, Assumption \ref{Eigenfunction Growth} is independent of the Mercer eigenvalues $\{\mu_i\}_{i = 1}^{\infty}$ and is hence applicable to any RKHS with a common eigenbasis (e.g. the popular Fourier basis in $L^2([-\pi, \pi]^d)$). 
Indeed, we will see in Lemma \ref{Isoperimetric Equivalance} that the growth rate in Assumption \ref{Eigenfunction Growth} can often be wholly characterized by the ``highest-dimensional'' small balls in the measure space $(\mathcal{X}, \nu)$, without requiring any explicit knowledge of the eigenfunctions $\{e_i\}_{i \in \mathbb{N}}$. In approximation theory, the square reciprocal of the left-hand side in Assumption \ref{Eigenfunction Growth} is known as the \textit{Christoffel function}, whose asymptotic growth (and its interaction with the measure $\nu$) is an active area of research \cite{mate1980bernstein, xu1996asymptotics, at983168rf141986orthogonal} with diverse applications to quadrature/interpolation and random matrices \cite{van2006asymptotics} and \cite{at983168rf141986orthogonal}. Moreover, it is important to note that Assumption 3 is significantly more general than the more typical assumption of uniformly bounded bases, which has been shown to be violated by several common kernels (see discussion in \cite{zhou2002covering, steinwart2012mercer}).  \\ 

As we will see in Lemma \ref{Optimal Smoothness Example}, the Christoffel function encodes highly relevant information regarding the support and singularities of the measure $\nu$, and has been central to the study of orthogonal polynomials \cite{simon2010szegHo}, Schrodinger operators \cite{jitomirskaya1999power}, and spectral theory \cite{widom1967polynomials} over the last half century. Recently, data-driven applications of the Christoffel function have been explored for outlier detection \cite{askari2018kernel}, leverage scoring \cite{pauwels2018relating}, and sampling \cite{fanuel2022nystrom}. These works have notably demonstrated that the population Christoffel function can be strongly approximated by its empirical counterpart, which can be constructed in linear sample complexity and enjoys uniform convergence guarantees \cite{lasserre2019empirical}. Moreover, the empirical Christoffel function may be evaluated globally on the domain, with its asymptotic properties providing key insights into structure and support of the measure $\nu$ (see Lemma \ref{Optimal Smoothness Example}). These features of the Christoffel function hence suggest that it may be a more natural object to study than the Mercer spectrum or operator-theoretic ``statistical'' (``effective'') dimension \cite{blanchard2018optimal}, which are notoriously intractable to data-driven estimation for non-(quasi)uniform measures, yet are almost always preconditioned in the kernel regression literature \cite{blanchard2018optimal, blanchard2020kernel, fischer2020sobolev}.  \\ 

While Assumptions \ref{Domain Condition}-\ref{Opt Smoothness} are not necessary for establishing the upper bound in Theorem \ref{Main Upper}, they are required to demonstrate its optimality. Namely, the radiality assumption on the kernel $K$ in Assumption \ref{Radial Kernel} enables the tight estimation of the approximation widths of $\mathcal{H}_K$ in $L^{\infty}(\mathcal{X})$ via the compactness properties of the domain $\mathcal{X}$, the latter of which are characterized by the entropy numbers specified in Assumption \ref{Domain Condition}. In section \ref{Non-Radial}, we discuss how we can relax this radiality condition, when our RKHS is the domain of a local, regular Dirichlet form --- in these settings it is sufficient that the associated Hunt process on $L^2(\mathcal{X}, dx)$ escapes ``uniformly fast'' from small balls (see Corollary \ref{Alternative Spectral Estimates with Escape Rate} in section \ref{Non-Radial} and \cite{grigor2014heat} for details). The dilation condition on $\mathcal{F}_d \kappa$ essentially imposes polynomial Fourier decay of the kernel $\kappa$ (a notable example of the latter being the popular Mat\'ern kernel). This dilation condition is primarily intended for streamlining the analysis, it could be removed at the expense of requiring Assumption \ref{Opt Smoothness} be satisfied instead by a dilation of the Fourier transform $\mathcal{F}_d \kappa$ (this tradeoff is highlighted in Remark \ref{Remark on Dilations} in the Appendix \ref{Lower Bound Proof}; note however the dilation condition is critical for the expository Lemmas \ref{Optimal Smoothness Example} and \ref{Isoperimetric Equivalance}). Similarly, the precise asymptotic behavior of the entropy numbers of $\mathcal{X}$ is only included for simplicity; the general case would require replacing $t\mathcal{F}_d\kappa(t^{\frac{1}{d}})$ with $\epsilon_t(\mathcal{X})^{-d}\mathcal{F}_d\kappa(\epsilon_t(\mathcal{X})^{-1})$ in Assumption \ref{Opt Smoothness}. In Lemmas \ref{Isoperimetric Equivalance} and \ref{Alternative Spectral Estimates} we see how entropic rate encodes the topological dimension of $\mathcal{X}$ (in contrast to the effective dimension of the weighted space $(\mathcal{X}, \nu)$). The requirement $\mathcal{X} \subset \mathbb{R}^d$ is also primarily made for simplicity; it suffices for $\mathcal{X} \subset \mathbb{A}$ where $\mathbb{A}$ is a $d$-dimensional locally compact Abelian group which, via Pontryagin duality, admits a Fourier transform. Indeed, we will briefly violate Assumption \ref{Domain Condition} in Lemmas \ref{Optimal Smoothness Example} and \ref{Mercer Differencing} when we allow $\mathcal{X} = \partial \mathbb{D}$. Assumption \ref{Opt Smoothness} ensures that the index function $\psi$ is indeed the ``optimal'' choice for characterizing the embedding of our hypothesis class $\mathcal{H}_K$ in $L^{\infty}(\mathcal{X})$. It is important to note that while Assumptions \ref{Embedding Condition} and \ref{Source Condition} implicitly depend on the ambient measure $\nu$ (as the index functions act spectrally on the Mercer operator $T_{\nu}$), the optimality condition is independent of this measure (which is typically unknown in practice and determined by the data-generating process) given the orthonormal basis $\{e_i\}_{i = 1}^{\infty}$ of $\mathcal{H}_K$ in Assumption \ref{Eigenfunction Growth} (or equivalently the effective dimension \cite{villani2009optimal} of the measure space $(\mathcal{X}, \nu)$ as in Lemma \ref{Isoperimetric Equivalance}). This contrasts with the exact decay rates of the eigenvalues $\mu_i(T_{\nu})$ \cite{fischer2020sobolev, blanchard2018optimal} traditionally required to establish minimax optimal learning rates. Indeed, as is shown in Appendix \ref{Lower Bound Proof}, the combination of Assumptions \ref{Embedding Condition}, \ref{Growth Conditions}, \ref{Eigenfunction Growth}, \ref{Domain Condition}, \ref{Radial Kernel}, and \ref{Opt Smoothness} enable us to infer that $\psi^{-1}(s(i)^{-1}) \preceq \mu_i \preceq \psi^{-1}(s(i)^{-\frac{1}{1 + \epsilon}})$ (for all $\epsilon > 0$) via a comparison of $c^{2}_n(\mathcal{H}_K, L^{\infty}(\mathcal{X}))$ (estimated by $n\mathcal{F}_d \kappa(n^{\frac{1}{d}}))$ and $b^2_n(L^{\infty}(\mathcal{X}), L^2(\nu))$ (estimated by $s(n)^{-1}$). As we will see, these assumptions collectively imply $ \frac{i\mathcal{F}_d \kappa(i^{\frac{1}{d}})}{s(i)} \preceq \mu_i \preceq (\psi^{-1} \circ \psi^{\frac{1}{1 + \epsilon}})\Big(\frac{i\mathcal{F}_d \kappa(i^{\frac{1}{d}})}{s(i)}\Big)$ for all $\epsilon > 0$, i.e. there is a \textit{gap} between the Fourier spectrum ($\mathcal{F}_d \kappa(i^{\frac{1}{d}})$) of $K(x, y)$ and the Mercer spectrum ($\mu_i$) in $L^2(\nu)$ quantified by the Christoffel function $s(i)^{-1}$. When such a gap occurs, Assumption \ref{Opt Smoothness} characterizes the ``tightness'' of the weak embedding $H^{\psi} \stackrel{w}{\hookrightarrow} L^{\infty}(\mathcal{X})$ in terms of the index function $\psi$. In the following result, we demonstrate that a gap can be produced when the density $d\nu$ decays rapidly near a zero. Intuitively, such measures possess ``higher-dimensional'' small balls in the vicinity of discrete ($\nu$-neglgible) sets that can still cause $T_{\nu}$ to betray the Fourier spectrum. 

\begin{lemma}
\label{Optimal Smoothness Example}
Let $\nu$ be the unique probability measure on the Riemannian circle $\partial \mathbb{D}$ with $d\nu \propto |1 - e^{i\theta}|^{2k}d\theta$ ($\theta \in [-\pi, \pi]$) and $k \in \mathbb{N}$. Suppose $\{e_i\}_{i \geq 1}$ is a basis of polynomials. Then,  $\mu_i \asymp \frac{i[\mathcal{F}\kappa](i)}{s(i)}$. Additionally, there exists a function $\psi$ satisfying Assumptions \ref{Embedding Condition}, \ref{Growth Conditions}, and \ref{Opt Smoothness}. 
\end{lemma}

Measures of this form play a pivotal role in probability theory, statistical physics, and econometrics \cite{bottcher2008asymptotic, samoradnitsky2017stable}. Most notably, in 
time-series analysis, $d\nu \propto |1 - e^{i\theta}|^{2k}d\theta$ is the spectral density of a moving average process with unit root at zero frequency. These processes form the canonical example of a noninvertible ARMA model, where the present state of the process cannot be expressed in terms of its past values. In such settings, $\nu$ may also arise from the \textit{differencing} operation, in which the characteristic polynomial of an ARMA process is multiplied by $(1 - z)^{\ell}$ to difference the data $\ell$ times and remove an unit roots in the autoregressive component that are causing nonstationarity. When a time series is ``overdifferenced'', this transformation may instead produce unit root(s) in the moving average component, leading to a spectral density such as $\nu$ above. Intuitively, differencing ``smooths'' the data and removes any non-constant trend. Analogously the Mercer operator $T_{\nu}$ (for $d\nu \propto |1 - e^{i\theta}|^{2k}d\theta$) can be perceived as smoothing (or differencing) its argument before convolving the result with the kernel $K(x, y) = \kappa(x - y)$. In the following lemma, we characterize Assumption \ref{Opt Smoothness} in terms of the range space of $T_{\nu}$:

\begin{lemma}
\label{Mercer Differencing}
Let $d\nu$ be as in Lemma \ref{Optimal Smoothness Example} and suppose $\psi(t) = t^{\beta}$ for some $\beta \in (0, 1)$, and $L^2(\partial \mathbb{D}) \subset L^2(\partial \mathbb{D}, \nu)$. Let $s = \frac{1}{2\beta} + \frac{k(1 - \beta)}{\beta}$. Then Assumption \ref{Opt Smoothness} is equivalent to $K_1 ||P_sf||_{L^2(\nu)} \leq ||T_{\nu}f||_{H^{s}(\partial\mathbb{D})} \leq K_2 ||P_sf||_{L^2(\nu)}$ for some $K_1, K_2 > 0$ (where $H^m(\partial\mathbb{D})$ denotes the Sobolev space of order $m > \frac{1}{2}$ on the unit circle $\partial\mathbb{D}$ and $P_m: L^2(\partial\mathbb{D}, \nu) \to H^m(\partial\mathbb{D})$ is the orthogonal projection in $L^2(\partial\mathbb{D}, \nu)$). 
\end{lemma}

Indeed, we observe that a larger value of $k$ results in a smoother range space for $T_{\nu}$, while the scale $\beta \in (0, 1)$ exhibits an inverse effect. Recall, that for power scales $\psi(t) = t^{\beta}$, $\beta$ characterizes the distance between $\mathcal{H}_K$ and $L^{\infty}(\mathcal{X})$, with a smaller $\beta$ reflecting a more compact embedding $\mathcal{H}_K \hookrightarrow L^{\infty}(\mathcal{X})$ and hence a smoother RKHS $\mathcal{H}_K$. In Lemma \ref{Optimal Target Space}, we consider a similar interpretation of Assumption \ref{Opt Smoothness} for nonperiodic domains, and demonstrate the latter condition characterizes ``optimal'' range space for the Mercer operator (the general case requires some additional technicalities and hence is deferred to to Appendix \ref{Fourier Capacity and Range Space}) 

\begin{remark}
Note that in Lemmas \ref{Optimal Smoothness Example} and \ref{Mercer Differencing} we briefly violate Assumption \ref{Domain Condition} as the Riemannian circle $\partial \mathbb{D} \not \subset \mathbb{R}$; however we see that this does not cause any complications as $\partial \mathbb{D}$ is a locally compact Abelian group (and hence admits a Fourier transform by Pontryagin duality) and one-dimensional (so we do not have to worry about radiality). 
\end{remark}

In Lemma \ref{Optimal Smoothness Example}, it is notable that the ``corruption'' to the Fourier spectrum by a factor of $\frac{i}{s(i)}$ in the Mercer spectrum is produced only by the volume growth of $\nu$ at the $\theta = 0$ ($z =1$). Indeed, as observed in the proof in Appendix \ref{Example Proofs} and elaborated in \cite{mastroianni2000weighted}, for doubling measures, such as the $\nu$ in Lemmas \ref{Optimal Smoothness Example} and \ref{Mercer Differencing}, the Christoffel function at a given point is precisely characterized by the local volume growth of $\nu$. Observe that while for any $\theta \in [-\pi, \pi] \setminus \{0\}$, $\nu([\theta-r, \theta + r]) \asymp C(\theta) r$, at $\theta = 0$, we have $\nu([-r, r]) \asymp r^{2k + 1}$ (here the constant $C(\theta)$ depends on $\theta$ with $C(\theta) \to 0$ as $\theta \to 0$). Hence, at $\theta = 0$, the one-dimensional measure $d\nu$ exhibits high-dimensional volume growth. This behavior can be attributed to $d\nu$ arising from the projection of a uniform measure in $\mathbb{S}^{2k + 1}$. Indeed, it can be readily shown that if $z \sim d\nu$, then $\text{Re}(z) \sim 2|X \cdot e_{2k + 2}|^2 - 1$, where $X$ is a uniformly distributed random vector  on the unit sphere $\mathbb{S}^{2k + 1}$ and $e_{2k + 2} \in \mathbb{S}^{2k + 1}$ is the standard basis vector with $1$ in its last entry. 

While for general, high-dimensional measures, tight estimates for the Mercer spectrum cannot typically be explicitly computed as in Lemma \ref{Optimal Smoothness Example} without additional assumptions, the spectral function $s(n, x) \equiv \sum_{i = 1}^{n} e^2_i(x)$ still plays a useful role as a gauge function \cite{kigami2004local} in on-diagonal bounds for the heat kernel on $(\mathcal{X}, \nu)$ \cite{filbir2010quadrature}. In the following result, we consider metric measure spaces $(\mathcal{X}, \nu, |||\cdot |||_d)$ that satisfy an $\text{RCD}^{*}(0, N)$ curvature-dimension condition, which roughly requires them to possess a Hilbertian Sobolev space $\mathcal{H}^1(\mathcal{X}, \nu)$, nonnegative \textit{weighted} Ricci curvature, and dimension bounded above by $N$ (see \cite{ambrosio2018short} for a formal definition). Here $s(n, x)$ does indeed coincide with the local volume growth of $\nu$ at $x \in \mathcal{X}$ (when $\nu$ is doubling), and on-diagonal heat kernel bounds may be applied to derive isoperimetric inequalities for the Dirichlet capacity of the RKHS. In Lemma \ref{Isoperimetric Equivalance}, we demonstrate that the weak embedding in Assumption \ref{Embedding Condition} implies an isocapacitary inequality for the $\mathcal{H}_K$-capacity and Assumption \ref{Opt Smoothness} ensures this inequality is tight.

\begin{lemma}
\label{Isoperimetric Equivalance}
Suppose $\mathcal{X} \subset \mathbb{R}^d$. Let $(\mathcal{X}, \nu, |||\cdot|||_{d})$ be an $\text{RCD}^{*}(0, N)$ metric measure space without boundary and $\nu(dx) = w(x)dx$. Suppose:
\begin{equation}
\label{eq: Average Volume Growth}
    \lim_{r \to 0} \int_{\mathcal{X}} \frac{r^d}{\nu(B(x, r))}d\nu(x) < \infty
\end{equation}
and $\nu$ is doubling, with $\{e_i\}_{i \geq 1}$ the  eigenfunctions of the Neumann Laplacian $\Delta_{\nu} \equiv \Delta + \frac{\nabla w}{w} \cdot \nabla$ on $L^2(\mathcal{X}, \nu)$. Suppose $\langle \cdot, \cdot \rangle_{K}$ is a local, regular Dirichlet form on $L^2(\mathcal{X}, dx)$ with domain $\mathcal{H}_K(\mathcal{X})$, with $\alpha_{\frac{1}{\mathcal{F}_d \kappa}} < \infty$ and $\beta_{\mathcal{F}_d \kappa} = -\beta_{\frac{1}{\mathcal{F}_d \kappa}}$. Then, Assumption \ref{Embedding Condition} implies:
\begin{equation}
\label{eq: Weak Embedding Implication}
\psi^{1+\epsilon}\Big(\frac{t\mathcal{F}_d \kappa(t^{\frac{1}{d}})}{s(t)}\Big) \preceq \frac{1}{s(t)} \hspace{2mm} \text{as} \hspace{1mm} t \to \infty
\end{equation}
for all $\epsilon > 0$. Moreover, Assumption \ref{Opt Smoothness} is equivalent to the isocapacitary inequality:
\begin{equation}
\label{eq: Isocapacitary Equivalence}
\frac{1}{\text{cap}_{\mathcal{X}}\Big(B\Big(x, \frac{r}{2}\Big), B(x, r); \mathcal{H}_K\Big)} \asymp \inf_{y \in \mathcal{X}} \frac{\psi^{-1}(\nu(B(y, r))}{\nu(B(y, r))} \hspace{2mm} \text{as} \hspace{1mm} r \to 0
\end{equation} 
for all $x \in \mathcal{X}$. 
\end{lemma}

\begin{remark}
The factor $\frac{1}{2}$ in \eqref{eq: Isocapacitary Equivalence} is not important and can indeed be replaced by any $K \in (0, 1)$, as can be seen from the proof in Appendix \ref{Example Proofs}. Moreover, the additional assumption of $\beta_{\mathcal{F}_d \kappa} = -\beta_{\frac{1}{\mathcal{F}_d \kappa}}$ is needed to ensure uniform integrability in evaluating norms $||\cdot||_{K}$ at potential functions as $r \to 0$. This is not particularly restrictive, and is satisfied, for example, by Mat\'ern kernels.
\end{remark}
From Lemma \ref{Isoperimetric Equivalance}, we see that Assumption \ref{Opt Smoothness} is indeed an optimality condition that ensures that the weak embedding $H^{\psi} \stackrel{w}{\hookrightarrow} L^{\infty}(\mathcal{X})$ in Assumption \ref{Embedding Condition} is sharp. Indeed, Assumption \ref{Embedding Condition} and therefore \eqref{eq: Weak Embedding Implication} can always be satisfied by choosing a $\psi$ with sufficiently rapid decay at $0$ as long as $\psi(t) \succeq t$ as $t \to 0$ (for example, the trivial choice $\psi(t) = t$ simply produces the obvious embedding  $\mathcal{H}_K \hookrightarrow L^{\infty}(\mathcal{X})$). However, Assumption \ref{Opt Smoothness} ensures that this embedding is sharp by requiring the $\preceq$ in \eqref{eq: Weak Embedding Implication} to be an asymptotic equivalence $\asymp$. In \eqref{eq: Isocapacitary Equivalence}, we see that this equivalence is simply a geometric isocapacitary condition, where $\frac{\psi^{-1}(t)}{t}$ is the profile function that relates the $\mathcal{H}_K$-resistance of the annulus $B(x, r) \setminus B\Big(x, \frac{r}{2}\Big)$  to its infimal $\nu$-volume growth as it contracts toward its center. As the annulus shrinks, it encounters less resistance and the associated Hunt process (recall $\langle \cdot, \cdot \rangle$ is Markovian) started at $x \in \mathcal{X}$ visits $\partial B\Big(x, \frac{r}{2}\Big)$ less frequently before exiting the larger ball $B(x, r)$. The isocapacitary relation in \eqref{eq: Isocapacitary Equivalence} relates this escape behavior from $B(x, r)$ to its $\nu$-volume growth. Note that in \eqref{eq: Isocapacitary Equivalence}, we only take an infimum over $\mathcal{X}$ on the RHS, as the resistance on the LHS is uniform in $\mathcal{X}$ by the radiality of the kernel $K$. 

The condition \eqref{eq: Average Volume Growth} was employed in \cite{ambrosio2018short} to demonstrate that the Laplacian $\Delta_{\nu}$ on $(\mathcal{X}, \nu)$ obeys Weyl asymptotics, i.e. its eigenvalue counting function $\mathcal{N}(\lambda) = \{\#i: \lambda_i \leq \lambda\}$ behaves like $\lambda^{\frac{d}{2}}$ (which we need to ensure agreement with Assumption \ref{Domain Condition} on our domain). Intutively, \eqref{eq: Average Volume Growth} ensures that the $\nu$-volume growth is Euclidean,  $\nu$-almost everywhere (this is notably satisfied by the $\nu$ in Lemmas \ref{Optimal Smoothness Example} and \ref{Mercer Differencing}; high-dimensional volume growth is restricted to a singleton, which  is $\nu$-negligible). Hence, together, Lemmas \ref{Optimal Smoothness Example}, \ref{Mercer Differencing}, and \ref{Isoperimetric Equivalance} suggest a curious phenomenon: a measure $\nu$ can perturb the Mercer spectrum of $K(x, y)$ without disturbing the asymptotics of the Laplacian $\Delta_{\nu}$. This phenomenon can be attributed to the fact that Dirichlet capacities provide a much finer notion of size than Lebesgue-absolutely continuous Radon measures; indeed the embedding $\mathcal{H}_K \hookrightarrow L^{\infty}(\mathcal{X})$ implies \cite{maz2005conductor, kigami2012resistance} that every nonempty set has positive $\mathcal{H}_K$-capacity (even singletons!). Hence, although high-dimensional regular sets  may be $\nu$-negligible and not disturb condition \eqref{eq: Average Volume Growth} (and the asymptotic spectrum of the Laplacian), they are still nontrivial in $\mathcal{H}_K$-capacity and hence disturb the Mercer spectrum. 

Finally, the exponent $N \equiv \lim_{r \to 0} \frac{\log \inf_{y \in \mathcal{X}} \nu(B(y, r))}{\log r}$ of the infimal volume growth in \eqref{eq: Isocapacitary Equivalence} is precisely the smallest $N$ for which the Bakry-Emery curvature condition is satisfied on $RCD^{*}(0, N)$ spaces \cite{sturm2006geometry, ambrosio2018short}, and can alternatively be viewed as the volume exponent realizing the Bishop-Gromov lower bound (Theorem 2.3 in \cite{sturm2006geometry}). Hence, in these settings, we see that the optimality of the embedding $H^{\psi} \stackrel{w}{\hookrightarrow} L^{\infty}(\mathcal{X})$ can be completely characterized by the interaction between the kernel ($\kappa$), topological dimension ($d$), and the ``effective dimension'' $N$ of the metric measure space $(\mathcal{X}, \nu, |||\cdot|||_{d})$. 

\subsubsection{Moving Past Radiality}
\label{Non-Radial}

While the tight spectral estimates derived in Appendix \ref{Lower Bound Proof} assume the radiality of $K(x, y)$ (or at least equivalence to a radial kernel), here we demonstrate that in the general case, the Mercer spectrum can alternatively be estimated under the isocapacitary condition in \eqref{eq: Isocapacitary Equivalence} in lieu of Assumption \ref{Opt Smoothness}. Implicitly, \eqref{eq: Isocapacitary Equivalence} relaxes the radiality assumption by only requiring $\text{cap}_{\mathcal{X}}(B\Big(x, \frac{r}{2}\Big), B(x, r); \mathcal{H}_K)$ be independent of $x \in \mathcal{X}$, which under an elliptic Harnack inequality, is equivalent to uniform escape times of the Hunt process associated with $(\langle \cdot,  \cdot \rangle_K, \mathcal{H}_K)$ on $L^2(\mathcal{X}, dx)$ (elaborated below); it is immediate that these escape times are uniform when $K$ is radial. 

\begin{lemma}
\label{Alternative Spectral Estimates}
Suppose $C_1 n^{-\frac{1}{d}} \leq \epsilon_n(\bar{\mathcal{X}}) \leq C_2 n^{-\frac{1}{d}}$ in Assumption \ref{Domain Condition}. Assume the hypotheses of Lemma \ref{Isoperimetric Equivalance} and suppose
\begin{equation}
\label{eq: Isocapacitary Condition}
    \inf_{y \in \mathcal{X}} \frac{\psi^{-1}\Big(\nu(B(y, r))\Big)}{\nu(B(y, r))} \asymp \frac{1}{\text{cap}_{\mathcal{X}}\Big(B\Big(x, \frac{C_1 r}{8}\Big), B\Big(x, \frac{C_1 r}{4}\Big); \mathcal{H}_K \Big)}
\end{equation}
as $r \to 0$ for all $x \in \mathcal{X}$. Then, under Assumptions \ref{Embedding Condition}-\ref{Eigenfunction Growth} and \ref{Domain Condition}, we obtain:
\begin{equation*}
    \psi^{-1}(s(n)^{-1}) \preceq \mu_{n} \preceq \psi^{-1}(s(n)^{-\frac{1}{1 + \epsilon}})
\end{equation*}
for all $\epsilon > 0$.
\end{lemma}

Our proof of Lemma \ref{Alternative Spectral Estimates} in Appendix \ref{Relaxing Radiality} relies on estimating $\mu_i$ from below using the geometric approach of Grigor'yan, Netrusov, and Yau \cite{grigor2004eigenvalues, grigor1999isoperimetric}. Namely, we construct a candidate eigenspace for $\mu_i$ spanned by the potentials of a collection of disjoint annuli that sufficiently cover $\mathcal{X}$. The Courant-Fisher minimax principle then gives us a lower bound for $\mu_i$ characterized by the ratio of the annulus capacity to $\nu$-volume growth (note that when annuli $\mathcal{H}_K$-resistances satisfy a $\Delta_2$ condition \cite{grigor2014heat}, the constants $\frac{C_1}{4}$ and $\frac{C_1}{8}$ in \eqref{eq: Isocapacitary Condition} are not important). We note that \eqref{eq: Isocapacitary Condition} significantly simplifies our lower bound proof, and we may completely avoid developing the interpolation theory employed in Appendix \ref{Lower Bound Proof}.

When the expected small ball escape times  of the  $\mathcal{H}_K$-associated Hunt process $\{X_t\}_{t \geq 0}$ on $L^2(\mathbb{R}^d, dx)$ (see Definition \ref{Dirichlet Form definition} and Remark \ref{Trace Clarification}) depend only on the ball's radius and not its position, we obtain translation invariant estimates for the capacity (Theorem 3.12 in \cite{grigor2014heat}), and \eqref{eq: Isocapacitary Condition} can be further explicated in terms of the exit time function using the so-called ``Einstein's relation'' (equation 3.3 in \cite{grigor2014heat}). If the exit time function further satisfies the $\Delta_2$ condition, then the factors $\frac{C_1}{4}$  and $\frac{C_1}{8}$ in \eqref{eq: Isocapacitary Condition}  can be absorbed into the $\asymp$. We summarize these observations in the following corollary.

\begin{corollary}
\label{Alternative Spectral Estimates with Escape Rate}
Assume the hypotheses of Lemma \ref{Alternative Spectral Estimates} and suppose further that the Dirichlet form $(\langle \cdot, \cdot \rangle_{K}, \mathcal{H}_K(\mathcal{X}))$ satisfies an elliptic Harnack inequality (see Definition 3.2 in \cite{grigor2014heat}). Let $\{X_t\}_{t \geq 0}$ be the Hunt process associated with $(\langle \cdot, \cdot \rangle_{K}, \mathcal{H}_K(\mathcal{X}))$ on $L^2(\mathcal{X}, dx)$ and consider its exit time function:
\begin{equation*}
    E^{\Omega}(x) = \mathbb{E}_x[\tau_{\Omega}]
\end{equation*}
where $\tau_{\Omega}$ is the exit time of $X_t$ from $\Omega \subset \mathcal{X}$ and $x \in \mathcal{X}$. Let $B = B(x, r) \subset \mathcal{X}$ and suppose there exists a nondecreasing function $F$ such that:
\begin{align*}
    \sup_{y \in B} E^{B}(y) \leq CF(r) \\
    \inf_{y \in \delta B} E^{B}(y) \geq C^{-1}F(r)
\end{align*}
for some $\delta \in (0, 1)$ and $C > 0$. If
\begin{equation}
\label{eq: Isocapacitary Condition with Escape}
    \inf_{y \in \mathcal{X}} \frac{\psi^{-1}\Big(\nu(B(y, r))\Big)}{\nu(B(y, r))} \asymp \frac{F(r)}{r^d}
\end{equation}
as $r \to 0$ for all $x \in \mathcal{X}$, then, under Assumptions \ref{Embedding Condition}-\ref{Eigenfunction Growth} and \ref{Domain Condition}, we obtain:
\begin{equation*}
    \psi^{-1}(s(n)^{-1}) \preceq \mu_{n} \preceq \psi^{-1}(s(n)^{-\frac{1}{1 + \epsilon}})
\end{equation*}
for all $\epsilon > 0$.
\end{corollary}

The proof of this result follows directly from Lemma \ref{Alternative Spectral Estimates} and Theorem 3.12 in \cite{grigor2014heat} upon noting that in the proof of the former (see Appendix \ref{Relaxing Radiality}) the factor $\frac{C_1}{8}$ in \eqref{eq: Isocapacitary Condition} can be replaced by any $\delta_1 < \frac{C_1}{4}$. Observing that $r^d \asymp \text{Vol}(B(x, r))$ in \eqref{eq: Isocapacitary Condition with Escape}, we see how Assumption \ref{Opt Smoothness} directly characterizes a change of measure (from Lebesgue to $\nu$) via the scale function $\psi$ and the dynamics of the $\mathcal{H}_K$-associated Hunt process $X_t$. We emphasize the Hunt process $X_t$ is completely independent of the measure $\nu$, and symmetric with respect to the Lebesgue measure on $\mathcal{X}$ (note that since the small balls in \eqref{eq: Isocapacitary Condition with Escape} are compactly supported in $\mathcal{X}$, we can equivalently consider the $\mathcal{H}_K$-associated Hunt process on $(\mathbb{R}^d, dx)$). We will further develop this probabilistic perspective on misspecification will in an upcoming paper. 
\begin{remark}
It is worth noting that the choice of Laplacian eigenfunctions $\{e_i\}_{i \in \mathbb{N}}$ in Lemmas \ref{Isoperimetric Equivalance}, \ref{Alternative Spectral Estimates}, and Corollary \ref{Alternative Spectral Estimates with Escape Rate} is not special, indeed, in light of \eqref{eq: Average Volume Growth}, the proofs of these lemmas may be adapted for the eigenfunctions of any compact self-adjoint operator on $L^2(\nu)$ which enjoys sharp upper and lower on-diagonal heat kernel estimates. See \cite{grigor2012two} for sufficient conditions. 
\end{remark}

\subsubsection{Related Work}
Assumptions \ref{Embedding Condition} and \ref{Source Condition} can be viewed as a generalization of those in \cite{fischer2020sobolev} to Hilbert scales. Namely, these assumptions reduce to the corresponding source and embedding conditions in \cite{fischer2020sobolev} with the choice of $\phi(t) = t^{\beta}$ and $\psi(t) = t^{\alpha}$ (using their notation). The additional embedding $H^{\phi} \hookrightarrow H^{\psi}$ in Assumption \ref{Embedding Condition} is not strictly necessary for the derivation of minimax rates, but helps simplify the analysis. Intuitively, this embedding ensures that $f^{*}$ is bounded. \\

We emphasize that a distinctive feature of our analysis is that we do not make any direct assumption on the eigendecay of $T_{\nu}$ or its effective dimension. Indeed, while \cite{blanchard2018optimal, blanchard2020kernel} established learning rates in the stronger $\|\cdot \|_{K}$ norm for more general regularization schemes, their analysis only considered Holder-type source conditions and hinged on a particular asymptotic eigendecay of $T_{\nu}$ (which depends on the unknown measure $\nu = P|_{\mathcal{X}}$). \cite{rastogi2020tikhonov} and \cite{rastogi2017optimal} consider general source conditions similar to this paper (and for more general regularization schemes), although they too place conditions on the statistical dimension of $T_{\nu}$ (which is strongly related to its eigendecay). Recently, \cite{lu2020balancing} considered Lepski-based parameter selection strategy for general regularization schemes, and demonstrated that known estimates on the statistical dimension of $T_{\nu}$ can be effectively leveraged for adapting to misspecification. \\

To the best of our knowledge, this is the first work to establish minimax learning rates under general source conditions without appealing to any direct assumptions on the eigendecay or statistical dimension of the Mercer operator $T_{\nu}$. The primary focus of this work is to elucidate the precise interaction of the kernel $K(x, y)$ and the sampling measure $\nu$ in determining the Mercer eigendecay, and the robustness of this spectrum to modifications of the measure $\nu$. In the process, we aim to identify the most generic characterization of $\nu$ which, in conjunction with $K$, completely establishes the Mercer spectrum. Our approach leverages the $L^{\infty}(\mathcal{X})$ embedding of the Hilbert scale $\mathcal{H}^{\psi}$, the distance between $L^{\infty}(\mathcal{X})$ and $L^2(\nu)$ (Assumption \ref{Eigenfunction Growth}), and a compatibility condition between $\psi$ and the Fourier transform of the kernel (Assumption \ref{Opt Smoothness}) to \textit{infer} spectral information on $T_{\nu}$. In a sense, our analysis demonstrates that, via an interpolation inequality, evaluating the tightness of the weak embedding $\mathcal{H}^{\psi} \stackrel{w}{\hookrightarrow} L^{\infty}(\mathcal{X})$ (and thereby the Mercer spectrum) is in a sense equivalent to evaluating the optimality of $\psi$ as an isocapacitary profile function (via Lemma \ref{Isoperimetric Equivalance}). When $\psi$ is indeed optimal, the interpolation inequality (Lemma \ref{Interpolation Inequality}) provides the correct ``change of measure'' (see also Corollary \ref{Alternative Spectral Estimates with Escape Rate}). However, this delicate relationship hinges on a certain degree of local isotropy, characterized either by radiality of the kernel or the uniformity of small-ball escape times in the $\mathcal{H}_K$-assoiated Hunt process on $L^2(\mathcal{X}, dx)$ (see Section \ref{Non-Radial}). 

\vspace{-0.5em}
\subsection{Results}
In presenting our main results, we will always assume the validity of Assumption \ref{Opt Smoothness}, but emphasize that, in light of section \ref{Non-Radial}, we may alternatively assume \eqref{eq: Isocapacitary Condition} or \eqref{eq: Isocapacitary Condition with Escape}. We first present our upper bound on the mean-square error in Theorem \ref{Main Upper}. 
\begin{theorem}
\label{Main Upper}
Suppose Assumptions \ref{Embedding Condition}-\ref{Opt Smoothness} hold. Let $\tilde{s}(t) = s^{-1}(t^{-1})$. Then for any sufficiently small $\epsilon > 0$, choosing $\lambda_n \asymp (\frac{\phi}{\tilde{s} \circ \psi^{1 + \epsilon}})^{-1}(n^{-1})$, we obtain, with probability $1 - 2\delta$:
\begin{equation}
\label{eq:Upper Bound}
  ||f_{D, \lambda_n} - f^{*}||_{L^2(\nu)} \preceq \log(\delta^{-1})\sqrt{\phi\Big(\Big(\frac{\phi}{\tilde{s} \circ \psi^{1 + \epsilon}}\Big)^{-1}(n^{-1})\Big)}
\end{equation}
\end{theorem}

\vspace{-0.5em}
Intuitively, the function $\phi\Big(\Big(\frac{\phi}{\tilde{s} \circ \psi^{1 + \epsilon}}\Big)^{-1}\Big)$ appearing in \eqref{eq:Upper Bound} captures the distance between $\mathcal{H}^{\phi}$ and the hypothesis class $\mathcal{H}_K$, i.e. the ``degree of misspecification''. Indeed, since $\psi^{-1}(s(i)^{-1}) \preceq \mu_i \preceq \psi^{-1}(s(i)^{-\frac{1}{1 + \epsilon}})$ (as shown in Appendix \ref{Lower Bound Proof}), the upper rate effectively compares $\phi(\mu_i)$ to $\mu_i$. For example, when $\phi(t) = t^{\beta}$, $\psi(t) = t^{\alpha}$, and $\mu_i \asymp i^{-\frac{1}{p}}$ (for $\alpha, \beta, p \in (0, 1)$) as in \cite{fischer2020sobolev}, then $n^{-\frac{\beta}{p + \beta}} \preceq \phi\Big(\Big(\frac{\phi}{\tilde{s} \circ \psi^{1 + \epsilon}}\Big)^{-1}\Big)(n^{-1}) \preceq n^{-\frac{\beta}{p(1 + \epsilon) + \beta}}$ for all sufficiently small $\epsilon > 0$. If $\frac{\beta}{p}$ is close to $1$, then $\mathcal{H}^{\phi}$ is much closer to $L^{\infty}(\mathcal{X})$ than $\mathcal{H}_K$ and the setting is poorly specified, producing a slow rate close to $n^{-\frac{1}{4}}$. If $\beta \gg p$ is large, i.e. $\mathcal{H}^{\phi}$ is close to $\mathcal{H}_K$, our setting is well-specified, and we approach the fast rate of $n^{-\frac{1}{2}}$. We note that the $\epsilon$-error in \eqref{eq:Upper Bound} is an unavoidable artifact of the weak embedding $\mathcal{H}^{\psi} \stackrel{w}{\hookrightarrow} L^{\infty}(\mathcal{X})$ in Assumption \ref{Embedding Condition}, namely by Definition \ref{Weak Embedding Def}, $\psi$ is merely an infimum of scale functions paramterizing spaces lying in $L^{\infty}(\mathcal{X})$. In other words, in order to identify the ``optimal'' scale function $\psi$ (Assumption \ref{Opt Smoothness}), we need to consider spaces $H^{\psi}$ that potentially lie \textit{just} outside $L^{\infty}(\mathcal{X})$. \\


The proof of Theorem \ref{Main Upper} hinges on the following bias-variance decomposition, which illustrates that, up to a constant, both components of the error can be controlled purely in terms of $\phi, \psi$, and $s$: 

\begin{lemma}
Suppose Assumptions \ref{Embedding Condition}-\ref{Opt Smoothness} hold. Then, for all sufficiently small $\epsilon > 0$:
\label{Bias-Variance Decomposition}
\begin{equation}
\label{eq:BV}
    ||f^{*} - f_{D, \lambda}||_{L^2(\nu)} \preceq  ||f^{*}||_{\phi}\sqrt{\phi(\lambda)} + ||k^{\psi^{1 + \epsilon}}||_{\infty}||f^{*}||_{\phi}\sqrt{\frac{\log(2\delta^{-1})}{n}\Big(\frac{1}{n\psi^{1 + \epsilon}(\lambda)} + \frac{\phi(\lambda)}{\psi^{1 + \epsilon}(\lambda)} + s^{-1}(\psi(\lambda)^{-1 -  \epsilon})\Big)}
\end{equation}
with probability $1 - 3\delta$. 
\end{lemma}

As noted above, we observe here that the kernel only appears in \eqref{eq:BV} via the constant term $||k^{\psi^{1 + \epsilon}}||_{\infty}$, and hence, given the geometric functions $\psi, \phi$, and $s$, $K$ does not directly influence the asymptotic behavior of the upper bound as $\lambda \to 0$. This surprising absence of the kernel (given the aforementioned profile functions), can be attributed to the embedding conditions in Assumption \ref{Embedding Condition}, which ensure that (via a Gagliardo-Nirenberg type inequality) uniform bounds on the bias depend only on the \textit{ratios} between norms in $\mathcal{H}^{\psi}$, and $\mathcal{H}^{\phi}$ and $\mathcal{H}_K$, respectively. Moreover, due to relative growth rates of the index functions implied by compactness, these capacity ratios depend solely on $\phi$, $\psi$, and $s$. \\

We now discuss the optimality of the upper bound in Theorem \ref{Main Upper}. The derivation of a minimax lower bound near-matching \eqref{eq:Upper Bound} is significantly more involved and hinges crucially on the sharpness of the weak embedding $\mathcal{H}^{\psi} \stackrel{w}{\hookrightarrow} L^{\infty}(\mathcal{X})$, which in turn depends on the compactness properties of the domain $\mathcal{X} \subset \mathbb{R}^d$ (captured by the entropy condition in Assumption \ref{Domain Condition}) and the smoothness of the kernel $\kappa$. Unlike the derivation of the upper bound, here the analysis does not reduce simply to a comparison of the index functions, as we must demonstrate the maximality of the spectral function $\psi$ in ``stretching'' the kernel class $\mathcal{H}_K$. Informally, we need to ensure that the enlarged (infinite-dimensional) ellipsoid $B(\psi(\mathcal{H}_K))$ is ``maximal'' relative to $L^{\infty}(\mathcal{X})$, otherwise our use of $\mathcal{H}^{\psi}$ as a proxy for $L^{\infty}(\mathcal{X})$ in the derivation of \eqref{eq:Upper Bound} is suboptimal (i.e. $||\cdot||_{\mathcal{H}^{\psi}}$ may provide a very coarse approximation of $||\cdot||_{\infty}$). We quantify this maximality via estimating the squared Gelfand widths of $\mathcal{H}_K$ in $L^{\infty}(\mathcal{X})$, which can then be compared to the eigenvalues $\mu_i$ (the squared Gelfand widths of $\mathcal{H}_K$ in $L^{2}(\nu)$; see e.g. Prop. 5 in \cite{mathe2008direct}) via the Gagliardo interpolation inequality. The former widths are independent of the measure $\nu$ and depend only on the smoothness of $\kappa$ (captured by its Fourier decay) and the compactness of $\mathcal{X}$. Intuitively, the weak embedding $\mathcal{H}^{\psi} \stackrel{w}{\hookrightarrow}  L^{\infty}(\mathcal{X})$ is sharp (and our learning rate in \eqref{eq:Upper Bound} is optimal) if the index function $\psi$ provides the correct ``change-of-measure'' from $\nu$ to the $d$-dimensional Lebesgue measure (see also Lemmas \ref{Isoperimetric Equivalance} and \ref{Alternative Spectral Estimates}, and Corollary \ref{Alternative Spectral Estimates with Escape Rate}). 

\begin{theorem}
\label{Lower Bound}
There exists a distribution $P$ on $\mathcal{X} \times \mathbb{R}$ satisfying Assumption \ref{Subexponential Noise} with $P|_{\mathcal{X}} = \nu$, $||f^{*}_{P}||_{\phi} \leq B_{\phi}$, $||f^{*}_{P}||_{\infty} \leq B_{\infty}$, such that, for any learning algorithm $\mathcal{D} \mapsto f_D$, we have with $P^{n}$ probability not less than $1 - \frac{432\tau K^{\log \tau}}{U_{d, \infty, \phi}\sigma^2\log 2} > 0$:
\begin{equation*}
    ||f_{D} - f^{*}_P||_{L^2(\nu)} \succeq \sqrt{\tau \phi\Big(\Big(\frac{\phi}{\tilde{s} \circ \psi}\Big)^{-1}(n^{-1})\Big)}
\end{equation*}
where $U_{d, \infty, \phi}$ does not depend on $n$ or $\tau$ and $K > 1$ is an absolute constant.
\end{theorem}

We note that in Theorem \ref{Lower Bound} we can always ensure the probability $1 - \frac{432\tau K^{\log \tau}}{U_{d, \infty, \phi}\sigma^2\log 2} > 0$ through a sufficiently small choice of $\tau > 0$. Collectively, Theorems \ref{Main Upper} and \ref{Lower Bound} demonstrate that the mean squared error of kernel ridge regression can be completely described in terms of \textit{three geometric quantities}: the \textit{effective dimension} of the metric measure space $(\mathcal{X}, \nu, |||\cdot|||_{d})$ (i.e. $s(n)$, see Lemmas \ref{Isoperimetric Equivalance} and \ref{Alternative Spectral Estimates}), and the \textit{``spectral distances''} from the hypothesis class $\mathcal{H}_K$ to $L^{\infty}(\mathcal{X})$ and the learning target $f^{*}$ (characterized by the profile functions $\psi$ and $\phi$, respectively).
\section{Conclusion}
In this paper, we derive novel minimax learning rates for ridge regression over Hilbert scales with general source conditions. Our analysis hinges on a detailed derivation of sharp estimates for kernel eigendecay, that elucidate the precise interaction of kernel, measure, dimension, and functional geometry in characterizing the complexity of the RKHS. Our approach is based on estimating the Gelfand widths of the kernel class in $L^{\infty}(\mathcal{X})$ via a new Fourier capacity condition, which characterizes the complexity of the kernel through an isocapacitary condition on small balls in the metric measure space $(\mathcal{X}, \nu, |||\cdot|||_{d})$. In the process, we examine the influence of kernel Dirichlet capacities (Lemmas \ref{Isoperimetric Equivalance} and \ref{Alternative Spectral Estimates}) and the local dynamics of the RKHS-associated Hunt process (Corollary \ref{Alternative Spectral Estimates with Escape Rate}) in determining the relative smoothness of the kernel class in $L^2(\nu)$. Our analysis suggests that a capacitary approach to large deviations, as opposed to a purely measure-theoretic perspective, may be promising in the context of finite-sample asymptotics for nonparametric regression. Future work involves further developing this probabilistic perspective on ridge regression in both offline and online learning settings. 

\printbibliography

\appendix

\section{Proofs of Lemma \ref{Optimal Smoothness Example}, Lemma \ref{Mercer Differencing}, and Lemma \ref{Isoperimetric Equivalance}}
\label{Example Proofs}
In the following result, we construct a ``pseudo-eigenvector'' for the Toeplitz (moment) matrix of the density $d\nu$, that enjoys certain desirable asymptotic properties which we will exploit in the proof of Lemma \ref{Optimal Smoothness Example}. For any $f \in C(\mathbb{R})$, our eigenfunction takes the form:
\begin{equation}
\label{Toeplitz Eigenfunction}
V_n f \equiv \frac{1}{\sqrt{n+1}}\sum_{m = -\infty}^{\infty} f\Big(\frac{(m + 1)\pi}{n+1}\Big)z^m
\end{equation}
If $f\Big(\frac{(m + 1)\pi}{n+1}\Big) \neq 0$ only for $m \in \{0, 1, \ldots, n\}$, we write $\tilde{V}_n f \in \mathbb{R}^{n + 1}$ for the corresponding vector with entries $(\tilde{V}_n(f))_{m} = \frac{f\Big(\frac{(m + 1)\pi}{n+1}\Big)}{\sqrt{n+1}}$ for $0 \leq m \leq n$. The constructed pseudomode in \eqref{Toeplitz Eigenfunction} resembles the approximate eigenvectors studied in \cite{bottcher2008asymptotic}; we demonstrate that for a careful choice of $f$, $V_n f$ exhibits stronger spectral approximation properties for the particular symbol $d\nu = |1 - z|^{2k}$ than those obtained in the general case in \cite{bottcher2008asymptotic} and \cite{bottcher2009first}. Our result may be of independent interest for the study of Toeplitz matrices. 
\begin{lemma}
\label{Pseudomodes}
Let $d\nu = w(\theta)d\theta \propto |1 - e^{i\theta}|^{2k}d\theta$ on $(0, 2\pi)$ and let $T_n(\nu)$ denote the Toeplitz matrix given by $T_n(\nu) = ((\mathcal{F}w)(i - j))_{i, j = 1, \ldots, n}$. Then, for any $C^{\infty}(\mathbb{R})$ function $\phi$ supported on $[0, \pi]$, we have $\langle \tilde{V}_{n-1} \phi, T_n(\nu) \tilde{V}_{n-1} \phi \rangle \asymp n^{-2k}$. 
\end{lemma}
\begin{proof}
For any $f \in L^1(\mathbb{R})$ we define the \textit{backward difference}:
\begin{equation*}
    \nabla_{h}[f](x) = f(x) - f(x-h)
\end{equation*}
where $h > 0$ and $x \in \mathbb{R}$. When $h = 1$, we simply write $\nabla[f](x)$. We first observe that for any $g = \sum_{m = -\infty}^{\infty} c(m) z^m$ and $k \geq 1$:
\begin{equation*}
    (1 - z)^kg = \sum_{m = -\infty}^{\infty} \nabla^k[c](m) z^m
\end{equation*}
Hence, we have:
\begin{equation*}
    \frac{1}{2\pi}\int_{0}^{2\pi} |1 - e^{i\theta}|^{2k}|g(e^{i\theta})|^2 d\theta = \sum_{m = -\infty}^{\infty} (\nabla^k[c](m))^2
\end{equation*}
Now, let $\phi$ be a $C^{\infty}(\mathbb{R})$ function supported on $[0, \pi]$ (e.g. $\phi(x) = e^{\Big(1 - \Big(\frac{2x}{\pi} - 1\Big)^2\Big)^{-1}}$ on $[0, \pi]$). Then, we have that:
\begin{align}
     \frac{1}{2\pi}\int_{0}^{2\pi} |1 - e^{i\theta}|^{2k}|V_n \phi(e^{i\theta})|^2 d\theta & = \frac{1}{n+1}\sum_{m = -\infty}^{\infty} \Big(\nabla^k_{(n+1)^{-1}}[\phi]\Big(\frac{(m + 1)\pi}{n+1}\Big)\Big)^2 \nonumber \\
     & \leq \frac{C}{n+1}\sum_{m = 0}^{n} \Big(\frac{1}{(n + 1)^k} + \frac{1}{(n + 1)^{k+1}}\Big)^2 \label{eq: Derivative Bound} \\
    & + \frac{1}{n+1}\sum_{n< m \leq n+k} \frac{1}{(n + 1)^{2k+2}} \nonumber \\
     & \leq \frac{5C}{(n + 1)^{2k}} \nonumber
\end{align}
where in \eqref{eq: Derivative Bound}, $C = \max_{0 \leq j \leq k} ||\phi^{(j)}||_{\infty} < \infty$ due to the compact support of $\phi$. 
\end{proof}
\begin{proof}[Proof of Lemma \ref{Optimal Smoothness Example}]
Let $\mathcal{P}_m = \text{span}\{1, z, \ldots, z^m\}$ denote the space of $m$-degree polynomials on $\partial \mathbb{D}$. Note that $\text{span}\{e_i\}_{i = 1}^{m} = \mathcal{P}_{m - 1}$ by the assumption that $e_i$ is a degree $i-1$ polynomial. Hence, by the Courant-Fisher theorem we may express the eigenvalues $\{\mu_i\}_{i = 1}^{\infty}$ of $T_{\nu}$ as:
\begin{align*}
    \mu_{m + 1} & = \min_{p \in \mathcal{P}_m} \frac{||p||^2_{L^2(\nu)}}{||p||^2_{K}} \\
    & =  \min_{p \in \mathcal{P}_m} \frac{||p||^2_{L^2(\nu)}}{\sum_{j = 0}^m \frac{p^2_j}{(\mathcal{F}\kappa)(j)}}
\end{align*}
where we have expressed $p \in \mathcal{P}_m$ as $p = \sum_{j = 0}^{m} p_j z^j$. The last line follows from the fact that by Bochner's theorem and the stationarity of $K(\theta_1, \theta_2) = \kappa(\theta_1 - \theta_2)$, the Fourier basis $\{1, z, \ldots, z^m\}$ is an orthogonal basis in $\mathcal{H}_K$ despite not being orthogonal in $L^2(\nu)$. We first demonstrate the lower bound. Borrowing notation from Lemma \ref{Pseudomodes} and noting that $\mathcal{F}(\kappa)$ is nonincreasing, we have that:
\begin{align*}
    \mu_{m + 1} & = \min_{p \in \mathcal{P}_m} \frac{||p||^2_{L^2(\nu)}}{\sum_{j = 0}^m \frac{p^2_j}{(\mathcal{F}\kappa)(j)}} \\
    & \geq (\mathcal{F}\kappa)(m) \cdot \min_{p \in \mathcal{P}_m} \frac{||p||^2_{L^2(\nu)}}{\sum_{j = 0}^m p^2_j} \\
    & = (\mathcal{F}\kappa)(m) \cdot \lambda_{\min}(T_{m+1}(\nu)) \\
    & \asymp (\mathcal{F}\kappa)(m)(m + 1)^{-2k}
\end{align*}
where the last line follows from the well-known asymptotic behavior of the smallest eigenvalue of $T_{n}(\nu)$; see e.g. Proposition 6.1 in \cite{babayan2022asymptotic} or Theorems 4.11 and 4.32 in \cite{bottcher2005spectral}. We now derive a matching upper bound using the test function constructed in Lemma \ref{Pseudomodes}. Indeed, let $\phi$ be a $C^{\infty}(\mathbb{R})$ function supported on $[0, \pi]$. By the previous lemma, we have that:
\begin{align}
     \mu_{m + 1} & = \min_{p \in \mathcal{P}_m} \frac{||p||^2_{L^2(\nu)}}{\sum_{j = 0}^m \frac{p^2_j}{(\mathcal{F}\kappa)(j)}} \nonumber \\
     & \leq \frac{||V_m \phi||^2_{L^2(\nu)}}{\sum_{j = 0}^m \frac{(V_m \phi)^2_j}{(\mathcal{F}\kappa)(j)}} \nonumber \\
     & = \langle \tilde{V}_m \phi, T_{m+1}(\nu) \tilde{V}_m \phi \rangle \Big(\sum_{j = 0}^m \frac{(V_m \phi)^2_j}{(\mathcal{F}\kappa)(j)}\Big)^{-1} \label{eq: Toeplitz Sub} \\
     & \asymp (m + 1)^{-2k}\Big(\frac{1}{m+1}\sum_{j = 0}^m \frac{\phi^2\Big(\frac{(j + 1)\pi}{m+1}\Big)}{(\mathcal{F}\kappa)(j)}\Big)^{-1} \label{eq: Apply Pseudomodes}
\end{align}
where \eqref{eq: Toeplitz Sub} follows since $V_m \phi$ is an analytic degree $m$ polynomial and \eqref{eq: Apply Pseudomodes} is Lemma \ref{Pseudomodes}. Then, from the condition $0 \geq \alpha_{\mathcal{F}(\kappa)}(t) > - \infty$ in Assumption \ref{Radial Kernel} (note that $\alpha_{\mathcal{F}(\kappa)} \leq 0$ follows immediately as $\mathcal{F}(\kappa)$ is nonincreasing) we have that there exists a $\delta \in (0, 1)$ such that:
\begin{equation}
\label{eq: Basic Asymp Decay}
    \frac{(\mathcal{F}\kappa)(t)}{(\mathcal{F}\kappa)(st)} \geq s^{-2\alpha_{\mathcal{F}_d \kappa}}
\end{equation}
for all $s \in (0, \delta)$ and $t \geq 1$.  Then, plugging in $s = \frac{j}{m}$ and $t = m$ into \eqref{eq: Basic Asymp Decay}, we have by the boundedness of $[\delta, 1)$ that:
\begin{align*}
    \frac{1}{m+1}\sum_{j = 0}^m \frac{\phi^2\Big(\frac{(j + 1)\pi}{m+1}\Big)}{(\mathcal{F}\kappa)(j)} & \succeq \frac{1}{(m+1)(\mathcal{F}\kappa)(m)}\sum_{j = 0}^m \phi^2\Big(\frac{(j + 1)\pi}{m+1}\Big)\Big(\frac{j}{m}\Big)^{-2\alpha_{\mathcal{F}(\kappa)}} \\
    & \asymp \frac{1}{(m+1)(\mathcal{F}\kappa)(m)}\sum_{j = 0}^m \phi^2\Big(\frac{(j + 1)\pi}{m+1}\Big)\Big(\frac{\pi(j + 1)}{m + 1}\Big)^{-2\alpha_{\mathcal{F}(\kappa)}} \\
    & = \frac{1}{(\mathcal{F}\kappa)(m)}\Big(\int_{0}^{\pi} \phi^2(t)t^{-2\alpha_{\mathcal{F}(\kappa)}}dt + o(1)\Big)
\end{align*}
Substituting the last line into \eqref{eq: Apply Pseudomodes}, we obtain:
\begin{equation*}
    \mu_{m+1} \preceq (\mathcal{F}\kappa)(m)(m + 1)^{-2k}
\end{equation*}
Now observing that for the specified measure $d\nu \propto |1 - z|^{2k}$, $s(n) \asymp n^{2k + 1}$ by Theorem 4.3 in \cite{mastroianni2000weighted}, we obtain our result. Verifying that there exists a $\psi$ satisfying Assumptions \ref{Opt Smoothness} and \ref{Growth Conditions} is straightforward: this follows immediately from the fact that $\frac{t(\mathcal{F}\kappa)(t)}{s(t)} \preceq \frac{1}{s(t)} < 1$ as $t \to \infty$ (as $t(\mathcal{F}\kappa)(t) \to 0$ by $\mathcal{H}_K \hookrightarrow L^{\infty}(\mathcal{X})$) --- hence we may choose $\psi$ such that $\psi^{-1}(t) \asymp ts^{-1}(t^{-1})(\mathcal{F}\kappa)\big(s^{-1}(t^{-1})\big)$ as $t \to 0$. The concavity of $\frac{t}{\psi(t)}$ follows immediately from the fact that both $\psi^{-1}(t)$ (and hence $\psi$) and $\frac{\psi^{-1}(t)}{t}$ (and hence $\frac{\psi(t)}{t}$) are nondecreasing. To verify that $H^{\psi} \stackrel{w}{\hookrightarrow} L^{\infty}(\mathcal{X})$, we first, abusing notation, denote by $s(t, x)$ the measure on $(0, \infty)$ given by $s(t, x) = \sum_{j \leq t} e^2_j(x)$ for $t \geq 1$ and $s(t, x) = 0$ for $t \in [0, 1)$. Then, observe that for any for any $\epsilon > 0$, we may write $f \in H^{\psi^{1 + \epsilon}}$ as $f = \sum_{i = 1}^{\infty} a_i \psi^{\frac{1+\epsilon}{2}}(\mu_i)e_i$ for $\{a_i\} \in \ell^2$. Since, as we have just shown, $\mu_i \asymp \frac{i(\mathcal{F}\kappa)(i)}{s(i)}$, we have that:
\begin{align}
    f(x) & = \sum_{i = 1}^{\infty} a_i \psi^{\frac{1+\epsilon}{2}}(\mu_i)e_i(x) \nonumber  \\
    & \asymp \sum_{i = 1}^{\infty} a_i \Big(\frac{1}{s(i)}\Big)^{\frac{1 + \epsilon}{2}} e_i(x)  \label{eq: Apply Opt Smoothness Ass} \\
    & \leq \Big(\sum_{i = 1}^{\infty} a^2_i\Big) \Big(\sum_{i = 1}^{\infty} \frac{1}{s(i)^{1 + \epsilon}}e^2_i(x)\Big) \nonumber \\
    & = ||f||^2_{\mathcal{H}^{\psi^{1 + \epsilon}}} \int_{1}^{\infty} \frac{1}{s(i)^{1 + \epsilon}} ds(i, x) \label{eq: Def of Sup} \\
    & \preceq ||f||^2_{\mathcal{H}^{\psi^{1 + \epsilon}}} \int_{1}^{\infty} \frac{1}{s(i)^{1 + \epsilon}} ds(i) \nonumber \\
    & \leq C||f||^2_{\mathcal{H}^{\psi^{1 + \epsilon}}}
\end{align}
where \eqref{eq: Apply Opt Smoothness Ass} follows from Assumption \ref{Opt Smoothness} and \eqref{eq: Def of Sup} follows from $ds(n, x) \preceq ds(n)$ by definition. Taking the supremum of the LHS, we verify the weak embedding $H^{\psi} \stackrel{w}{\hookrightarrow} L^{\infty}(\mathcal{X})$. 
\end{proof}
\begin{proof}[Proof of Lemma \ref{Mercer Differencing}]
We recall from Lemma \ref{Optimal Smoothness Example} that $s(n) \asymp n^{2k+1}$ for $L^2(\nu)$. Let $f = \sum_{m \in \mathbb{Z}} f_m z^m \in L^2(\nu)$ and set $s = \frac{1}{2\beta} + \frac{k(1 - \beta)}{\beta}$. Analogously to Lemma \ref{Pseudomodes}, for $g \in L^2(\mathbb{R})$, we define the \textit{discrete Laplacian} operator as:
\begin{equation*}
    \Delta_h[g](x) = 2g(x) - g(x + h) - g(x-h)
\end{equation*}
for $x \in \mathbb{R}$ and $h > 0$ (with a similar abbreviation of $\Delta[g](x)$ when $h = 1$). Observe that for $w(\theta) = |1 - e^{i\theta}|^{2k}$ and $g \in L^1(\partial \mathbb{D}) \cap L^2(\partial \mathbb{D})$, $\mathcal{F}[wg] = \Delta^k[\mathcal{F}g]$. It is easy to see that for $s > 0$, $H^s(\partial \mathbb{D}) \subset L^2(\partial \mathbb{D}) \subset L^2(\nu)$ (see e.g. Theorem 4.13 in \cite{bottcher2005spectral}). We then obtain that for any $g \equiv \sum_{m \in \mathbb{Z}} g_m z^m \in H^s(\mathbb{T})$:
\begin{align}
    \langle g, T_{\nu}f \rangle_{H^s(\partial\mathbb{D})} & = \sum_{m \in \mathbb{Z}} (1 + m^2)^s g_m (T_{\nu}f)_m \nonumber  \\
    & \asymp  \sum_{m \in \mathbb{Z}} m^{2s} g_m (T_{\nu}f)_m \nonumber \\
    & =  \sum_{m \in \mathbb{Z}} m^{\frac{1}{\beta} + \frac{2k(1 - \beta)}{\beta}} g_m (T_{\nu}f)_m \nonumber \\
    & = \sum_{m \in \mathbb{Z}} \frac{m g_m (T_{\nu}f)_m}{\psi^{-1}(m^{-2k-1})m^{2k + 1}} \nonumber \\
    & \asymp \sum_{m \in \mathbb{Z}} \frac{m g_m [\mathcal{F}\kappa](m) (\Delta^{k}[\mathcal{F}f])(m)}{\psi^{-1}(m^{-2k-1})m^{2k + 1}} \label{eq: Fourier Conv}  \\
    & \asymp \sum_{m \in \mathbb{Z}} g_m (\Delta^{k}[\mathcal{F}f])(m) \label{eq: Apply Opt Smoothness}  \\
    & = \langle g, f \rangle_{L^2(\nu)} \nonumber
\end{align}
where \eqref{eq: Fourier Conv} follows from the Fourier convolution formula and in \eqref{eq: Apply Opt Smoothness} we have applied Assumption \ref{Opt Smoothness} (noting $s(n) \asymp n^{2k+1}$ as in Lemma \ref{Optimal Smoothness Example}). Taking the supremum of both sides over all $g \in \mathcal{B}(H^s(\partial\mathbb{D}))$, we obtain our result.  
\end{proof}

\begin{proof}[Proof of Lemma \ref{Isoperimetric Equivalance}]
We first begin with some preliminary calculations of the relative $\mathcal{H}_K$-capacity.  We first note that since our cutoff functions considered here will have compact support in $\mathcal{X}$ and vanish at the boundary we can simply work with $\mathcal{H}_K$ norms as opposed to $\mathcal{H}_K(\mathcal{X})$ by the discussion after Definition \ref{Capacity Definition} (and hence will omit this specification here notationally). By Lemma 2.1.1 in \cite{fukushima2010dirichlet}, we can reformulate $\text{cap}_{\mathcal{X}}(E, F; \mathcal{H}_K)$ in Definition \ref{Capacity Definition} as:
\begin{equation}
\label{Reformulated Capacity}
    \text{cap}_{\mathcal{X}}(E, F; \mathcal{H}_K) = \inf\{||f||^2_{K}: f \in \mathcal{H}_K; f = 1 \hspace{1mm} \text{on} \hspace{1mm} E; f = 0 \hspace{1mm} \text{on} \hspace{1mm} \mathcal{X} \setminus  \text{int}(F); 0 \leq f \leq 1\}
\end{equation}
Let $\text{Adm}_{\mathcal{X}}(E, F; \mathcal{H}_K) = \{f \in \mathcal{H}_K: f = 1 \hspace{1mm} \text{on} \hspace{1mm} E; f = 0 \hspace{1mm} \text{on} \hspace{1mm} \mathcal{X} \setminus \text{int}(F); 0 \leq f \leq 1\}$ denote the set of admissible functions for $\text{cap}_{\mathcal{X}}(E, F; \mathcal{H}_K)$.  Since $\mathcal{X}$ contains no boundary points by assumption, it is easy to see that the sets $\text{Adm}_{\mathcal{X}}(B(x, \frac{r}{2}), B(x, r); \mathcal{H}_K)$ are radial for all $r < r_{x}$; i.e. there exists a $r_x > 0$, such that for all $0 < r < r_x$:
\begin{equation*}
    \text{Adm}_{\mathcal{X}}\Big(B\Big(x, \frac{r}{2}\Big), B(x, r); \mathcal{H}_K\Big) = \Big\{\tilde{f}(y) = f\Big(x + \frac{r_x(y - x)}{r}\Big): f \in \text{Adm}_{\mathcal{X}}\Big(B\Big(x, \frac{r_x}{2}\Big), B(x, r_x); \mathcal{H}_K\Big)\Big\}
\end{equation*}
We may choose $r_x \in \Big(0, \frac{1}{2\sqrt[d]{C_d}}\Big)$, where $C_d$ is the volume of the unit ball in $\mathbb{R}^d$. Let $\tilde{f} \in \text{Adm}_{\mathcal{X}}\Big(B\Big(x, \frac{r}{2}\Big), B(x, r); \mathcal{H}_K\Big)$. Then, $\mathcal{F}\tilde{f}(\xi) = (rr^{-1}_x)^d e^{-i\xi(1 - rr_{x}^{-1})x}\mathcal{F}f(rr_{x}^{-1}\xi)$ for some $f \in \text{Adm}_{\mathcal{X}}\Big(B\Big(x, \frac{r_x}{2}\Big), B(x, r_x); \mathcal{H}_K\Big)$. Hence, we have as $r \to 0$:
\begin{align}
    ||\tilde{f}||^2_{K} & = \int_{\mathbb{R}^d} \frac{|\mathcal{F}\tilde{f}(\xi)|^2}{\mathcal{F}_d \kappa(||\xi||)}d\xi \nonumber \\
    & = (rr^{-1}_x)^{2d}\int_{\mathbb{R}^d} \frac{|\mathcal{F}f(rr_{x}^{-1}\xi)|^2}{\mathcal{F}_d \kappa(||\xi||)}d\xi \nonumber \\
    &  = (rr^{-1}_{x})^d \int_{\mathbb{R}^d} \frac{|\mathcal{F}f(\xi)|^2}{\mathcal{F}_d \kappa(r_{x}r^{-1}||\xi||)}d\xi \label{eq: Growth-Plugins}
\end{align}
Now, observe that, uniformly in $t \geq 1$, we have:
\begin{align}
    \lim_{||\xi|| \to \infty} \frac{1}{\log ||\xi||} \cdot \log \Big(\frac{\mathcal{F}_d \kappa(t)}{\mathcal{F}_d \kappa(t||\xi||)}\Big) & = -\beta_{F_{d}\kappa}  \label{eq: Asymptotic Growth} \\
     \lim_{||\xi|| \to 0} \frac{1}{\log ||\xi||} \cdot \log \Big(\frac{\mathcal{F}_d \kappa(t)}{\mathcal{F}_d \kappa(t||\xi||)}\Big) & \leq \alpha_{\frac{1}{F_{d}\kappa}} \label{eq: Asymptotic Decay}
\end{align}
which follows from the assumption that the extension indices of $\mathcal{F}_d \kappa$ and $\frac{1}{F_{d}\kappa}$ (see \eqref{eq: Lower Extension Index} and \eqref{eq: Upper Extension Index}) are finite (namely $\infty < \alpha_{\mathcal{F}_d \kappa}, \beta_{\mathcal{F}_d \kappa} \leq 0 \leq \alpha_{\frac{1}{F_{d}\kappa}}$) and $\beta_{\mathcal{F}_d \kappa} = -\beta_{\frac{1}{\mathcal{F}_d \kappa}}$ (note the extension indices of $\mathcal{F}_d \kappa$ are nonpositive as $\mathcal{F}_d \kappa$ is nonincreasing). Then, choosing $t = r_{x}r^{-1}$ and $t = 1$ in \eqref{eq: Asymptotic Growth}, and $||\xi|| = r_{x}$ in \eqref{eq: Asymptotic Decay} we have that for any $\epsilon > 0$ (which we will choose later), there exists a $s_{1}(\epsilon) \in (0, 1)$ and $s_2(\epsilon) \in (1, \infty)$ depending only on $\epsilon$ such that:
\begin{align}
   ||\xi||^{-2\epsilon} \leq  \frac{\mathcal{F}_d \kappa(||\xi||)}{\mathcal{F}_d \kappa(r_{x}r^{-1}||\xi||)} \frac{\mathcal{F}_d \kappa(r_{x}r^{-1})}{\mathcal{F}_d \kappa(1)} & \leq ||\xi||^{2\epsilon} \hspace{3mm} \forall \hspace{1mm} ||\xi|| \in [s_{2}(\epsilon), \infty) \label{eq: Epsilon Poly Growth} \\
    \frac{\mathcal{F}_d \kappa(r^{-1})}{\mathcal{F}_d \kappa(r_{x}r^{-1})}  & \geq r_x^{2\alpha_{(F_{d}\kappa)^{-1}}} \hspace{2mm} \forall r \in (0, 1) \label{eq: Crude Lower Bound}
\end{align}
when $r_x \in (0, s_{1}(\epsilon))$ (note that in fact we can always choose $r_{x} < \frac{1}{2}$ sufficiently small so that \eqref{eq: Crude Lower Bound} holds). Substituting \eqref{eq: Epsilon Poly Growth} and \eqref{eq: Crude Lower Bound} into \eqref{eq: Growth-Plugins}, we have by the uniformity of \eqref{eq: Asymptotic Growth}-\eqref{eq: Asymptotic Decay} and the compactness of $[1, s_2(\epsilon)]$ that there exists constants $C_1, C_2 > 0$ independent of $r > 0$ (but dependent on $r_x$) such that:
\begin{align}
    ||\tilde{f}||^2_{K} & \geq C_1 r^d\Big(\int_{B(0, 1)} |\mathcal{F}f(\xi)|^2 d\xi + \frac{1}{\mathcal{F}_d \kappa(r^{-1})} \int_{\mathbb{R}^d \setminus B(0, 1)} \frac{||\xi||^{-2\epsilon}|\mathcal{F}f(\xi)|^2}{\mathcal{F}_d \kappa(||\xi||)} d\xi \Big) \label{eq: Asymptotic Capacity Lower Bound}\\
    ||\tilde{f}||^2_{K} & \leq \frac{C_2 r^d}{\mathcal{F}_d \kappa(r^{-1})} \Big(\int_{B(0, 1)}  |\mathcal{F}f(\xi)|^2 d\xi + \int_{\mathbb{R}^d \setminus B(0, 1)} \frac{||\xi||^{2\epsilon}|\mathcal{F}f(\xi)|^2}{\mathcal{F}_d \kappa(||\xi||)} d\xi \Big) \label{eq: Asymptotic Capacity Upper Bound} 
\end{align}
Let:
\begin{align*}
    A_1 & \equiv \inf_{f \in \text{Adm}_{\mathcal{X}}(B(x, \frac{r_{x}}{2}), B(x, r_{x}); \mathcal{H}_K)} \int_{\mathbb{R}^d \setminus B(0, 1)} \frac{||\xi||^{-2\epsilon}|\mathcal{F}f(\xi)|^2}{\mathcal{F}_d \kappa(||\xi||)} d\xi \ \\
    A_2 & \equiv \inf_{f \in \text{Adm}_{\mathcal{X}}(B(x, \frac{r_{x}}{2}), B(x, r_{x}); \mathcal{H}_K)} \int_{\mathbb{R}^d \setminus B(0, 1)} \frac{||\xi||^{2\epsilon}|\mathcal{F}f(\xi)|^2}{\mathcal{F}_d \kappa(||\xi||)} d\xi 
\end{align*}
We will demonstrate that both $A_1, A_2 \in (0, \infty)$ with an appropriate choice of $\epsilon > 0$. Observe that it suffices to prove $A_1 > 0$ and $A_2 < \infty$. For the first statement, we first note that $\epsilon > 0$ can always be chosen so that $\frac{t^{-2\epsilon}}{\mathcal{F}_d \kappa(t)}$ is asymptotically nondecreasing, i.e. with the choice $\epsilon < \epsilon_1 \equiv -\frac{1}{2} \lim_{t \to \infty} \log \mathcal{F}_d \kappa(t)$. Hence, we have that:
\begin{align}
    \int_{\mathbb{R}^d \setminus B(0, 1)} \frac{||\xi||^{-2\epsilon}|\mathcal{F}f(\xi)|^2}{\mathcal{F}_d \kappa(||\xi||)} d\xi & \succeq  \frac{1}{\mathcal{F}_d \kappa(1)} \int_{\mathbb{R}^d \setminus B(0, 1)} |\mathcal{F}f(\xi)|^2 d\xi \nonumber \\
    &  = \frac{1}{\mathcal{F}_d \kappa(1)}\Big(||f||^2_{L^2(\mathbb{R}^d)} - \int_{B(0, 1)} |\mathcal{F}f(\xi)|^2 d\xi\Big) \nonumber \\
    & \geq \frac{1}{\mathcal{F}_d \kappa(1)}\Big(||f||^2_{L^2(\mathbb{R}^d)} -  C_d||f||^2_{L^1(\mathbb{R}^d)}\Big) \label{eq: Upper Fourier Bound} \\
    & \geq \frac{1}{\mathcal{F}_d \kappa(1)}(2^{-d}r^d_x - C_d r^{2d}_x) \label{eq: Magnitude of Potential}
\end{align}
where $\eqref{eq: Upper Fourier Bound}$ follows from the bound $||\mathcal{F}f||_{L^{\infty}(\mathbb{R}^d)} \leq ||f||_{L^{1}(\mathbb{R}^d)}$ (recalling $C_d$ is the volume of the unit ball) and $\eqref{eq: Magnitude of Potential}$ follows from the fact that by definition, $f \in \text{Adm}_{\mathcal{X}}(B(x, \frac{r_{x}}{2}), B(x, r_{x}); \mathcal{H}_K)$ implies $f = 1$ on $B(x, \frac{r_{x}}{2})$ (and hence $||f||^2_{2} \geq 2^{-d}r^d_x$) while $0 \leq f \leq 1$ on $B(x, r_{x})$ and hence $||f||_{L^1(\mathbb{R}^d)} \leq r^d_x$. Finally $2^{-d}r^d_x - C_dr^{2d}_x > 0$ by the assumption $r_x \in \Big(0, \frac{1}{2\sqrt[d]{C_d}}\Big)$. The conclusion $A_1 > 0$ follows. 

Now, let $f^{*} \in \text{Adm}_{\mathcal{X}}(B(x, \frac{r_{x}}{2}), B(x, r_x); \mathcal{H}_K)$ be the unique potential which achieves $\text{Cap}_{\mathcal{X}}(B(x, \frac{r_{x}}{2}), B(x, r_x); \mathcal{H}_K)$ (the existence of $f^{*} \in \mathcal{H}_K$ follows from Lemma 2.1.1 in \cite{fukushima2010dirichlet}). Then, we can always choose $\epsilon > 0$, such that $$\epsilon < \epsilon_2 \equiv \frac{1}{2}\lim_{||\xi|| \to \infty}\log \mathcal{F}_d \kappa(||\xi||) - \log |\mathcal{F}f^{*}(\xi)|^2  - d$$ so that the integral in $A_2$ converges with $f = f^{*}$, and therefore $A_2 < \infty$ ($A_2 > 0$ is trivial as $A_2 \geq A_1$). Hence, choosing $\epsilon < \min\{\epsilon_1, \epsilon_2\}$ and taking the infimum over $f \in \text{Adm}_{\mathcal{X}}(B(x, \frac{r_{x}}{2}), B(x, r_{x}); \mathcal{H}_K)$ in \eqref{eq: Asymptotic Capacity Lower Bound}-\eqref{eq: Asymptotic Capacity Upper Bound}, we have:
\begin{equation}
\label{Capacity Asymptotics}
    \text{Cap}_{\mathcal{X}}\Big(B(x, \frac{r}{2}), B(x, r); \mathcal{H}_K\Big) \asymp \frac{r^d}{\mathcal{F}_d \kappa(r^{-1})}
\end{equation}
as $r \to 0$. We now note that it is sufficient to prove \eqref{eq: Weak Embedding Implication} for sufficiently small $\epsilon > 0$; the statement then follows for all $\epsilon > 0$ since $\frac{t\mathcal{F}_d \kappa(t^{\frac{1}{d}})}{s(t)}$ is nonincreasing. Now, by Lemma \ref{Interpolation Inequality}, we have:
\begin{equation*}
    \frac{||f||^2_{\psi^{1 + \epsilon}}}{||f||^2_{K}} \leq \frac{t}{\psi^{1 + \epsilon}}\Big(\frac{||f||^2_{L^2(\nu)}}{||f||^2_{K}}\Big)
\end{equation*}
for sufficiently small $\epsilon > 0$. Since $\mathcal{H}_{\psi^{1 + \epsilon}} \hookrightarrow L^{\infty}(\mathcal{X})$, this becomes:
\begin{equation}
\label{Interpolation Restate}
    ||f||_{\infty} \preceq ||f||^2_{K} \cdot \frac{t}{\psi^{1 + \epsilon}}\Big(\frac{||f||^2_{L^2(\nu)}}{||f||^2_{K}}\Big)
\end{equation}
Since for sufficiently small $\epsilon > 0$, $\frac{t}{\psi^{1 + \epsilon}}$ is concave and nondecreasing, it follows that the right-hand side of \eqref{Interpolation Restate} is jointly nondecreasing in $||f||^2_{K}$ and $||f||^2_{L^2(\nu)}$ (the latter is obvious; to see the former  write $y = \frac{||f||^2_{L^2(\nu)}}{||f||^2_{K}}$, $\tilde{\psi}(t) = \frac{t}{\psi^{1 + \epsilon}(t)}$, and observe the RHS of  \eqref{Interpolation Restate} can be expressed as $||f||^2_{L^2(\nu)} \cdot \frac{\tilde{\psi}(y)}{y}$ --- the statement follows from $\frac{\tilde{\psi}(y)}{y}$ nonincreasing as $\tilde{\psi}(t) = \frac{t}{\psi^{1 + \epsilon}(t)}$ is concave  by assumption). Let $f^{*} \in \text{Adm}_{\mathcal{X}}(B(x, \frac{r}{2}), B(x, r); \mathcal{H}_K)$ be the potential at which the infimum in \eqref{Reformulated Capacity} is realized (the existence of this unique minimizer follows again from Lemma 2.1.1 in \cite{fukushima2010dirichlet}). Since $f^{*} \in \text{Adm}_{\mathcal{X}}(B(x, \frac{r}{2}), B(x, r); \mathcal{H}_K)$, it follows by definition that:
\begin{equation*}
    ||f||^2_{L^2(\nu)} \leq \nu(B(x, r))
\end{equation*}
Moreover, noting that $||f^{*}||_{\infty} = 1$ and substituting \eqref{Capacity Asymptotics} into \eqref{Interpolation Restate}, we have:
\begin{equation}
\label{Capacity Interpolation}
    1 \preceq \frac{r^d}{\mathcal{F}_d \kappa(r^{-1})} \cdot \frac{t}{\psi^{1 + \epsilon}}\Big(\frac{\mathcal{F}_d \kappa(r^{-1})\nu(B(x, r))}{r^d}\Big)
\end{equation}
as $r \to 0$. Now we note that by the volume doubling property of $\nu$ and the $\text{RCD}^{*}(K, N)$ condition on $(\mathcal{X}, \nu)$, we have by Corollary 3.6 of \cite{ambrosio2018short} that:
\begin{equation}
\label{Heat Kernel Decay}
    p_{\mathcal{X}}(t, x, x) \asymp \frac{1}{\nu(B(x, \sqrt{t}))} \hspace{2mm} \text{as} \hspace{1mm} t \to 0
\end{equation}
where $p_{\mathcal{X}}(t, x, y)$ is the heat kernel on $(\mathcal{X}, \nu)$. Let $\tilde{s}(t, x)$ denote the measure given by:
$\tilde{s}(t, x) = \sum_{\lambda_{i} \leq t} e_i^2(x)$, where $\{\lambda_{i}\}_{i \in \mathbb{N}}$ are the eigenvalues of the weighted Laplacian $\Delta_{\nu}$ on $(\mathcal{X}, \nu)$ with corresponding eigenfunctions $\{e_i\}_{i \in \mathbb{N}}$. Now, observing that:
\begin{equation*}
    p_{\mathcal{X}}(t, x, x) = \int_{0}^{\infty} e^{-ty}d\tilde{s}(y, x)
\end{equation*}
we have by the Karamata-Tauberian theorem (e.g Theorem 2.5 in \cite{ambrosio2018short}) and \ref{Heat Kernel Decay} that:
\begin{equation}
\label{eq: Eigenmeasure Growth}
    \tilde{s}(t, x) \asymp \frac{1}{\nu(B(x, t^{-\frac{1}{2}}))}
\end{equation}
as $t \to \infty$. Now, in light of \eqref{eq: Average Volume Growth} and Theorem 4.3 in \cite{ambrosio2018short}, we have that the weighted Laplacian $\Delta_{\nu}$ on $(\mathcal{X}, \nu)$ follows Weyl-type asymptotics and $N(\lambda) = \#\{i: \lambda_i \leq \lambda\} \asymp \lambda^{\frac{d}{2}}$ as $\lambda \to \infty$. Hence, from \eqref{eq: Eigenmeasure Growth}, we have that 
\begin{align}
  s(t, x) & = \sum_{i \leq t} e^2_i(x) \nonumber \\
  & \asymp \sum_{\lambda_i \leq t^{\frac{2}{d}}} e^2_i(x) \nonumber \\ 
  & = \tilde{s}(t^{\frac{2}{d}}, x) \nonumber \\
  & \asymp \frac{1}{\nu(B(x, t^{-\frac{1}{d}}))} \label{eq: Weyl Eigenfunction Growth}  
\end{align}
Now substituting the latter and $r = t^{-\frac{1}{d}}$ (as $t \to \infty$) in \eqref{Capacity Interpolation}, we obtain:
\begin{align*}
        1 & \preceq \frac{t^{-1}}{\mathcal{F}_d \kappa(t^{\frac{1}{d}})} \cdot \frac{t}{\psi^{1 + \epsilon}}\Big(\frac{\mathcal{F}_d \kappa(t^{\frac{1}{d}})t}{s(t, x)}\Big) \\
        & \asymp \frac{1}{s(t, x)\psi^{1 + \epsilon}\Big(\frac{\mathcal{F}_d \kappa(t^{\frac{1}{d}})t}{s(t, x)}\Big)}
\end{align*}
and \eqref{eq: Weak Embedding Implication} follows by taking the supremum in $x \in \mathcal{X}$ of the right-hand side and noting that $t\psi^{1 + \epsilon}(t^{-1})$ is nonincreasing in $t > 0$ by Assumption \ref{Growth Conditions} for sufficiently small $\epsilon > 0$. \eqref{eq: Isocapacitary Equivalence} also follows directly by substituting \eqref{Capacity Asymptotics} and \eqref{eq: Weyl Eigenfunction Growth} into Assumption \ref{Opt Smoothness} again with the choice of $r = t^{-\frac{1}{d}}$ and taking the supremum.
\end{proof}

\section{Proof of Theorem \ref{Main Upper}}
\label{Upper Bound Proof}
The main effort in the proof of Theorem \ref{Main Upper} involves proving Lemma \ref{Bias-Variance Decomposition}, from which the Theorem \ref{Main Upper} easily follows; hence we will begin with the proof of Lemma \ref{Bias-Variance Decomposition}. We start with the following bound on the bias:
\begin{lemma}
\label{Bias Bound}
\begin{equation*}
    ||f^{*} - f_{\lambda}||^2_{L^2(\nu)} \leq \phi(\lambda)||f^{*}||^2_{\phi}
\end{equation*}
\end{lemma}
\begin{proof}
Since $f \in \mathcal{H}^{\phi}$, we have that there exists a $\{a_i\}_{i = 1}^{\infty} \in \ell^2$ such that:
\begin{equation*}
    f^{*} = \sum_{i = 1}^{\infty} a_i \phi^{\frac{1}{2}}(\mu_i)e_i
\end{equation*}
Hence, 
\begin{align}
\label{eq:Error}
    f^{*} - f_{\lambda} = \sum_{i = 1}^{\infty} \frac{a_i \lambda \phi^{\frac{1}{2}}(\mu_i)}{\mu_i + \lambda} e_i
\end{align}
Therefore, by the Cauchy-Schwartz inequality:
\begin{align*}
    ||f^{*} - f_{\lambda}||^2_{L^2(\nu)} & \leq \sum_{i = 1}^{\infty} \Big(\frac{a_i \lambda \phi^{\frac{1}{2}}(\mu_i)}{\mu_i + \lambda}\Big)^2  \\
    & \leq \Big(\sup_{i}  \frac{\lambda^2 \phi(\mu_i)}{(\mu_i + \lambda)^2}\Big)\sum_{i = 1}^{\infty} a^2_i \\
    & \leq \Big(\sup_{i}  \frac{\lambda^2 \phi(\mu_i)}{(\mu_i + \lambda)^2}\Big)||f^{*}||^2_{\phi}
\end{align*}
Now, observe that for $\mu_i < \lambda$:
\begin{equation*}
    \sup_{i: \mu_i < \lambda}  \frac{\lambda^2 \phi(\mu_i)}{(\mu_i + \lambda)^2} \leq \phi(\lambda)
\end{equation*}
as $\phi$ is nondecreasing. For $\mu_i \geq \lambda$, we have that:
\begin{equation*}
    \sup_{i: \mu_i \geq \lambda}  \frac{\lambda^2 \phi(\mu_i)}{(\mu_i + \lambda)^2} = \sup_{i: \mu_i \geq \lambda}  \frac{\lambda^2 \mu_i^2 \phi(\mu_i)}{\mu^2_i(\mu_i + \lambda)^2} \leq \sup_{i: \mu_i \geq \lambda} \frac{\lambda^2 \mu_i^2 \phi(\lambda)}{\lambda^2(\mu_i + \lambda)^2} \leq \phi(\lambda)
\end{equation*}
where the second inequality follows from the fact that $\frac{\phi(t)}{t^2}$ is nonincreasing (from Assumption \ref{Growth Conditions}). Putting this together, we obtain our result. 
\end{proof}
\begin{lemma}
\label{Uniform Bias}
For sufficiently small $\epsilon > 0$:
\begin{equation*}
    ||f^{*} - f_{\lambda}||^2_{\infty} \leq \frac{\phi(\lambda)||k^{\psi^{1 + \epsilon}}||^2_{\infty}||f^{*}||^2_{\phi}}{\psi^{1+\epsilon}(\lambda)}
\end{equation*}
\end{lemma}
\begin{proof}
Let $\epsilon > 0$. We have from \eqref{eq:Error}:
\begin{align*}
    ||f^{*} - f_{\lambda}||^2_{\infty} & \leq \Big|\sum_{i = 1}^{\infty} \frac{a_i\lambda \phi(\mu_i)^{\frac{1}{2}} e_i(\cdot)}{\mu_i + \lambda} \Big|^2_{\infty} \\
    & \leq \Big(\sum_{i} \frac{a^2_i\lambda^2 \phi(\mu_i)}{(\mu_i + \lambda)^2\psi^{1 + \epsilon}(\mu_i)}\Big)\Big|\sum_{i = 1}^{\infty} \psi^{1 + \epsilon}(\mu_i) e^2_i(\cdot)\Big|_{\infty} \\
    & \leq \frac{\phi(\lambda)||k^{\psi^{1 + \epsilon}}||^2_{\infty}||f^{*}||^2_{\phi}}{\psi^{1 + \epsilon}(\lambda)}
\end{align*}
where the last line follows from the definition of $||k^{\psi^{1 + \epsilon}}||$ and the same logic as in the proof of Lemma \ref{Bias Bound}, after noting that $\frac{\phi(t)}{t^2\psi^{1 + \epsilon}(t)}$ is decreasing for $t \geq \lambda$ and $\frac{\phi(t)}{\psi^{1+ \epsilon}(t)}$ is increasing for $t < \lambda$ for sufficiently small $\epsilon > 0$ by Assumption \ref{Growth Conditions}. 
\end{proof}

We now derive a bound on the variance of our estimator $f_{\lambda} - f_{D, \lambda}$: 
\begin{lemma}
\label{Variance Bound}
For all sufficiently small $\epsilon > 0$
\begin{equation*}
    ||f_{\lambda} - f_{D, \lambda}||_{L^2(\nu)} \preceq \log(\delta^{-1})\sqrt{\frac{1152\sigma^2||k^{\psi^{1 + \epsilon}}||^2_{\infty}||f^{*}||^2_{\phi}}{n\psi^{1 + \epsilon}(\lambda)}\Big(\phi(\lambda) + \frac{1}{n}\Big) + \frac{s^{-1}(\psi(\lambda)^{-1 - \epsilon})}{n}}
\end{equation*}
with probability $1 - 3\delta$
\end{lemma}
\begin{proof}
We begin with a standard decomposition applied in the proof of Theorem 16 in \cite{fischer2020sobolev} and Theorem 7 in \cite{talwai2022sobolev}. We first observe that:
\begin{equation}
\label{eq: Basic Decomposition}
    f_{\lambda} - f_{D, \lambda} = f_{\lambda} - (C_{D} + \lambda)^{-1}g_{D} = (C_{D} + \lambda)^{-1}((C_D + \lambda)f_{\lambda} - g_{D})
\end{equation}
where $g_D = \sum_{i = 1}^n y_i k(x_i, \cdot) \in \mathcal{H}_K$. Hence, we have that:
\begin{align}
    ||f_{\lambda} - f_{D, \lambda}||_{L^2(\nu)} & = ||C^{\frac{1}{2}}_{\nu}(f_{\lambda} - f_{D, \lambda})||_{K}  \label{eq: Norm Conversion}\\
    & \leq ||C^{\frac{1}{2}}_{\nu}(C_{\nu} + \lambda)^{-\frac{1}{2}}|| ||(C_{\nu} + \lambda)^{\frac{1}{2}}(C_D + \lambda)^{-1}(C_{\nu} + \lambda)^{\frac{1}{2}}||||(C_{\nu} + \lambda)^{-\frac{1}{2}}((C_D + \lambda)f_{\lambda} - g_{D})|| \label{eq: Variance Breakdown}
\end{align}
where \eqref{eq: Norm Conversion} follows from the fact that for any $f \in \mathcal{H}_K$, $\langle I_{\nu}f, I_{\nu}f \rangle_{2} = \langle f, I^{*}_{\nu}I_{\nu}f \rangle_{K} = \langle f, C_{\nu} f \rangle_K$ and \eqref{eq: Adjoint Formula}, and \eqref{eq: Variance Breakdown} follows from \eqref{eq: Basic Decomposition} and algebraic manipulation. It is easy to see that:
\begin{equation*}
    ||C^{\frac{1}{2}}_{\nu}(C_{\nu} + \lambda)^{-\frac{1}{2}}|| \leq 1
\end{equation*}
We now analyze second factor in \eqref{eq: Variance Breakdown}. Indeed, we have that:
\begin{equation*}
    (C_{\nu} + \lambda)^{\frac{1}{2}}(C_D + \lambda)^{-1}(C_{\nu} + \lambda)^{\frac{1}{2}} = \Big(I - (C_{\nu} + \lambda)^{-\frac{1}{2}}(C_{\nu} - C_D)(C_{\nu} + \lambda)^{-\frac{1}{2}}\Big)^{-1}
\end{equation*}
by the same algebraic manipulation as in the proof of Theorem 16 in \cite{fischer2020sobolev} (for the sake of space, we avoid repeating the argument verbatim here). Hence, estimating the middle factor in \eqref{eq: Variance Breakdown} boils down to estimating the concentration of $(C_{\nu} + \lambda)^{-\frac{1}{2}}(C_{\nu} - C_D)(C_{\nu} + \lambda)^{-\frac{1}{2}}$. We first note that:
\begin{equation*}
    (C_{\nu} + \lambda)^{-\frac{1}{2}}(C_{\nu} - C_D)(C_{\nu} + \lambda)^{-\frac{1}{2}} = \mathbb{E}_{X}[h(X, \cdot) \otimes h(X, \cdot)] - \mathbb{E}_{\mathcal{D}}[h(X, \cdot) \otimes h(X, \cdot)]
\end{equation*}
where: 
\begin{equation}
\label{Define h}
    h(x, \cdot) \equiv (C_{\nu} + \lambda)^{-\frac{1}{2}}k(x, \cdot)
\end{equation}
Note, by Lemma \ref{h-norm Bound}, we have that $||h(x, \cdot)||_{K} \leq \sqrt{\frac{||k^{\psi^{1 + \epsilon}}||_{\infty}^2}{\psi^{1 + \epsilon}(\lambda)}}$ for all sufficiently small $\epsilon > 0$. Then, analogously to Lemma 17 and p.26 in \cite{fischer2020sobolev}, we may demonstrate that:
\begin{equation*}
    \|(C_{\nu} + \lambda)^{-\frac{1}{2}}(C_{\nu} - C_D)(C_{\nu} + \lambda)^{-\frac{1}{2}}\| \leq \frac{2}{3}
\end{equation*}
with probability $1 - 2\delta$, when $n \geq 8\log(\delta^{-1})||k^{\psi^{1 + \epsilon}}||^2_{\infty}g_{\lambda}\psi(\lambda)^{-1 - \epsilon}$, with $g_{\lambda} = \log\Big(2e\mathcal{N}(\lambda)\Big(1 + \frac{\lambda}{||C_{\nu}||}\Big)\Big)$, where $\mathcal{N}(\lambda) = \text{tr}(C_{\nu}(C_{\nu} + \lambda)^{-1})$. Applying a Neumann series expansion, we obtain (like in \cite{fischer2020sobolev}):
\begin{equation*}
    \|(C_{\nu} + \lambda)^{\frac{1}{2}}(C_D + \lambda)^{-1}(C_{\nu} + \lambda)^{\frac{1}{2}}\| = ||\Big(I - (C_{\nu} + \lambda)^{-\frac{1}{2}}(C_{\nu} - C_D)(C_{\nu} + \lambda)^{-\frac{1}{2}}\Big)^{-1}|| \leq 3
\end{equation*}
for $n \geq 8\log(\delta^{-1})||k^{\psi^{1 + \epsilon}}||^2_{\infty}g_{\lambda}\psi(\lambda)^{-1 - \epsilon}$. Now, it remains to bound the third factor in \eqref{eq: Variance Breakdown}. We observe that we may write:
\begin{align*}
    (C_{\nu} + \lambda)^{-\frac{1}{2}}((C_D + \lambda)f_{\lambda} - g_{D}) & = (C_{\nu} + \lambda)^{-\frac{1}{2}}((C_D - C_{\nu})f_{\lambda} +  S_{\nu}f^{*} - g_{D}) \\
    & = \mathbb{E}_{\mathcal{D}}[h(X, \cdot) (f_{\lambda}(X) - Y)] - \mathbb{E}_{XY}[h(X, \cdot) (f_{\lambda}(X) - Y)]
\end{align*}
which follow by the definition of $g_D$, $f_{\lambda}$ (see Remark \ref{Notation Remark}), $h(x, \cdot)$ in \eqref{Define h},  and the fact that $\mathbb{E}[Y | X] = f^{*}(X)$. In order to estimate this concentration, we compute absolute moments:
\begin{align}
    \mathbb{E}_{XY}[||h(X, \cdot) (f_{\lambda}(X) - Y)||_{K}^p] & \leq \mathbb{E}_{XY}[||h(X, \cdot) (f_{\lambda}(X) - f^{*}(X) + f^{*}(X) - Y)||_{K}^p] \nonumber \\
    & \leq 2^{p-1}\mathbb{E}_{XY}[||h(X, \cdot)||^{p}_{K}(|f_{\lambda}(X) - f^{*}(X)|^p + |f^{*}(X) - Y|^p)] \nonumber \\
    & \leq 2^{p-1}\mathbb{E}_{X}\Big[||h(X, \cdot)||^{p}_{K}\Big(\|f_{\lambda} - f^{*}\|_{\infty}^{p - 2}|f_{\lambda}(X) - f^{*}(X)|^2 + \frac{p!L^{p-2}\sigma^2}{2}\Big)\Big] \label{eq: Subgaussian} \\
    & \leq 2^{p-1}\Big(\phi^{\frac{p}{2}}(\lambda)\psi^{1 + \epsilon}(\lambda)\Big(\frac{||k^{\psi^{1 + \epsilon}}||_{\infty}||f^{*}||_{\phi}}{\psi^{1 + \epsilon}(\lambda)}\Big)^{p} + \frac{p!\sigma^2}{2}\Big(\frac{||k^{\psi^{1 + \epsilon}}||^2_{\infty}L^2}{\psi^{1 + \epsilon}(\lambda)}\Big)^{\frac{p-2}{2}}\text{tr}(C_{\nu}(C_{\nu} + \lambda)^{-1})\Big) \label{eq: Apply Helpers} \\
    & \preceq 2^{p-1}\Big(\phi^{\frac{p}{2}}(\lambda)\Big(\frac{||k^{\psi^{1 + \epsilon}}||_{\infty}||f^{*}||_{\phi}}{\psi^{1 + \epsilon}(\lambda)}\Big)^{p} + \frac{p!\sigma^2}{2}\Big(\frac{||k^{\psi^{1 + \epsilon}}||^2_{\infty}L^2}{\psi^{1 + \epsilon}(\lambda)}\Big)^{\frac{p-2}{2}}\text{tr}(C_{\nu}(C_{\nu} + \lambda)^{-1})\Big)  \label{eq: Lambda Decreasing} \\
   & \preceq 2^{p-1}\Big(\Big(\frac{||k^{\psi^{1 + \epsilon}}||_{\infty}||f^{*}||_{\phi}L}{\sqrt{\psi^{1 + \epsilon}(\lambda)}}\Big)^{p-2}\Big[\Big(\frac{||k^{\psi^{1 + \epsilon}}||^2_{\infty}||f^{*}||^2_{\phi}\phi(\lambda)}{\psi^{1 + \epsilon}(\lambda)}\Big) + \frac{\sigma^2 p!}{2}s^{-1}(\psi(\lambda)^{-1 - \epsilon})\Big]\Big) \label{eq: Apply Growth Rate} \\
   & \leq \frac{p!L_{\lambda}^{p-2}R^2}{2} \nonumber
\end{align}
with $L_{\lambda} = \frac{2L ||f^{*}||_{\phi} ||k^{\psi^{1 + \epsilon}}||_{\infty}}{\sqrt{\psi^{1 + \epsilon}(\lambda)}}$ and $R^2 = \frac{4||f^{*}||^2_{\phi}||k^{\psi^{1 + \epsilon}}||^2_{\infty}\phi(\lambda)}{\psi^{1 + \epsilon}(\lambda)} + 4\sigma^2 s^{-1}(\psi(\lambda)^{-1 - \epsilon})$. Here, \eqref{eq: Subgaussian} follows from Assumption \ref{Subexponential Noise}, \eqref{eq: Apply Helpers} invokes Lemmas \ref{h-norm Bound}, \ref{Bias Bound} and \ref{Uniform Bias} and the fact that $\mathbb{E}_{X}[||h(X, \cdot)||^2_K] = \text{tr}(C_{\nu}(C_{\nu} + \lambda)^{-1})$ by \eqref{Define h}, \eqref{eq: Lambda Decreasing} follows as $\psi^{1 + \epsilon}(\lambda) \to 0$ as $\lambda \to 0$, and finally \eqref{eq: Apply Growth Rate} follows from Lemma \ref{Effective Dimension Bound} and the fact that $\frac{\phi(\lambda)}{\psi^{1 + \epsilon}(\lambda)} \to 0$ as $\lambda \to 0$ for sufficiently small $\epsilon > 0$ by Assumption \ref{Growth Conditions}. Hence, applying Theorem 26 in \cite{fischer2020sobolev}, we have that, as $\lambda \to 0$:
\begin{equation*}
    ||(C_{\nu} + \lambda)^{-\frac{1}{2}}((C_D + \lambda)f_{\lambda} - g_{D})||^2_{K} \preceq \frac{32\log^2(\delta^{-1})}{n}\Big(\frac{4||f^{*}||^2_{\phi}||k^{\psi^{1 + \epsilon}}||^2_{\infty}\phi(\lambda)}{\psi^{1 + \epsilon}(\lambda)} + 4\sigma^2 s^{-1}(\psi(\lambda)^{-1 - \epsilon}) + \frac{4L^2||k^{\psi^{1 + \epsilon}}||^2_{\infty}||f^{*}||^2_{\phi}}{n\psi^{1 + \epsilon}(\lambda)}\Big)
\end{equation*}
with probability $1 - 2\delta$. Putting this all together, we have that:
\begin{equation}
    ||f_{\lambda} - f_{D, \lambda}||_{L^2(\nu)} \preceq \log(\delta^{-1})\sqrt{\frac{1152L^2||k^{\psi^{1 + \epsilon}}||^2_{\infty}||f^{*}||^2_{\phi}}{n\psi^{1 + \epsilon}(\lambda)}\Big(\phi(\lambda) + \frac{1}{n}\Big) + \frac{s^{-1}(\psi(\lambda)^{-1 - \epsilon})}{n}}
\end{equation}
with probability $1 - 3 \delta$, when $n \geq 8\log(\delta^{-1})||k^{\psi^{1 + \epsilon}}||^2_{\infty}g_{\lambda}\psi(\lambda)^{-1 - \epsilon}$. We demonstrate that this latter condition is eventually met, when the regularization program is  $\lambda_n \asymp (\frac{\phi}{\tilde{s} \circ \psi^{1 + \epsilon}})^{-1}(n^{-1})$. Indeed, we have that, as $n \to \infty$:
\begin{align*}
    \frac{8\log(\delta^{-1})||k^{\psi^{1 + \epsilon}}||^2_{\infty}g_{\lambda_n}}{n\psi^{1 + \epsilon}(\lambda_n)} & = \frac{8\log(\delta^{-1})||k^{\psi^{1 + \epsilon}}||^2_{\infty}\log\Big(2e\mathcal{N}(\lambda_n)\Big(1 + \frac{\lambda_n}{||C_{\nu}||}\Big)\Big)}{n\psi^{1 + \epsilon}(\lambda_n)} \\
    & \leq \frac{8\log(\delta^{-1})||k^{\psi^{1 + \epsilon}}||^2_{\infty}\log\Big(4e\mathcal{N}(\lambda_n)\Big)}{n\psi^{1 + \epsilon}(\lambda_n)} \\
    & \leq \frac{16\log(\delta^{-1})||k^{\psi^{1 + \epsilon}}||^2_{\infty}\log\Big(s^{-1}(\psi(\lambda_n)^{-1 - \epsilon})\Big)}{n\psi^{1 + \epsilon}(\lambda_n)}
\end{align*}
where the last line follows by Lemma \ref{Effective Dimension Bound}. Hence, it is sufficient to demonstrate that $\frac{\log\Big(s^{-1}(\psi(\lambda_n)^{-1 - \epsilon})\Big)}{n\psi^{1 + \epsilon}(\lambda_n)} \to 0$ for the regularization program. Indeed, since $\lambda_n \asymp (\frac{\phi}{\tilde{s} \circ \psi^{1 + \epsilon}})^{-1}(n^{-1})$, we have that, for sufficiently small $\epsilon > 0$:
\begin{align}
    \frac{\log\Big(s^{-1}(\psi(\lambda_n)^{-1 - \epsilon})\Big)}{n\psi^{1 + \epsilon}(\lambda_n)} & \preceq \frac{\log\Big(s^{-1}(\psi(\lambda_n)^{-1 - \epsilon})\Big)}{n\phi(\lambda_n)} \label{eq: Apply Relative Capacity} \\
    & \asymp  \frac{\log\Big(s^{-1}(\psi(\lambda_n)^{-1 - \epsilon})\Big)}{\tilde{s}(\psi^{1 + \epsilon}(\lambda_n))} \nonumber \\ 
    & = \frac{\log\Big(s^{-1}(\psi(\lambda_n)^{-1 - \epsilon})\Big)}{s^{-1}(\psi(\lambda_n)^{-1- \epsilon})} \to 0 \label{eq: Definition of Weird s}
\end{align}
where \eqref{eq: Apply Relative Capacity} follows from Assumption \ref{Growth Conditions} and \eqref{eq: Definition of Weird s} follows from the definition of $\tilde{s}$. 

\begin{proof}[Proof of Lemma \ref{Bias-Variance Decomposition}]
This follows immediately by combining Lemma \ref{Variance Bound} with Lemma \ref{Bias Bound}. 
\end{proof}
\begin{proof}[Proof of Theorem \ref{Main Upper}]
From Lemma \ref{Bias-Variance Decomposition}, we have:
\begin{align}
    ||f^{*} - f_{D, \lambda_n}||_{L^2(\nu)} & \preceq ||f^{*}||_{\phi}\sqrt{\phi(\lambda_n)} + \log(\delta^{-1})\sqrt{\frac{\sigma^2||k^{\psi^{1 + \epsilon}}||^2_{\infty}||f^{*}||^2_{\phi}}{n\psi^{1 + \epsilon}(\lambda_n)}\Big(\phi(\lambda_n) + \frac{1}{n}\Big) + \frac{s^{-1}(\psi(\lambda_n)^{-1 - \epsilon})}{n}} \nonumber \\
    & \preceq ||f^{*}||_{\phi}\sqrt{\phi(\lambda_n)} + \log(\delta^{-1})\sqrt{\frac{s^{-1}(\psi(\lambda_n)^{-1 - \epsilon})}{n}} \label{eq: To-Balance}
\end{align}
where the last line follows from the fact that $\frac{\phi(\lambda_n)}{\psi^{1 + \epsilon}(\lambda_n)} \to 0$ for sufficiently small $\epsilon > 0$ by Assumption \ref{Growth Conditions} and $n\psi^{1 + \epsilon}(\lambda_n) \succeq n\phi(\lambda_n) \asymp \tilde{s}(\psi^{1 + \epsilon}(\lambda_n)) = s^{-1}(\psi(\lambda_n)^{-1 - \epsilon}) \to \infty$ when $\lambda_n \asymp (\frac{\phi}{\tilde{s} \circ \psi^{1 + \epsilon}})^{-1}(n^{-1})$ as $n \to \infty$. The result follows immediately after substituting the latter choice of regularizer into \eqref{eq: To-Balance} (which was indeed proposed via balancing \eqref{eq: To-Balance}). 
\end{proof}
\end{proof}
\section{Proof of Theorem \ref{Lower Bound}}
\label{Lower Bound Proof}
\begin{lemma}
\label{Bernstein-Eigenvalue Formula}
For all sufficiently small $\epsilon > 0$:
\begin{equation*}
    c^2_{n-1}(\mathcal{H}_K, L^{\infty}(\mathcal{X})) \leq ||k^{\psi^{1 + \epsilon}}||^2_{\infty}\Big(\frac{t}{\psi^{1 + \epsilon}}\Big)(\mu_{n})
\end{equation*}
\end{lemma}
\begin{proof}
We first note that, for sufficiently small $\epsilon > 0$, $\frac{t}{\psi(t)}$ is concave by Assumption \ref{Growth Conditions}; hence, we have by the Gagliardo-Nirenberg interpolation inequality (Lemma \ref{Interpolation Inequality}), that for all $f \in \mathcal{H}_K$:
\begin{equation*}
    \frac{||f||^2_{\psi^{1 + \epsilon}}}{||f||^2_{K}} \leq \Big(\frac{t}{\psi^{1 + \epsilon}}\Big)\Big(\frac{||f||^2_{L^2(\nu)}}{||f||^2_{K}}\Big)
\end{equation*}
By the embedding condition $\mathcal{H}^{\psi^{1 + \epsilon}} \hookrightarrow L^{\infty}(\mathcal{X})$, we have that:
\begin{equation*}
    \frac{||f||^2_{\infty}}{||f||^2_{K}} \leq ||k^{\psi^{1 + \epsilon}}||^2_{\infty}\Big(\frac{t}{\psi^{1 + \epsilon}}\Big)\Big(\frac{||f||^2_{L^2(\nu)}}{||f||^2_{K}}\Big)
\end{equation*}
Hence, we have that:
\begin{align*}
    \sup_{\text{codim}(Z) < n} \inf_{f \in Z \cap \mathcal{H}_K} \frac{||f||^2_{\infty}}{||f||^2_{K}} & \leq \sup_{\text{codim}(Z) < n} \inf_{f \in Z \cap \mathcal{H}_K} ||k^{\psi^{1 + \epsilon}}||^2_{\infty}\Big(\frac{t}{\psi^{1 + \epsilon}}\Big)\Big(\frac{||f||^2_{L^2(\nu)}}{||f||^2_{K}}\Big) \\
    & = ||k^{\psi^{1 + \epsilon}}||^2_{\infty}\Big(\frac{t}{\psi^{1 + \epsilon}}\Big)\Big(\sup_{\text{codim}(Z) \geq n} \inf_{f \in Z \cap \mathcal{H}_K} \frac{||f||^2_{L^2(\nu)}}{||f||^2_{K}}\Big)
\end{align*}
where the last step follows from the fact that $\frac{t}{\psi^{1 + \epsilon}(t)}$ is nondecreasing and continuous for sufficiently small $\epsilon > 0$. Hence, by the definition of the Gelfand width \eqref{eq: Gelfand Def}, and the fact that, on the RHS, Gelfand widths coincide with the singular values ($\sqrt{\mu_n}$) of the Hilbert space embedding $\mathcal{H}_K \hookrightarrow L^2(\nu)$ (e.g. Proposition 5 in \cite{mathe2008direct}), we obtain our result. 
\end{proof}

\begin{lemma}
\label{Gelfand-Fourier}
\begin{equation*}
    c_{n}(\mathcal{H}_K, L^{\infty}(\mathcal{X})) \geq \sqrt{C_d n\mathcal{F}_d \kappa(n^{\frac{1}{d}})}
\end{equation*}
for some constant $C_d > 0$ independent of $n$. 
\end{lemma}
\begin{proof}
Let $\{z_i\}_{i = 1}^{2n} \subset \mathcal{X}$ be a set of $2n$ distinct points in $\mathcal{X}$. By the assumption that $\kappa$ is positive-definite on $\mathbb{R}^d$, we have that for any $y \in \mathbb{R}^{2n}$, we can find a unique $f \in \text{span}\{K(z_1, \cdot), \ldots, K(z_{2n})\}$ such that $f(z_i) = y_i$ $\forall y \in [2i]$. Indeed, if we let $(K_{2n})_{i, j} = k(z_i, z_j)$ denote the kernel Gram matrix on $\{z_i\}_{i = 1}^{2n}$, then:
\begin{equation*}
    f = \sum_{i = 1}^{2n} c_i k(z_i, \cdot)
\end{equation*}
where $c = K^{-1}_{2n}y$. Let $T:\ell^{2n}_2 \to \mathcal{H}_K$ denote this bijectve mapping, i.e. $(Ty)(z_i) = y_i$ for $y \in \mathbb{R}^{2n}$. Observe that:
\begin{align*}
    ||T||_{\ell^{2n}_2 \to \mathcal{H}_K} & = \max_{||y||_{2} = 1} \sqrt{\Big\langle \sum_{i = 1}^{2n} c_i k(z_i, \cdot), \sum_{i = 1}^{2n} c_i k(z_i, \cdot) \Big\rangle_K} \\
    & = \max_{||y||_{2} = 1} \sqrt{c^TK_{2n}c} \\
    & = \max_{||y||_{2} = 1} \sqrt{y^TK^{-1}_{2n}y} \\
    & = \sqrt{||K^{-1}_{2n}||}
\end{align*}
where in the penultimate line, we recall that $c = K^{-1}_{2n}y$. Finally, let $J: L^{\infty}(\mathcal{X}) \to l^{2n}_{\infty}$ denote the restriction mapping $(Jf)_{i} = f(z_i)$ for $f \in L^{\infty}(\mathcal{X})$ and $i \in [2n]$. Then, we obtain the following factorization of the imbedding $\ell^{2n}_{2} \hookrightarrow \ell^{2n}_{\infty}$:
\begin{equation*}
    \ell^{2n}_{2} \stackrel{T}{\to} \mathcal{H}_K \stackrel{I}{\hookrightarrow} L^{\infty}(\mathcal{X}) \stackrel{J}{\to} \ell^{2n}_{\infty}
\end{equation*}
where $I:\mathcal{H}_K  \hookrightarrow L^{\infty}(\mathcal{X})$ is the canonical embedding of interest. Then, by the submultiplicativity of the Gelfand numbers, we have that:
\begin{equation*}
    c_n(\ell^{2n}_{2}, \ell^{2n}_{\infty}) \leq ||T||c_n(\mathcal{H}_K, L^{\infty}(\mathcal{X}))||J|| \leq \sqrt{||K^{-1}_{2n}||}c_n(\mathcal{H}_K, L^{\infty}(\mathcal{X}))
\end{equation*}
as clearly $||J|| \leq 1$. Hence, we have that:
\begin{align}
    c_n(\mathcal{H}_K, L^{\infty}(\mathcal{X})) & \geq \frac{c_n(\ell^{2n}_{2}, \ell^{2n}_{\infty})}{\sqrt{||K^{-1}_{2n}||}} \nonumber \\
    & = \frac{d_n(\ell^{2n}_{1}, \ell^{2n}_{2})}{\sqrt{||K^{-1}_{2n}||}} \label{eq: Gelfand-Kolmogorov duality} \\
    & = \sqrt{\frac{1}{\Big(1 - \frac{n}{2n}\Big)||K^{-1}_{2n}||}} \label{eq: Stechkin} \\
    & = \sqrt{\frac{2}{||K^{-1}_{2n}||}} \label{eq: Gelfand Lower 1}
\end{align}
where in \eqref{eq: Gelfand-Kolmogorov duality} we have used the duality of Gelfand and Kolmogorov numbers, and in \eqref{eq: Stechkin} we have applied Stechkin's identity $d_n(\ell^{2n}_{1}, \ell^{2n}_{2}) = \Big(1 - \frac{n}{2n}\Big)^{-\frac{1}{2}}$ (see e.g. Lemma 4.3.19 in \cite{dung2018hyperbolic}). We now wish to upper bound $||K^{-1}_{2n}||_2$ --- we do this by using the following lower bound on the minimum eigenvalue of $K_{2n}$, given in Theorem 12.3 of \cite{wendland2004scattered}:
\begin{equation*}
    \lambda_{\min}(K_{2n}) \geq C_{d}\Big(\frac{12.76d}{q_{z}}\Big)^d \mathcal{F}_d \kappa\Big(\frac{12.76d}{q_{z}}\Big)
\end{equation*}
where $q_{z} \equiv \min_{i \neq j} ||z_i - z_j||$ and $C_d > 0$ is an absolute constant depending only the dimension $d$. We choose a maximal $2n$-set $\{z_i\}_{i = 1}^{2n}$, so that $q_z$ is the packing number $p_n(\mathcal{X})$ of the domain $\mathcal{X}$. From the standard equivalence $p_n(\mathcal{X}) \asymp \epsilon_n(\mathcal{X})$ and Assumption \ref{Domain Condition}, we have that:
\begin{equation*}
    \lambda_{\min}(K_{2n}) \geq C_{d}n\mathcal{F}_d \kappa \Big(12.76d (2n)^{\frac{1}{d}}\Big)
\end{equation*}
From the the dilation condition $|\beta_{\mathcal{F}_d \kappa}| < \infty$ in Assumption \ref{Radial Kernel}, we therefore have that:
\begin{equation*}
    \lambda_{\min}(K_{2n}) \geq C_{d}n\mathcal{F}_d \kappa(n^{\frac{1}{d}})
\end{equation*}
where again we recall that the overhead constant $C_{d}$ depends only on the ambient dimension in Assumption \ref{Domain Condition} and the growth parameter $\beta_{\mathcal{F}_d \kappa}$. Noting that $\|K^{-1}_{2n}\|_2 = \lambda_{\min}(K_{2n})^{-1}$ and plugging back into \eqref{eq: Gelfand Lower 1}, we have that:
\begin{equation*}
    d_{n}(\mathcal{H}_K, L^{\infty}(\mathcal{X})) \geq \sqrt{C_d n\mathcal{F}_d \kappa(n^{\frac{1}{d}})}
\end{equation*}
\end{proof}
\begin{corollary}
\label{Eigenvalue Lower Bound}
For all $\epsilon > 0$:
\begin{equation*}
    \psi^{-1}(s(n)^{-1}) \preceq \mu_{n} \preceq \psi^{-1}(s(n)^{-\frac{1}{1 + \epsilon}})
\end{equation*}
\end{corollary}
\begin{proof}
From Assumption \ref{Eigenfunction Growth}, for any $\alpha \in (0, 1)$, we can choose a $c \in l^n_{2}$ with $\sum_{i = 1}^n c^2_i = 1$ and $x^{*} \in \mathcal{X}$, such that:
\begin{equation*}
    \sum_{i = 1}^n c_i e_i(x^{*}) \succeq \alpha \sqrt{s(n)}
\end{equation*}
Hence, noting that $\{c_i\}_{i = 1}^n$ can always be chosen such that $c_i e_i(x^{*}) \geq 0$, for all $i \in [n]$, we have that that:
\begin{align*}
    \alpha \sqrt{\psi^{1 + \epsilon}(\mu_n)s(n)} & \preceq \sum_{i = 1}^n \sqrt{\psi^{1 + \epsilon}(\mu_i)}c_i e_i(x^{*}) \\
    & \leq ||k^{\psi^{1 + \epsilon}}||_{\infty}
\end{align*}
by Cauchy-Schwartz. Observing that $||k^{\psi^{1 + \epsilon}}||_{\infty} < \infty$ since $H^{\psi^{1 + \epsilon}} \hookrightarrow L^{\infty}(\mathcal{X})$ by Assumption \ref{Embedding Condition} and that $\alpha$ was chosen arbitrarily, we obtain the upper estimate:
\begin{equation*}
    \mu_n \preceq \psi^{-1}(s(n)^{-\frac{1}{1 + \epsilon}})
\end{equation*}
where we can factor out any constants in the argument of $\psi^{-1}$ by the $\Delta_2$ condition in Assumption \ref{Growth Conditions}. For the lower estimate, we apply Lemma \ref{Bernstein-Eigenvalue Formula} and Lemma \ref{Gelfand-Fourier} to obtain:
\begin{equation}
\label{eq: Bernstein Chain}
    \frac{t}{\psi^{1 + \epsilon}}\Big(\mu_n\Big) \succeq c^2_{n-1}(\mathcal{H}_K, L^{\infty}(\mathcal{X})) \succeq n\mathcal{F}_d k(n^{\frac{1}{d}})
\end{equation}
Now, from Assumption \ref{Opt Smoothness}, we have that:
\begin{align*}
    \psi^{-1}(s(t)^{-1})  & = \Big(\frac{t}{\psi}\Big)^{-1}\Big(\Big(\frac{t}{\psi}\Big)(\psi^{-1}(s(t)^{-1}))\Big) \\
    & = \Big(\frac{t}{\psi}\Big)^{-1}\Big(\frac{\psi^{-1}(s(t)^{-1})}{s(t)^{-1}}\Big) \\
    & \asymp \Big(\frac{t}{\psi}\Big)^{-1}\Big(\frac{s(t)t\mathcal{F}_d k(t^{\frac{1}{d}})}{s(t)}\Big)\\
    & = \Big(\frac{t}{\psi}\Big)^{-1}(t\mathcal{F}_d k(t^{\frac{1}{d}})) \\
    & \preceq \Big(\frac{t}{\psi^{1 + \epsilon}}\Big)^{-1}(t\mathcal{F}_d k(t^{\frac{1}{d}}))
\end{align*}
from \eqref{eq: Bernstein Chain} and the fact that $t\mathcal{F}_d k(t^{\frac{1}{d}}) \to 0$ as $ t \to \infty$ by the assumed compactness of the imbedding $\mathcal{H}_K \stackrel{c}{\hookrightarrow} L^{\infty}(\mathcal{X})$. Hence, combining this with \eqref{eq: Bernstein Chain}, we have obtain the lower estimate:
\begin{equation*}
    \mu_n \succeq \psi^{-1}(s(n)^{-1})
\end{equation*}
\end{proof}

\begin{remark}
\label{Remark on Dilations}
Observe, that in light of the proof of Lemma \ref{Gelfand-Fourier}, Corollary \ref{Eigenvalue Lower Bound} would still hold if in Assumption \ref{Opt Smoothness}, $\mathcal{F}_d \kappa(t^{\frac{1}{d}})$ was replaced by $\mathcal{F}_d \kappa\Big(12.76d (2t)^{\frac{1}{d}}\Big)$. However, in order to streamline the analysis, we have chosen the simpler assumption, as also discussed in section \ref{Ridge Assumptions}. 
\end{remark}
\begin{proof}[Proof of Theorem \ref{Lower Bound}]
We follow the general outline of the proof of Theorem 2 in \cite{fischer2020sobolev} that was earlier utilized in \cite{caponnetto2007optimal} and \cite{blanchard2018optimal}. 
Our goal is to construct a sequence of probability measures $\{P_j\}_{j = 1}^{M_{\epsilon}}$ on $\mathcal{X} \times \mathbb{R}$ with $P_j|_{\mathcal{X}} = \nu$ (i.e. the measures share the same marginal on $\mathcal{X}$) that are \textit{hard} to learn, i.e. their regression functions $f_j$ satisfy the source conditions but are sufficiently ``well spread out'' in $L^2(\nu)$ so that it is sufficiently hard to distinguish between these functions given only the sample data. We make these ideas more precise in the subsequent analysis. Our candidate regression functions take the following form --- for some fixed $\epsilon \in (0, 1)$ and $m \in \mathbb{N}$, we consider the element:
\begin{equation*}
    f_{\omega} = 2\sqrt{\frac{8\epsilon}{m}}\sum_{i = 1}^m \omega_i e_{i + m}
\end{equation*}
where $\omega \in \{0, 1\}^m$ is some binary string. Since the sum is finite, we have that $f \in \mathcal{H}_K \subset L^{\infty}(\mathcal{X}) \cap \mathcal{H}^{\phi}$. We will demonstrate that $||f||_{\phi} \leq B_{\phi}$ and $||f||_{\infty} \leq B_{\infty}$ for sufficiently large choices of $m$. Indeed, we have that, for any $\delta > 0$ (which we will fix soon):
\begin{align*}
    ||f_{\omega}||^2_{\phi} & = \frac{32\epsilon}{m}\sum_{i = 1}^m \frac{\omega^2_i}{\phi(\mu_{i + m})} \leq \frac{32\epsilon}{\phi(\mu_{2m})} \\
    ||f_{\omega}||^2_{\infty} & \leq ||k^{\psi^{1 + \delta}}||_{\infty}^2||f_{\omega}||^2_{\psi^{1 + \delta}} = \frac{32||k^{\psi^{1 + \delta}}||_{\infty}^2\epsilon}{m}\sum_{i = 1}^m \frac{\omega^2_i}{\psi^{1 + \delta}(\mu_{i + m})} \leq \frac{32||k^{\psi^{1 + \delta}}||_{\infty}^2\epsilon}{\psi^{1 + \delta}(\mu_{2m})}
\end{align*}
Now, by Corollary \ref{Eigenvalue Lower Bound}, we have that: 
\begin{align*}
    ||f_{\omega}||^2_{\phi} & \leq \frac{32\epsilon}{\phi(C_d\psi^{-1}(s(m)^{-1}))} \\
    ||f_{\omega}||^2_{\infty} & \leq \frac{64||k^{\psi^{1 + \delta}}||_{\infty}^2\epsilon}{C_{d}s(m)^{-1 - \delta}}
\end{align*}
where $C_d > 0$ depends only $d$ (and the constants in the $\Delta_2$ condition on $\phi, \psi$, and $s$). Hence, by the $\Delta_2$ condition on $\phi$ and $\psi$, we have that $||f||^2_{\phi} \leq B^2_{\phi}$ and $||f||^2_{\infty} \leq B^2_{\infty}$, when:
\begin{equation*}
    m \leq s^{-1}\Big(\min\Big\{\Big(\frac{C_d B^2_{\infty}}{64||k^{\psi^{1 + \delta}}||_{\infty}^2 \epsilon}\Big)^{\frac{1}{1 + \delta}}, \Big(\psi\Big(C_d\phi^{-1}\Big(\frac{32 \epsilon}{B^2_{\phi}}\Big)\Big)\Big)^{-1}\Big\}\Big)
\end{equation*}
(where we have redefined the constant $C_d$, as we will throughout this proof). By Assumption \ref{Growth Conditions}, we have that $\frac{\phi(t)}{\psi(t)} \to 0$ as $t \to 0$, and thus if we fix a sufficiently small $\delta > 0$ \Big(e.g. $\delta < \lim_{t \to 0} \log \frac{\phi(t)}{\psi(t)} / \lim_{t \to 0} \log \psi(t)$\Big), we have $\frac{\phi(t)}{\psi^{1 + \delta}(t)} \to 0$ as $t \to 0$. Hence, for sufficiently small $t > 0$, we have that $\frac{\psi(\phi^{-1}(t))}{t^{\frac{1}{1 + \delta}}} = \frac{\psi(\phi^{-1}(t))}{\phi^{\frac{1}{1 + \delta}}(\phi^{-1}(t))} \geq 1$ for this fixed choice of $\delta > 0$. Hence, we may choose $\epsilon_1 \in (0, 1)$, such that for $\epsilon \in (0, \epsilon_1]$, we have $||f||^2_{\phi} \leq B^2_{\phi}$ and $||f||^2_{\infty} \leq B^2_{\infty}$, when:
\begin{equation*}
    m \leq U_{d, \infty, \phi}s^{-1}\Big(\frac{1}{\psi(\phi^{-1}(\epsilon))}\Big) = U_{d, \infty, \phi}\tilde{s}(\psi(\phi^{-1}(\epsilon)))
\end{equation*}
where we have recalled the definition of $\tilde{s}(t) = s^{-1}\Big(\frac{1}{t}\Big)$ and $U_{d, \infty, \phi} > 0$ is a constant depending only on $d, B_{\infty}, B_{\phi}$ that we have again extracted via the $\Delta_2$ conditions on $s$, $\psi$ and $\phi$. We can ensure that there is an $m \geq 1$ satisfying this bound by choosing $\epsilon \leq \epsilon_2 = \min \{\epsilon_1, \phi(\psi^{-1}(s(U_{d, \infty, \phi}^{-1})^{-1}))\}$. Moreover, we have for any two binary strings $\omega^{(1)}, \omega^{(2)} \in \{0, 1\}^m$:
\begin{equation*}
    ||f_{\omega^{(1)}} - f_{\omega^{(2)}}||^2_{L^2(\nu)} = \frac{32\epsilon}{m}\sum_{i = 1}^m (\omega^{(1)}_{i} - \omega^{(2)}_{i})^2 \leq 32\epsilon
\end{equation*}
Now, observe that if we choose $\epsilon \leq \epsilon_3 \equiv \min\{\epsilon_2, \phi(\psi^{-1}(s(9U_{d, \infty, \phi}^{-1})^{-1}))\}$, and choose $m_{\epsilon} = \lfloor U_{d, \infty, \phi}\tilde{s}(\psi(\phi^{-1}(\epsilon))) \rfloor$ like before, we have $m_{\epsilon} \geq 9$. Then, by the Gilbert-Varshamov bound (see e.g. Lemma 24 in \cite{fischer2020sobolev}), there exists a $M_{\epsilon} \geq 2^{\frac{m_{\epsilon}}{8}}$, and some binary strings $\omega^{(1)}, \omega^{(2)}, \ldots, \omega^{(M_{\epsilon})}$, such that:
\begin{equation*}
    \sum_{i = 1}^{m_{\epsilon}} (\omega^{(k)}_i - \omega^{(l)}_i)^2 \geq \frac{m_{\epsilon}}{8}
\end{equation*}
for all $k, l \in [M_{\epsilon}]$ with $k \neq l$. Therefore, we have:
\begin{equation}
\label{eq: Separation}
    ||f_{\omega^{(k)}} - f_{\omega^{(l)}}||^2_{L^2(\nu)} \geq 4\epsilon
\end{equation}
for $k, l \in [M_{\epsilon}]$. We now construct measures $\{P_{j}\}_{j = 1}^{M_{\epsilon}}$ on $\mathcal{X} \times \mathcal{Y}$ such that $P_j|_{\mathcal{X}} = \nu = P_X$, with $\mathbb{E}[Y | X] = f_{\omega^{(j)}}(X) \equiv f_j(X)$ and $Y - \mathbb{E}[Y | X] \sim \mathcal{N}(0, \bar{\sigma}^2)$ for $(X, Y) \sim P_j$ and all $j \in [M_{\epsilon}]$ (where $\bar{\sigma} = \max\{\sigma, L\}$). The measures $P_j$ satisfy Assumption \ref{Subexponential Noise} by Lemma 21 in \cite{fischer2020sobolev}.  Letting $P_0$ be such that $f_0 = 0$. Then, we have that:
\begin{equation*}
    \frac{1}{M_{\epsilon}}\sum_{i = 1}^{M_{\epsilon}} \text{KL}(P^n_i || P^n_0) = \frac{n}{2\sigma^2 M_{\epsilon}}\sum_{i = 1}^{M_{\epsilon}} ||f_i - f_0||^2_{L^2(\nu)} \leq \frac{16n\epsilon}{\sigma^2} \equiv \alpha^{*}
\end{equation*}
Thus, for any measurable function $\Theta: (\mathcal{X} \times \mathbb{R})^n \to \{0, 1, \ldots, M_{\epsilon}\}$, we have by Theorem 20 in \cite{fischer2020sobolev} that:
\begin{equation}
\label{eq: KL Lower Bound}
    \max_{j \in [M_{\epsilon}]} P_j(\Theta(D) \neq j) \geq \frac{\sqrt{M_{\epsilon}}}{1 + \sqrt{M_{\epsilon}}}\Big(1 - \frac{48n\epsilon}{\sigma^2\log M_{\epsilon}} - \frac{1}{2\log M_{\epsilon}}\Big)
\end{equation}
Note that $M_{\epsilon} \geq 2^{\frac{m_{\epsilon}}{8}} \geq 2^{\frac{U_{d, \infty, \phi}\tilde{s}(\psi(\phi^{-1}(\epsilon)))}{9}}$ using the definition of $m_{\epsilon}$ and the fact that $m_{\epsilon} \geq 9$. Substituting this into \eqref{eq: KL Lower Bound}, we obtain:
\begin{equation*}
      \max_{j \in [M_{\epsilon}]} P_j(\Theta(D) \neq j) \geq \frac{\sqrt{M_{\epsilon}}}{1 + \sqrt{M_{\epsilon}}}\Big(1 - \frac{432n\epsilon}{\tilde{s}(\psi(\phi^{-1}(\epsilon))) \cdot U_{d, \infty, \phi}\sigma^2\log 2} - \frac{1}{2\log M_{\epsilon}}\Big)  
\end{equation*}
Now, we choose $\epsilon_n = \tau \phi\Big((\frac{\phi}{\tilde{s} \circ \psi})^{-1}(n^{-1})\Big)$ where $\tau < 1$ has been chosen so that $\epsilon_n \leq \epsilon_3$.  Define:
\begin{equation*}
    \Theta(D) \equiv \text{arg} \min_{j \in [M_{\epsilon_n}]} ||f_{D} - f_j||_{L^2(\nu)}
\end{equation*}
Then for any $j \neq \Theta(D)$, we have by \eqref{eq: Separation}:
\begin{equation*}
    2\sqrt{\epsilon_n} \leq ||f_{\Theta(D)} - f_j||_{L^2(\nu)} \leq ||f_D - f_{\Theta(D)}||_{L^2(\nu)} + ||f_D - f_{j}||_{L^2(\nu)} \leq 2||f_D - f_{j}||_{L^2(\nu)}
\end{equation*}
Thus, we have:
\begin{align*}
    \max_{j \in [M_{\epsilon_n}]} P_j(||f_D - f_{j}||_{L^2(\nu)} \geq \sqrt{\epsilon_n}) & \geq  \max_{j \in [M_{\epsilon_n}]} P_j(\Theta(D) \neq j) \\
    & \geq  \frac{\sqrt{M_{\epsilon_n}}}{1 + \sqrt{M_{\epsilon_n}}}\Big(1 - \frac{432n\epsilon_n}{\tilde{s}(\psi(\phi^{-1}(\epsilon_n))) \cdot U_{d, \infty, \phi}\sigma^2\log 2} - \frac{1}{2\log M_{\epsilon_n}}\Big)  \\
    & \geq \frac{\sqrt{M_{\epsilon_n}}}{1 + \sqrt{M_{\epsilon_n}}}\Big(1 - C_{\tau} - \frac{1}{2\log M_{\epsilon_n}}\Big)
\end{align*}
after substituting the choice of $\epsilon_n$ and letting $C_{\tau} = \frac{432\tau K^{u}}{U_{d, \infty, \phi}\sigma^2\log 2}$ (where $u = \log \tau$ and $K = K(\psi, s, \phi) > 1$ is a constant depending only on $D^1_{\psi}, D^2_{\psi}, D^1_{\phi}, D^2_{\phi}, D^1_{s}, D^2_{s} > 1$). Observing that $M_{\epsilon_n} \to \infty$ as $\epsilon_n \to 0$ (as $n \to \infty$), we obtain our result. 
\end{proof}

\section{Relaxing Radiality}
\label{Relaxing Radiality}
 In this section, we prove Lemma \ref{Alternative Spectral Estimates}. 

\begin{proof}[Proof of Lemma \ref{Alternative Spectral Estimates}]
The upper bound follows analogously to Corollary \ref{Eigenvalue Lower Bound}. Without loss of generality, suppose $0 \in \mathcal{X}$ (this can always be achieved via translation). By assumption, we have that for $n \geq (2C_1)^d$, $\epsilon_n((1 - C_1 n^{-\frac{1}{d}})\bar{\mathcal{X}}) = (1 - C_1 n^{-\frac{1}{d}}) \epsilon_n(\bar{\mathcal{X}}) \geq \frac{C_1 n^{-\frac{1}{d}}}{2}$. Hence, we can find $n$ points $\{x_i\}_{i = 1}^{n}$ such that the annuli in the family $\mathcal{A} = \Big\{B\Big(x_i, \frac{C_1 n^{-\frac{1}{d}}}{4}\Big) \setminus B\Big(x_i, \frac{C_1 n^{-\frac{1}{d}}}{8}\Big): i = 1, \ldots, n\Big\}$ are disjoint and contained in $\mathcal{X}$ (we can choose these points in $(1 - C_1 n^{-\frac{1}{d}})\mathcal{X}$; containment in $\mathcal{X}$ is then guaranteed). Then, let $\phi_i \in \mathcal{H}_K$ be the cutoff potential for the annulus $A_i \in \mathcal{A}$, i.e the unique minimizer such that $||\phi_i||^2_{K} = \text{Cap}_{\mathcal{X}}(A_i; \mathcal{H}_K) = \text{Cap}_{\mathcal{X}}\Big(B\Big(x_i, \frac{C_1 n^{-\frac{1}{d}}}{8}\Big), B\Big(x_i, \frac{C_1 n^{-\frac{1}{d}}}{4}\Big); \mathcal{H}_K\Big)$ (the existence of this minimizer again follows from Lemma 2.1.1 in \cite{fukushima2010dirichlet}). Observe that for any $i \in [n]$, we have that, as $n \to \infty$:
\begin{align}
    \frac{||\phi_i||^2_{L^2(\nu)}}{||\phi_i||^2_{K}} & \geq \frac{\nu\Big(B\Big(x_i, \frac{C_1 n^{-\frac{1}{d}}}{8}\Big)\Big)}{\text{Cap}_{\mathcal{X}}(A_i; \mathcal{H}_K)} \label{eq: L^2 lower bound}\\
    & \asymp \frac{\nu(B(x_i, n^{-\frac{1}{d}}))}{\text{Cap}_{\mathcal{X}}(A_i; \mathcal{H}_K)} \nonumber \\
    & \asymp \nu(B(x_i, n^{-\frac{1}{d}})) \cdot \inf_{x \in \mathcal{X}} \frac{\psi^{-1}(\nu(B(x, n^{-\frac{1}{d}})))}{\nu(B(x, n^{-\frac{1}{d}}))} \label{eq: Apply Isocapacitary Condition} \\
    & \succeq \inf_{x \in \mathcal{X}} \psi^{-1}(\nu(B(x, n^{-\frac{1}{d}}))) \nonumber \\
    & \asymp \psi^{-1}(s(n)^{-1}) \label{eq: Eigenfunction-Volume} 
\end{align}
where \eqref{eq: L^2 lower bound} follows from the fact that $\phi_i = 1$ on $B\Big(x_i, \frac{n^{-\frac{1}{d}}}{8}\Big)$ by definition, \eqref{eq: Apply Isocapacitary Condition} follows from \eqref{eq: Isocapacitary Condition}, and finally \eqref{eq: Eigenfunction-Volume} follows from \eqref{eq: Weyl Eigenfunction Growth}. Then, consider the $n$-dimensional subspace $\text{span}(\{\phi_i\}_{i \in [n]})$, and note by the disjointness of the $A_i$ and the locality of $\langle \cdot, \cdot \rangle_{K}$, we have that for each $\phi \in \text{span}(\{\phi_i\}_{i \in [n]})$:
\begin{equation*}
    \frac{||\phi||^2_{L^2(\nu)}}{||\phi||^2_{K}} \succeq \psi^{-1}(s(n)^{-1})
\end{equation*}
Applying the Courant-Fisher minimax principle, we thus have:
\begin{equation*}
    \mu_n \succeq \psi^{-1}(s(n)^{-1})
\end{equation*}
\end{proof}

\section{Auxiliary Results}
\begin{lemma}
\label{Interpolation Inequality}
Suppose $\frac{t}{\psi(t)}$ is concave. Then, if Assumption \ref{Embedding Condition} holds, we have for all $f \in \mathcal{H}_K$ and sufficiently small $\epsilon > 0$:
\begin{equation*}
    \frac{||f||^2_{\psi^{1 + \epsilon}}}{||f||^2_{K}} \leq \frac{t}{\psi^{1 + \epsilon}}\Big(\frac{||f||^2_{L^2(\nu)}}{||f||^2_{K}}\Big)
\end{equation*}
\end{lemma}
\begin{proof}
Observe that for sufficiently small $\epsilon > 0$, $\frac{t}{\psi^{1 + \epsilon}(t)}$ maintains the concavity in Lemma \ref{Growth Conditions}. Then the result follows immediately from Proposition 7 in \cite{mathe2008direct} with $\tilde{\kappa} = \sqrt{t}$, $\tilde{\phi} = \sqrt{\psi^{1 + \epsilon}(t)}$, and $\tilde{\psi}(t) = 1$ (note here $\tilde{\kappa}$, $\tilde{\phi}$, and $\tilde{\psi}$ denote the variable notation used in Prop 7 of \cite{mathe2008direct} and do not coincide with the objects $\kappa$, $\phi$, and $\psi$ in this paper). 
\end{proof}
\begin{lemma}
\label{h-norm Bound}
For all sufficiently small $\epsilon > 0$
\begin{equation*}
    \sup_{x \in \mathcal{X}} ||(C_{\nu} + \lambda)^{-\frac{1}{2}}k(x, \cdot)||_{K} \leq \sqrt{\frac{||k^{\psi^{1 + \epsilon}}||_{\infty}^2}{\psi^{1 + \epsilon}(\lambda)}}
\end{equation*}
\end{lemma}
\begin{proof}
We first note that, for any linear operator $T$ on $\mathcal{H}_K$, we have:
\begin{align*}
    \sup_{x \in \mathcal{X}} ||Tk(x, \cdot)||_{K} & = \sup_{||f||_{K} = 1, x \in \mathcal{X}} \langle Tk(x, \cdot), f \rangle_K \\
    & = \sup_{||f||_{K} = 1, x \in \mathcal{X}} \langle k(x, \cdot), T^{*}f \rangle_K \\
    & = \sup_{||f||_{K} = 1, x \in \mathcal{X}} (T^{*}f)(x) \\
    & = \sup_{||f||_{K} = 1} ||T^{*}f||_{\infty} \\
    & = ||T^{*}||_{\mathcal{H} \to L^{\infty}(\mathcal{X})}
\end{align*}
Hence, since $(C_{\nu} + \lambda)^{-\frac{1}{2}}$ is self-adjoint, we have by the interpolation inequality (Lemma \ref{Interpolation Inequality}) and the weak embedding $H^{\psi} \hookrightarrow L^{\infty}(\mathcal{X})$, that for sufficiently small $\epsilon > 0$:
\begin{equation}
\label{eq: Operator Interpolation}
    ||(C_{\nu} + \lambda)^{-\frac{1}{2}}f||^2_{\infty} \leq \frac{||k^{\psi^{1 + \epsilon}}||^2_{\infty}||(C_{\nu} + \lambda)^{-\frac{1}{2}}f||^2_{2}}{\psi^{1 + \epsilon}\Big(\frac{||(C_{\nu} + \lambda)^{-\frac{1}{2}}f||^2_{2}}{||(C_{\nu} + \lambda)^{-\frac{1}{2}}f||^2_{K}}\Big)} 
\end{equation}
Taking the supremum over $f \in \mathcal{B}(\mathcal{H}_K)$, and noting that the RHS of \eqref{eq: Operator Interpolation} is jointly nondecreasing in $||(C_{\nu} + \lambda)^{-\frac{1}{2}}f||_{2}$ and $||(C_{\nu} + \lambda)^{-\frac{1}{2}}f||_{K}$ by the fact that, for sufficiently small $\epsilon > 0$, $\frac{t}{\psi^{1 + \epsilon}(t)}$ is concave and nondecreasing by Assumption \ref{Growth Conditions}, we obtain:
\begin{equation*}
    ||(C_{\nu} + \lambda)^{-\frac{1}{2}}||^2_{\mathcal{H} \to L^{\infty}(\mathcal{X})} \leq \frac{||k^{\psi^{1 + \epsilon}}||^2_{\infty}}{\psi^{1 + \epsilon}(\lambda)}
\end{equation*}
where we have used the fact that $\sup_{f \in \mathcal{B}(\mathcal{H}_K)} ||(C_{\nu} + \lambda)^{-\frac{1}{2}}f||^2_{2} = \sup_{f \in \mathcal{B}(\mathcal{H}_K)} ||C^{\frac{1}{2}}_{\nu}(C_{\nu} + \lambda)^{-\frac{1}{2}}f||^2_{K} \leq 1$ and $\sup_{f \in \mathcal{B}(\mathcal{H}_K)} ||(C_{\nu} + \lambda)^{-\frac{1}{2}}f||^2_{K} = \lambda^{-1}$. 
\end{proof}
\begin{lemma}
\label{Effective Dimension Bound}
For any $\epsilon > 0$
\begin{equation*}
    \mathcal{N}(\lambda) \equiv \text{tr}(C_{\nu}(C_{\nu} + \lambda)^{-1}) \preceq s^{-1}(\psi(\lambda)^{-1 - \epsilon})
\end{equation*}
\end{lemma}
\begin{proof}
From Corollary \ref{Eigenvalue Lower Bound}, we obtain $\mu_i \preceq \psi^{-1}(s(i)^{-\frac{1}{1 + \epsilon}})$ for any $\epsilon > 0$. Therefore, we have:
\begin{align*}
    \mathcal{N}(\lambda) & \equiv \text{tr}(C_{\nu}(C_{\nu} + \lambda)^{-1}) \\
    & = \sum_{i = 1}^{\infty} \frac{\mu_i}{\mu_i + \lambda} \\
    & = \sum_{i = 1}^{\infty} \frac{1}{1 + \lambda \mu^{-1}_i} \\
    & \leq \sum_{i = 1}^{\infty} \frac{1}{1 + \frac{C_1\lambda}{\psi^{-1}(s(i)^{-\frac{1}{1 + \epsilon}})}} \\
    & \leq \int_{0}^{\infty} \frac{dt}{1 + \frac{C_1\lambda}{\psi^{-1}(s(t)^{-\frac{1}{1 + \epsilon}})}} \\
    & = \int_{0}^{s^{-1}(\psi(\lambda)^{-1 -\epsilon})} \frac{dt}{1 + \frac{C_1\lambda}{\psi^{-1}(s(t)^{-\frac{1}{1 + \epsilon}})}} + \int_{s^{-1}(\psi(\lambda)^{-1 - \epsilon})}^{\infty} \frac{dt}{1 + \frac{C_1\lambda}{\psi^{-1}(s(t)^{-\frac{1}{1 + \epsilon}})}} 
\end{align*}
We observe that, by Assumption \ref{Growth Conditions} we have that:
\begin{equation*}
   2\psi^{-1}\Big(\frac{x}{D^{\psi}_2}\Big) \leq \psi^{-1}(x) \leq 2\psi^{-1}\Big(\frac{x}{D^{\psi}_1}\Big) 
\end{equation*}
and by Assumption \ref{Eigenfunction Growth}:
\begin{equation*}
   \Big(D^s_2 s\Big(\frac{x}{2}\Big)\Big)^{-\frac{1}{1 + \epsilon}} \leq s(x)^{-\frac{1}{1 + \epsilon}} \leq  \Big(D^s_1 s\Big(\frac{x}{2}\Big)\Big)^{-\frac{1}{1 + \epsilon}}
\end{equation*}
Putting these together, we have:
\begin{equation*}
   2\psi^{-1}\Big(s(D^{\psi, s}_1 x)^{-\frac{1}{1 + \epsilon}}\Big) \leq \psi^{-1}(s(x)^{-\frac{1}{1 + \epsilon}}) \leq 2\psi^{-1}\Big(s(D^{\psi, s}_2 x)^{-\frac{1}{1 + \epsilon}}\Big) 
\end{equation*}
where $D^{\psi, s}_1 = 2^{(1 + \epsilon)\log_{D^s_1} D^{\psi}_2}$ and $D^{\psi, s}_2 = 2^{(1 + \epsilon)\log_{D^s_2} D^{\psi}_1}$. Hence, for $t > s^{-1}(\psi(\lambda)^{-1})$, we have:
\begin{equation*}
    \frac{\lambda}{\psi^{-1}(s(t)^{-\frac{1}{1 + \epsilon}})} = \frac{\psi^{-1}(s(s^{-1}(\psi(\lambda)^{-1 - \epsilon}))^{-\frac{1}{1 + \epsilon}})}{\psi^{-1}(s(t)^{-\frac{1}{1 + \epsilon}})} = \frac{\psi^{-1}(s(t^{-1}s^{-1}(\psi(\lambda)^{-1 - \epsilon})t)^{-\frac{1}{1 + \epsilon}})}{\psi^{-1}(s(t)^{-\frac{1}{1 + \epsilon}})} \geq 2^{-\log_{D^{\psi, s}_2} t^{-1} s^{-1}(\psi(\lambda)^{-1 - \epsilon})}
\end{equation*}
Observing that $2^{-\log_{D^{\psi, s}_2} t^{-1} s^{-1}(\psi(\lambda)^{-1 - \epsilon})} = \Big(\frac{t}{s^{-1}(\psi(\lambda)^{-1 - \epsilon})}\Big)^{r}$, where $r \equiv \log_{D^{\psi, s}_2}(2)$. Hence, we have:
\begin{align*}
     \mathcal{N}(\lambda) & \leq \int_{0}^{s^{-1}(\psi(\lambda)^{-1 -\epsilon})} \frac{dt}{1 + \frac{C_1\lambda}{\psi^{-1}(s(t)^{-\frac{1}{1 + \epsilon}})}} + \int_{s^{-1}(\psi(\lambda)^{-1 - \epsilon})}^{\infty} \frac{dt}{1 + \frac{C_1\lambda}{\psi^{-1}(s(t)^{-\frac{1}{1 + \epsilon}})}}   \\
     & \leq s^{-1}(\psi(\lambda)^{-1 - \epsilon}) + \int_{s^{-1}(\psi(\lambda)^{-1 - \epsilon})}^{\infty}  \frac{dt}{1 + C_1\Big(\frac{t}{s^{-1}(\psi(\lambda)^{-1 - \epsilon})}\Big)^{r}} \\
     & \leq s^{-1}(\psi(\lambda)^{-1 - \epsilon}) + s^{-1}(\psi(\lambda)^{-1 - \epsilon})\int_{1}^{\infty} \frac{ds}{1 + C_1 s^r} \\
     & \leq Cs^{-1}(\psi(\lambda)^{-1 - \epsilon})
\end{align*}
for some $C>0$, since the second integral converges as $r = \log_{D^{\psi, s}_2}(2) > 1$ (as we can always choose $D^s_2 > D^{\psi}_1$). 
\end{proof}

\section{Fourier Capacity and the Optimal Range Space}
\label{Fourier Capacity and Range Space}

In this section, we prove the analogue of Lemma \ref{Mercer Differencing} on nonperiodic domains $\mathcal{X} \subset \mathbb{R}^d$. Before we present the lemma, we first define the trace of a Dirichlet form. As in section \ref{Potential Preliminaries}, our discussion here is primarily conceptual, and we refer the reader to section 6.2 of \cite{fukushima2010dirichlet} for finer technical details . 

Let $(\langle \cdot, \cdot \rangle_{\mathcal{H}}, \mathcal{H})$ be a local, transient Dirichlet form on $L^2(\mathbb{R}^d)$ and $(\mathcal{Z}, \mu)$ a measure space, with $\mathcal{Z} \subset \mathbb{R}^d$ and $\mu$ a Radon measure supported on $\bar{\mathcal{Z}}$. Let $\tilde{\mathcal{H}}$ denote the extended Dirichlet space obtained by taking pointwise limits of Cauchy sequences in $\mathcal{H}$ (see Theorem 1.5.2 in \cite{fukushima2010dirichlet} for details; this is a Hilbert space by Theorem 1.5.3 therein). We orthogonally decompose $\tilde{\mathcal{H}}$ as:
\begin{equation}
\label{Decomposition of Extended Dirichlet Space}
    \tilde{\mathcal{H}} = \tilde{\mathcal{H}}_{\mathbb{R}^d \setminus \tilde{\mathcal{Z}}} \otimes \mathscr{H}_{\tilde{\mathcal{Z}}}
\end{equation}
where $\tilde{\mathcal{Z}} \subset \mathcal{Z}$ is a quasi-support of $\mu$ defined in (5.1.21) of \cite{fukushima2010dirichlet} in terms of the unique (up to equivalence) positive, continuous additive functional (PCAF) associated with $\mu$. We note that $\mu(\mathcal{Z} \setminus \tilde{\mathcal{Z}}) = 0$. In \eqref{Decomposition of Extended Dirichlet Space}, $\tilde{\mathcal{H}}_{\mathbb{R}^d \setminus \tilde{\mathcal{Z}}} = \{f \in \mathcal{H}: f = 0 \hspace{1mm} \text{on} \hspace{1mm} \tilde{\mathcal{Z}}\}$. We are now ready to define the trace of a Dirichlet form:

\begin{definition}
\label{Trace Definition}
Let $P_{\tilde{\mathcal{Z}}}$ denote the orthogonal projection from $\tilde{\mathcal{H}}$ to $\mathscr{H}_{\tilde{\mathcal{Z}}}$. The trace of $(\langle \cdot, \cdot \rangle_{\mathcal{H}}, \mathcal{H})$ on $L^2(\mathbb{R}^d)$ to $(\mathcal{Z}, \mu)$ is $(\langle \cdot, \cdot \rangle_{\mathcal{H}(\mathcal{Z})}, \mathcal{H}(\mathcal{Z}))$ where:
\begin{align*}
    \mathcal{H}(\mathcal{Z}) & = \{f \in L^2(\mathcal{Z}, \mu): f = g \hspace{2mm} \mu \hspace{1mm} \text{a.e. on} \hspace{1mm} \mathcal{Z} \hspace{1mm} \text{for} \hspace{1mm} g \in  \tilde{\mathcal{H}}\} \\
    \langle f, g \rangle_{\mathcal{H}(\mathcal{Z})} & = \langle P_{\tilde{\mathcal{Z}}}\tilde{f}, P_{\tilde{\mathcal{Z}}}\tilde{g} \rangle_{\mathcal{H}}
\end{align*}
where $\tilde{f}, \tilde{g} \in \tilde{\mathcal{H}}$ with $f = \tilde{f}$ and $g = \tilde{g}$ a.s. on $\mathcal{Z}$. 
\end{definition}
We remark that the inner product $\langle \cdot, \cdot \rangle_{\mathcal{H}(\mathcal{Z})}$ in Definition \ref{Trace Definition} is well-defined by Lemma 6.2.1 in \cite{fukushima2010dirichlet}. Moreover, when $\mathcal{H}$ is itself a Hilbert space in $||\cdot||_{\mathcal{H}}$ and the imbedding $\mathcal{H} \hookrightarrow L^2(\mathbb{R}^d)$ is injective (as is the case for $\mathcal{H}_K(\mathbb{R}^d)$, the global version of our RKHS), then the extended Dirichlet space $\tilde{\mathcal{H}}$ simply coincides with $\mathcal{H}$ (by Theorem 1.5.5 in \cite{fukushima2010dirichlet}). 

\begin{lemma}
\label{Optimal Target Space}
Assume $L^2(\mathcal{X}) \subset L^2(\mathcal{X}, \nu)$. Suppose $\psi(t) = t^{\beta}$ and $s(t) = t^{\alpha}$ for $\beta \in (0, 1)$ and $\alpha \geq 1$. Let $r = \frac{d}{2}\Big(\frac{\alpha(1 - \beta)}{\beta} + 1\Big)$. Then, Assumption \ref{Opt Smoothness} is equivalent to:
\begin{equation*}
    K_1||Pf||_{L^2(\nu)} \leq ||T_{\nu}f||_{H^r(\mathcal{X})} \leq K_2||Pf||_{L^2(\nu)}
\end{equation*}
for some constants $K_1, K_2 > 0$. Here $H^r(\mathcal{X})$ is the trace of the Sobolev space $H^r(\mathbb{R}^d)$ to $(\mathcal{X}, \nu)$ (as in Definition \ref{Trace Definition}) and $P: L^2(\nu) \to H^r(\mathcal{X})$ is the canonical projection. 
\end{lemma}

\begin{proof}
Let $\tilde{\mathcal{X}} \subset \mathcal{X}$ be the quasi-support of $\nu$ defined after \eqref{Decomposition of Extended Dirichlet Space}. Decompose $H^r(\mathbb{R}^d)$ as $H^r(\mathbb{R}^d) = H^r_{\mathbb{R}^d \setminus \tilde{\mathcal{X}}} \otimes \mathscr{H}_{\tilde{\mathcal{X}}}$ as in \eqref{Decomposition of Extended Dirichlet Space}. Let $P_{\tilde{\mathcal{X}}}$ be the projection onto $\mathscr{H}_{\tilde{\mathcal{X}}}$ in $H^r(\mathbb{R}^d)$ (note in this paragraph we use the same conventions as in Definition \ref{Trace Definition} simply replacing the generic $\mathcal{Z}$ with $\mathcal{X}$). Let $I_{m}: \mathcal{H}_K \to L^2(\mathbb{R}^d)$ denote the canonical imbedding. We first demonstrate that $\text{Im}(I_m S_{\nu}) \in H^r(\mathbb{R}^d)$ (where we recall $S_{\nu} = I_{\nu}^{*}$ is defined in \eqref{eq: Adjoint Formula}). Let $d\nu(x) = w(x)dx$. Then, we have for $f \in \mathcal{H}_K$:
\begin{align}
    ||I_m S_{\nu}f||^2_{H^r(\mathbb{R}^d)} & = \int_{\mathbb{R}^d} (1 + ||\xi||^2)^r \mathcal{F}[I_m S_{\nu} f](\xi)^2 d\xi \nonumber \\
    & = \int_{\mathbb{R}^d} (1 + ||\xi||^2)^r \mathcal{F}_d \kappa(||\xi||)^2 \mathcal{F}[fw](\xi)^2 d\xi \label{eq: Fourier Convolution} \\
    & \asymp \int_{\mathbb{R}^d} (1 + ||\xi||^2)^r \cdot \frac{s(||\xi||^d)\psi^{-1}(s(||\xi||^d)^{-1})}{||\xi||^d} \mathcal{F}_d\kappa(||\xi||) \mathcal{F}[fw](\xi)^2 d\xi \label{eq: Apply Assumption 8*} \\
    & = \int_{\mathbb{R}^d} (1 + ||\xi||^2)^{\frac{d}{2}\Big(\frac{\alpha(1 - \beta)}{\beta} + 1\Big)} \cdot ||\xi||^{d\Big(\frac{-\alpha(1 - \beta)}{\beta} - 1\Big)} \mathcal{F}_d\kappa(||\xi||) \mathcal{F}[fw](\xi)^2 d\xi \label{eq: Recall Formulas} \\
    & \asymp \int_{\mathbb{R}^d} \mathcal{F}_d\kappa(||\xi||) \mathcal{F}[fw](\xi)^2 d\xi \nonumber \\
    & = \int_{\mathbb{R}^d} \mathcal{F}[fw](\xi) \cdot \mathcal{F}_d\kappa(||\xi||)\mathcal{F}[fw](\xi)d\xi \nonumber \\
    & = \int_{\mathbb{R}^d} f(x) \Big(\int_{\mathbb{R}^d} K(x, y)f(y)w(y)dy\Big) w(x)dx \label{eq: Recall Adjoint Def} \\
    & = \int_{\mathcal{X}} f(x) \int_{\mathcal{X}} K(x, y)f(y)d\nu(y) d\nu(x) \nonumber \\
    & = \langle f, T_{\nu}f \rangle_{\nu} \nonumber 
\end{align} 
where \eqref{eq: Fourier Convolution} follows from the Fourier convolution formula, \eqref{eq: Apply Assumption 8*} follows from Assumption \ref{Opt Smoothness}, \eqref{eq: Recall Formulas} follows from our assumption on $s(t)$ and $\psi(t)$, and \eqref{eq: Recall Adjoint Def} follows from the definition \eqref{eq: Adjoint Formula} of $S_{\nu}$ and recalling $K(x, y) = \kappa(||x - y||)$. Hence, by the boundedness of $T_{\nu}$, we have that $||I_m S_{\nu}f||^2_{H^r(\mathbb{R}^d)} < \infty$. Now, observe that $I_m S_{\nu}f|_{\mathcal{X}} = T_{\nu} f$ a.s. (with respect to either $dx$ or $d\nu$, by the absolute continuity of the latter). Let $\tilde{g} \in H^s(\mathcal{X})$ --- then there exists a $g \in H^s(\mathbb{R}^d)$ such that $g = \tilde{g}$ a.s. on $\mathcal{X}$ (the choice of extension does not matter by Lemma 6.2.1 in \cite{fukushima2010dirichlet}). Hence, for $f \in L^2(\nu)$, we have:
\begin{align*}
    \langle \tilde{g}, T_{\nu}f\rangle_{H^r(\mathcal{X})} & = \langle P_{\tilde{\mathcal{X}}}g, I_m S_{\nu}f \rangle_{H^r(\mathbb{R}^d)} \\
    & = \int_{\mathbb{R}^d} (1 + ||\xi||^2)^r \mathcal{F}[P_{\tilde{\mathcal{X}}}g](\xi)\mathcal{F}[I_m S_{\nu}f](\xi) d\xi \\
    & = \int_{\mathbb{R}^d} (1 + ||\xi||^2)^r \mathcal{F}[P_{\tilde{\mathcal{X}}}g](\xi) \mathcal{F}_d\kappa(||\xi||) \mathcal{F}[fw](\xi) d\xi \\
    & \asymp \int_{\mathbb{R}^d} (1 + ||\xi||^2)^{\frac{d}{2}\Big(\frac{\alpha(1 - \beta)}{\beta} + 1\Big)} \cdot ||\xi||^{d\Big(\frac{-\alpha(1 - \beta)}{\beta} - 1\Big)} \mathcal{F}[P_{\tilde{\mathcal{X}}}g](\xi) \mathcal{F}[fw](\xi) d \xi \\
    & \asymp \int_{\mathbb{R}^d} \mathcal{F}[P_{\tilde{\mathcal{X}}}g](\xi) \mathcal{F}[fw](\xi) d \xi \\
    & = \int_{\mathbb{R}^d} P_{\tilde{\mathcal{X}}}g(x)f(x)w(x)dx \\
    & = \int_{\mathcal{X}} g(x)f(x)d\nu(x)
\end{align*}
where the last line follows upon recalling that $P_{\tilde{\mathcal{X}}}g = g$ a.s. on $\mathcal{X}$ and the previous lines follow analogously to above. Taking the supremum over $g \in \mathcal{B}(H^r(\mathcal{X}))$, we obtain our result.  
\end{proof}

\begin{remark}
Note that in Lemma \ref{Optimal Target Space}, we may have equivalently defined $H^r(\mathcal{X})$ as the trace of $H^r(\mathbb{R}^d)$ to $(\mathcal{X}, dx)$ due to the absolute continuity of $\nu$ and its support being $\bar{\mathcal{X}}$. 
\end{remark}


\end{document}